\documentclass[11pt]{article}

\usepackage[margin=1in]{geometry}
\usepackage{framed} 
\usepackage{mathtools}
\usepackage[]{amsmath,amssymb,epsfig}
\usepackage{amsthm}
\usepackage{amsmath}
\usepackage{amssymb}
\usepackage{graphicx}
\usepackage{epstopdf}
\usepackage{comment}
\usepackage{array}
\usepackage{algorithm}
\usepackage{url}
\usepackage{ifthen}
\usepackage{wrapfig}
\usepackage{lscape}
\usepackage{algpseudocode}
\usepackage{setspace}
\usepackage{multicol}
\usepackage{multirow}
\usepackage{color}
\usepackage{colortbl}
\usepackage{xcolor}
\usepackage{rotating}
\usepackage{caption}
\usepackage{float}
\usepackage{ifthen}
\usepackage{placeins}
\usepackage{framed}
\usepackage{bbm}
\usepackage{enumitem}  
\usepackage{arydshln}
\usepackage{mathdots}
\usepackage[%
    font={small,sf},
    labelfont=bf,
    format=hang,    
    format=plain,
    margin=0pt,
    width=0.8\textwidth,
]{caption}
\usepackage{mathrsfs}
\usepackage{algpseudocode}
\usepackage[list=true]{subcaption}
\usepackage{bbold}

\newtheorem{theorem}{Theorem}

\newtheorem{lemma}{Lemma}

\newtheorem{proposition}{Proposition}
\newtheorem{remark}{Remark}

\numberwithin{equation}{section}
\newcommand{\rev}[1]{\textcolor{black}{#1}}

\newcommand{\calG}{\ensuremath{\mathcal{G}}}

\newcommand{\calP}{\ensuremath{\mathcal{P}}}

\newcommand{\calS}{\ensuremath{\mathcal{S}}}

\newcommand{\calN}{\ensuremath{\mathcal{N}}}

\newcommand{\calK}{\ensuremath{\mathcal{K}}}

\newcommand{\calE}{\ensuremath{\mathcal{E}}}
\newcommand{\calL}{\ensuremath{\mathcal{L}}}
\newcommand{\calO}{\ensuremath{\mathcal{O}}}

\DeclareMathOperator{\rank}{rank}

\newcommand{\norm}[1]{\Big\|{#1} \Big\|}
\newcommand{\normnew}[1]{\|{#1}\|}

\newcommand{\set}[1]{\left\{{#1}\right\}}
\newcommand{\dotprod}[2]{\langle#1,#2\rangle}
\newcommand{\est}[1]{\widehat{#1}}
\newcommand{\expec}{\ensuremath{\mathbb{E}}}
\newcommand{\matR}{\ensuremath{\mathbb{R}}}

\newcommand{\matN}{\ensuremath{\mathbb{N}}}
\newcommand{\argmax}[1]{\underset{#1}{\operatorname{argmax}}}
\newcommand{\argmin}[1]{\underset{#1}{\operatorname{argmin}}}

\newcommand{\prob}{\ensuremath{\mathbb{P}}}

\newcommand{\indic}{\ensuremath{\mathbbm{1}}} 

\newcounter{ale}

\newenvironment{liste}{\begin{itemize}}{\end{itemize}}
\newcommand{\aliste}{\begin{liste} \setcounter{ale}{1}}
\newcommand{\zliste}{\end{liste}}

\usepackage[colorlinks=true,linkcolor=red,filecolor=green,citecolor=red]{hyperref}

\definecolor{purple}{RGB}{178,55,250} % 

\newcounter{noteHTctr} 

\definecolor{applegreen}{rgb}{0.55, 0.71, 0.0}
\newcounter{noteEActr} 

\newcommand{\tx}{\tilde{x}}
\newcommand{\hX}{X^\#}
\newcommand{\hx}{x^\#}
\newcommand{\txi}{\tilde{\xi}}
\newcommand{\cvar}{\textbf{CVAR}($1,d,T;A^\ast,\pi^\ast,\sigma$)}
\newcommand{\cvarA}{\textbf{CVAR}($1,d,T;A,\pi,\sigma$) }
\newcommand{\MLEpi}{\est{\Pi}_{\operatorname{MLE}}}

\newcommand{\ty}{\tilde{y}}

\usepackage{ifthen}

\newboolean{withSmall}
\setboolean{withSmall}{false}
\setboolean{withSmall}{true} 
\ifthenelse{\boolean{withSmall}}{   }{}

\newcommand\withMaybeSmall{\ifthenelse{\boolean{withSmall}}{ \small }{}}
\newboolean{arXivMode}
\setboolean{arXivMode}{true}
\usepackage{cleveref}
\usepackage{lineno}
\begin{document}

\title{Matching correlated VAR time series\footnote{Authors are listed alphabetically}}

\author{Ernesto Araya{$^1$}\\ \texttt{araya@math.lmu.de} \and Hemant Tyagi\footnote{HT was supported by a Nanyang Associate Professorship (NAP) grant from NTU Singapore} {$^2$} \\ \texttt{hemant.tyagi@ntu.edu.sg}}
% \author{Ernesto Araya\footnotemark[1]\\ \texttt{araya@math.lmu.de}\and Hemant Tyagi \footnotemark[4] \footnotemark[5]\\ \texttt{hemant.tyagi@inria.fr}}

% \renewcommand{\thefootnote}{\fnsymbol{footnote}}
% \footnotetext[1]{Department of Mathematics, Ludwig-Maximilians-Universit\"at M\"unchen}
% \footnotetext[2]{Department of Statistics \& Mathematical Institute, Oxford-Man Institute of Quantitative Finance, University of Oxford, Oxford, UK}
% \footnotetext[3]{The Alan Turing Institute, London, UK}
% \footnotetext[4]{Inria, Univ. Lille, CNRS, UMR 8524 - Laboratoire Paul Painlev\'{e}, F-59000}
% \footnotetext[5]{Corresponding author \\ \indent \indent Authors are listed in alphabetical order. Part of this work was done while E.A was affiliated to Inria Lille.}

\date{$^1$Department of Mathematics, Ludwig-Maximilians-Universit\"at M\"unchen\\
$^2$Division of Mathematical Sciences, 
      SPMS, 
      NTU Singapore 637371}

\renewcommand{\thefootnote}{\arabic{footnote}}

\maketitle

\begin{abstract}
%\HT{Add abstract}
We study the problem of matching correlated VAR time series databases, where a multivariate time series is observed along with a perturbed and permuted version, and the goal is to recover the unknown matching between them. To model this, we introduce a probabilistic framework in which two time series $(x_t)_{t\in[T]},(\hx_t)_{t\in[T]}$ are jointly generated, such that $\hx_t=x_{\pi^*(t)}+\sigma \tilde{x}_{\pi^*(t)}$, where $(x_t)_{t\in[T]},(\tilde{x}_t)_{t\in[T]}$ are independent and identically distributed vector autoregressive (VAR) time series of order $1$ with Gaussian increments, for a hidden $\pi^*$. The objective is to recover $\pi^*$, from the observation of $(x_t)_{t\in[T]},(\hx_t)_{t\in[T]}$. This generalizes the classical problem of matching independent point clouds to the time series setting.% with envisaged applications in privacy and sensor fusion.

We derive the maximum likelihood estimator (MLE), leading to a quadratic optimization over permutations, and theoretically analyze an estimator based on linear assignment. For the latter approach, we establish recovery guarantees, identifying thresholds for $\sigma$ that allow for perfect or partial recovery. Additionally, we propose solving the MLE by considering convex relaxations of the set of permutation matrices (e.g., over the Birkhoff polytope). This allows for efficient estimation of $\pi^*$ and the VAR parameters via  alternating minimization. 
Empirically, we find that linear assignment  often matches or outperforms MLE relaxation based approaches. 
%These findings highlight the theoretical and practical effectiveness of efficient algorithms for structured time series alignment.
\end{abstract}


{
  \tableofcontents
}
%
%
%
%---------------
% Introduction
%---------------
%
\section{Introduction}\label{sec:introduction}
We consider the problem of matching two point clouds in $\matR^d$. 
Let $X,X^{\#} \in\matR^{d\times T}$ denote the matrices corresponding to the two point clouds, each containing $T$ points in $\matR^d$. We say that the point clouds are correlated if there exists a permutation\footnote{Extensions for which this holds only for a subset of $[m]$, considering sub-permutations, can be formulated analogously.} map $\pi^*:[T]\rightarrow [T]$ such that every column $x^{\#}_{\pi^*(t)}$ of $X^{\#}$ is correlated with the column $x_t$ of $X$. Given $X$ and $X^\#$, the goal then is to recover the unknown permutation $\pi^*$. This problem has a rich history with applications in computational geometry and computer vision, multi-object tracking etc. 

On the theoretical front, this has received attention recently when the columns of $X,X^\#$ are assumed to be random i.i.d vectors \cite{KuniskyWeed,schwengber2024geometricplantedmatchingsgaussian}. Specifically, \cite{KuniskyWeed} considered the setup where we first (a) draw $x_1,\dots,x_T \sim \calN(0,I_d)$  independently, then (b) draw the noise vectors $\tilde{x}_1,\dots,\tilde{x}_T \sim \calN(0, I_d)$  independently, and finally (c) obtain $x_t^\#$  as $$x_t^\# = x_{\pi^*(t)} + \sigma \tilde{x}_{\pi^*(t)} \ \text{ for } t=1,\dots,T.$$ 
It was shown that the maximum likelihood estimator (MLE), which is essentially a linear assignment problem, provably recovers $\pi^*$ provided the noise level $\sigma$ is less than a threshold. This was shown for different recovery criteria such as exact recovery and partial recovery (with sublinear number of errors). The setting in \cite{schwengber2024geometricplantedmatchingsgaussian} extended these results to more general distributions, along with information-theoretic lower bounds; see Section \ref{subsec:rel_work} for a more detailed overview of related work.

%
%\EA{Those papers considered different setting, but they share the same spirit.... on the other hand, in \cite{KuniskyWeed} they do not cite \cite{CollierDalalyan} ...}. We focus our attention on the case where the columns within $Y$ and $Y'$ represent correlated time series. 
%
% \todo{
% \begin{itemize}
%     \item Do a proper literature review. Some of the things we should look, are the matching point clouds, matching databases, matching features, maybe in the OT literature there are related examples.  
%     \item Finds one or two applications to motivate matching time series. 
% \end{itemize}
% }
% %

%\EA{Some potential motivation settings below}. 
\paragraph{Motivation for matching time series data.}
The assumption that the points within a point cloud are drawn independently was motivated in \cite{KuniskyWeed} by a stylized model for multi-target tracking involving $T$ independent standard Brownian motions. Here, $x_1,\dots,x_T$ correspond to the position of the (unlabeled) particles at a given time instant, and $x_1^\#,\dots,x_T^\#$ corresponds to their positions at the next time instant. In this paper, we consider a generalization of this setting to one where $(x_1,\dots,x_T)$ and $(\tilde{x}_1,\dots,\tilde{x}_T)$ are the realization of a stochastic process, hence the individual $x_t$'s and $\tilde{x}_t$'s will be respectively dependent. The motivation for studying this \emph{time-series} setting comes from the following applications.
\begin{itemize}
    \item \textbf{Time stamp shuffling as a privacy mechanism.} One natural way to obfuscate temporal data is to release values while discarding or shuffling their time stamps. This was recently considered in the context of differential privacy  for sensitive time series data such as health care records, financial transactions etc. \cite{temporal_privacy21, temporal_privacy23}. The intuition is that the resulting data remains useful for certain aggregate statistics (e.g., mean) while concealing the temporal structure. Consider wage data: observing a person’s monthly income over several years without the ordering allows one to compute their average salary, but conceals whether their earnings followed a steady upward trend or fluctuated seasonally. From the shuffled sequence alone, these scenarios may look indistinguishable. This raises the question of whether time-stamp shuffling offers meaningful privacy protection: if an adversary has access to an auxiliary, correlated time series, they may be able to partially reconstruct the original ordering and recover sensitive temporal information.
    \item \textbf{Sensor fusion and lost timestamps.} In distributed sensing networks (seismology, wireless sensors, Internet of Things), different sensors may not be synchronized. One sensor provides a clean signal with timestamps, while another provides a noisy but related signal without ordering (e.g. due to packet loss, buffering, or asynchronous sampling). For e.g., in the Internet of Things context, one may have many cheap sensors scattered around, each with limited processing power and no global clock synchronization. The data streams can arrive out of order, delayed, or with missing timestamps, as reported in \cite{temporalOrder_networks11} when analyzing data from a real time system \cite{sensorNet_app_09}. A natural goal then would be to align unordered sensor readings with the reference signal by exploiting temporal correlations, thus improving reconstruction accuracy. 
    \item \textbf{Time series alignment.} A fundamental problem in time series analysis is to align different, potentially misaligned sequences that reflect the same underlying phenomenon. Misalignment may be caused by temporal stretching, delays, or nonlinear warping. In the classical formulation, each series preserves its internal temporal structure, and the task is to find an appropriate monotone correspondence between time indices. In contrast, a shuffled time-stamp setting can model more severe distortions such as measurement defects, packet loss, or corrupted logging systems, where the ordering of observations is partially lost. This makes the alignment problem more challenging and closer in spirit to matching under unknown permutations. Applications of time series alignment are widespread, e.g., in neuroscience \cite{Neural_Time_series_warping}, % aligning neural recordings across trials or subjects is essential for identifying reproducible activity patterns; 
    speech and gesture recognition \cite{DTW_speech}, %dynamic time warping is a standard technique to handle differences in speaking rate or motion speed; 
    and bioinformatics \cite{DTW_gene_expression} to name a few. %alignment methods are used to compare gene expression or protein folding trajectories; and 
    %in climate science %, aligning sensor measurements with drifting clocks enables joint analysis of environmental data collected across heterogeneous devices.
    %
\end{itemize}
%
%
%
%
%\EA{Regarding that, I think in the context of privacy motivations of Graph Matching, one can assume that the labels of one of the graph are known. As far as I remember, this is presented in this way in some papers. I would need to find the precise references. }

%
%
%
\subsection{Correlated VAR model for planted matching} \label{subsec:corr_var_model}
%
%We begin by describing a standard autoregressive (AR) model for time series. 
Given a matrix \( A^\ast \in \mathbb{R}^{d \times d} \), where \( d \in \mathbb{N} \), suppose $(x_t)_{t \in [T]}$ is generated as follows.
\begin{align}
    x_{t+1} &= A^\ast x_t + \xi_{t+1}, \quad \text{for } t = 1, \dots, T-1, %\label{eq:def_X} 
    \\  
    x_1 &= \xi_1,  %\label{eq:def_X_2}
\end{align}  
where $ (\xi_t)_{t \in [T]} $ is a sequence of i.i.d. standard Gaussian vectors in $\mathbb{R}^d $. The above model is a \emph{Vector Autoregressive} model of order $1$, hereby referred to as \textbf{VAR}($1$, $d$, $T$). The matrix $A^*$ contains the coefficients that determine this temporal relationship and is referred to as the \emph{system matrix}.

\paragraph{CVAR: A model for correlated time series.}
Let \( (\tilde{x}_t)_{t \in [T]} \) be an independent copy of the ``base'' time series \( (x_t)_{t \in [T]} \), drawn from \textbf{VAR}($1$, $d$, $T$), but with an independent sequence of i.i.d. standard Gaussian vectors \( (\tilde{\xi}_t)_{t \in [T]} \). Given a noise parameter \( \sigma \geq 0 \), we first construct the perturbed time series 
\begin{equation} \label{eq:def_Xprime}
    x'_t = x_t + \sigma \tilde{x}_t, \quad \text{for all } t \in [T].
\end{equation}  
Here, \( (x'_t)_{t \in [T]} \) is a noisy version of \( (x_t)_{t \in [T]} \), with \( \sigma \) controlling the noise magnitude. Finally, given a permutation \( \pi^\ast \), we define 
\begin{equation*}
    x^\#_t = x'_{\pi^\ast(t)}, \quad \text{for all } t \in [T].
\end{equation*}  
The above model is referred to as the correlated VAR model with parameters $A^*, \pi^*$ and $\sigma$, or \textbf{CVAR}($1$, \( d \), \( T \); \( A^\ast \), \( \pi^\ast \),~\( \sigma \)) in short. The pair \( \big( (x_t)_{t \in [T]}, (x^\#_t)_{t \in [T]} \big) \) is then a realization from this model.
Our goal is to recover the planted matching $\pi^*$ given the observations 
$((x_t)_{t \in [T]}, (x^\#_t)_{t \in [T]}),$ where the matrix $A^\ast$ is unknown. 
Before delving into strategies for estimating $\pi^*$, the following remarks are worth noting.
\begin{enumerate}
    \item The above setup is a generalization of the point clouds matching problem, presented in \cite{KuniskyWeed}. Indeed, if $A^\ast=0$, we have, for all $t\in[T]$, 
    \begin{align*}
    x_t = \xi_t \ \text{ and } \ \tx_t = \txi_t.
    \end{align*}
    In this case,  $x'_t=\xi_t+\sigma\txi_t$, which is identical to the setting in \cite{KuniskyWeed}.

    \item In the noiseless case ($\sigma=0$) the permutation can be perfectly recovered with a simple algorithm. It suffices to match $x_t$ with $x^\#_{t'}$ such that $x_t=x^\#_{t'}$. Given that, for any $s,t\in [T]$ with $t\neq s$, $\prob(x_t=x_s)=0$, the algorithm returns the correct permutation almost surely. 

    \item Notice that we assumed that the base time series is presented with its temporal ordering known. Hence, one may interpret the matching problem as that of recovering the temporal ordering of the unordered series $(x^\#_t)_{t \in [T]}$ using information available from $(x_t)_{t \in [T]}$. We touch upon this point briefly below, see Remarks \ref{rem:unord_base_time_series} and \ref{rem:LA_agnostic_temporal} for a more detailed explanation.
\end{enumerate}
%
%the next remark could go somewhere, to justify the assumption of one known temporal ordering
%\EA{ I added the following remark regarding the fact that one time-series ordering is known. It can be removed if the point is already clear from the previous text.
\begin{remark}
Interpreting the time-stamps as labels for each data point in the time series, the assumption that the correct temporal order is known for the base time series corresponds to a standard assumption in data privacy applications. Specifically, in data de-anonymization settings, one database is assumed to be public with known labels, while the other must remain private. The goal in this context is to recover the private labels using the public database as a reference (see, e.g., \cite{Narayanan2008RobustDO} for a seminal work in this area).
\end{remark}

%-----------------
% Contributions
%-----------------
\subsection{Contributions}
Our contributions are outlined below.
\begin{enumerate}
    \item We propose, to our knowledge, a novel statistical model for planted matching in the context of time series data. For this model, we derive the MLE estimator for estimating $A^*, \pi^*$ which amounts to minimizing a biconvex objective subject to nonconvex constraints (due to the set of permutations). We formulate several methods to solve it based on the alternating minimization framework, by considering different convex relaxations of the set of permutation matrices. The empirical performance of these relaxations are compared through several experiments on synthetic data. Interestingly, the linear assignment (LA) estimator performs comparably to the best MLE relaxations even under high noise levels, when $\|A^*\|_2\leq 1$, raising the question of whether it achieves optimal, or near-optimal, performance in this regime. Experimentally, we found that for some values $\|A^*\|_2 > 1$, MLE relaxations slightly outperform the LA estimator.

    \item On the theoretical front, we analyze the statistical performance of the LA estimator for recovering $\pi^*$, which is well studied for geometric matching problems in the i.i.d setting \cite{KuniskyWeed,schwengber2024geometricplantedmatchingsgaussian}. This  estimator is model agnostic and also does not need the base time series to be temporally ordered (see Remark \ref{rem:LA_agnostic_temporal}). Assuming $\|A^*\|_2 < 1$, we derive bounds on the number of mismatched points by following the technique of counting augmenting paths, considered recently in \cite{KuniskyWeed} (and later in \cite{schwengber2024geometricplantedmatchingsgaussian}); see Theorem \ref{thm:main_upper_bound}. Theorem \ref{thm:main_upper_bound} shows various  thresholds on the noise level $\sigma$ which imply different levels of recovery of $\pi^*$ (e.g., exact recovery, partial recovery). It's statement is analogous to that obtained in \cite{KuniskyWeed}, up to an additional factor proportional to $(1-\|A^*\|_2)^5$ appearing in the bounds; see Remark \ref{rem:spec_grap_factor}. The proof, while along the lines of that in \cite{KuniskyWeed} is considerably more challenging -- not simply in the sense of more tedious calculations, but also in terms of technical difficulties imposed by the \textbf{CVAR} model; see Remark \ref{rem:comp_kunisky} for details.
\end{enumerate}

\subsection{Related work} \label{subsec:rel_work}
%\HT{Add related work}
%
\paragraph{Matching point clouds. }As discussed earlier, our statistical model generalizes the i.i.d setting (within a point cloud) in \cite{KuniskyWeed} to the dependent setup, where each point cloud is a VAR time series. The work \cite{schwengber2024geometricplantedmatchingsgaussian} extended the setup in \cite{KuniskyWeed} to handle more general classes of distributions, and also provides information-theoretic lower bounds on the expected error for any estimator. The latter was achieved by making a connection with matchings in random geometric graphs. Some of the results in \cite{KuniskyWeed} were shown to be information theoretically optimal in \cite{pmlr-v178-wang22a}. While we do not study lower bounds for the \textbf{CVAR} model, it is an interesting and non-trivial direction to pursue for future work.

\paragraph{Feature matching.}A closely related problem referred to as the feature matching problem was studied in \cite{collier_2016}. Here the goal is to match two sets of points in $\matR^d$ (of potentially unequal size), and the proposed statistical model  assumes all the points to be independently generated Gaussian's. The means of the Gaussian distributions for one point cloud are denoted $(\theta^*_i)_i$ while that of the other point cloud are $(\theta^*_{\pi^*(i)})_i$, with $\pi^*$ the latent permutation. The performance of different estimators, including the LA estimator (referred therein as the least sum of squares estimator) was studied theoretically with aim of establishing the minimax rate of separation between the $\theta_i$'s for consistent recovery of $\pi^*$. This result was extended in \cite{outlierfeaturemap_2022} to the setting where the second point set can contain outliers.

\paragraph{Covariance alignment.}In \cite{Rig_Covariance_matching}, the authors, motivated by the feature matching problem, study the task of covariance matrix alignment. Specifically, two independent samples of i.i.d points are observed: one distributed as $\mathcal{N}(0, \Sigma)$, and the other as $\mathcal{N}(0, \Pi^* \Sigma {\Pi^*}^\top)$, where $\Pi^*$ is a hidden permutation matrix. In this setting, $\Sigma$ can be viewed as a nuisance parameter, and the goal is to align the sample covariance matrices of the two datasets. To recover $\Pi^*$, the authors derive a quasi–maximum likelihood estimator, which reduces to solving a quadratic optimization problem over the permutation set—an instance of the Quadratic Assignment Problem (QAP) (known to be NP-hard). They propose a relaxation over the Birkhoff polytope, referred to as the Gromov–Wasserstein estimator, due to its connection with optimal transport, and show that this estimator is minimax optimal. 

\paragraph{Graph matching. } In the graph matching problem, the goal is to find an assignment between the vertices of two graphs such that their edges are maximally aligned. This problem has found many applications in computer vision \cite{fifty_years_GM}, data de-anonymization \cite{Narayanan2008RobustDO} and protein-protein interactions \cite{Zaslavskiy2009}, to name a few. In the statistical version of the problem, the pair of graphs are realizations of a probabilistic model for correlated random graphs. The most popular models are the \emph{correlated Erd\H{o}s-R\'enyi} model \cite{Pedarsani} and the \emph{correlated Gaussian Wigner} model \cite{DingMaWuXu2021}. For both models, the MLE corresponds solving a QAP problem. Many algorithms rely on convex relaxations \cite{Aflalo_convexGM,relax_own_risk,fan2022spectralI,ArayaTyagi_GM_fods}, and some of our proposed relaxations for MLE in our setting draw inspiration from those of graph matching. A related line of work exists for recovering planted matchings in weighted bipartite graphs, without a latent geometric structure \cite{Ding2021ThePM, PhysRevE.102.022304, plantmatch_moharrami21}.

\paragraph{Learning dynamical systems from non-sequenced data.} The problem of learning dynamical systems from \emph{non-sequenced} data has received significantly less attention than the setting of sequenced observations. This was first proposed in \cite{HuangNonSeqLinSys09} for linear VAR models where it was assumed that multiple i.i.d realizations of the model are first generated and then a single state is sampled at random from each trajectory. An Expectation-Maximization (EM) algorithm was proposed and tested on synthetic data. This was extended to non-linear VAR models in \cite{huang2010NonlinNonseq} 
where the single-trajectory setting was also considered. For this setup, \cite{huang2010NonlinNonseq} proposed a convex program over the Birkhoff polytope for estimating the latent ordering of the points. The work \cite{huang2011AutoNonseq} considered linear VAR models where a small amount of non-sequenced data drawn from the stationary distribution of the model is also available. A penalized least-squares method was proposed where the penalization is based on the Lyapunov equation concerning the covariance matrix of the stationary distribution. Finally, \cite{huang2013NonSeqHMM} considered learning (first order or Hidden) Markov models from non-sequenced data, and proposed methods based on tensor decomposition along with theoretical guarantees.

%------------
% Notation
%------------
\subsection{Notation}
We reserve lowercase letters for vectors and uppercase letters for matrices. 
For $x \in \mathbb{R}^d$, we write $\|x\|_2$ for the standard Euclidean ($\ell_2$) norm. Similarly, for a matrix $X \in \mathbb{R}^{k \times k'}$, $\|X\|_F$ denotes its Frobenius norm while $\normnew{X}_2$ denotes its spectral norm. Given matrices $M\in\matR^{k\times k'}$ and $N\in \matR^{l\times l'}$, the standard Kronecker product between $M$ and $N$ is denoted by $M \otimes N\in \matR^{kl\times k'l'}$.  
For a symmetric matrix $M \in \mathbb{R}^{k \times k}$, its eigenvalues are denoted by $\lambda_1(M) \geq \lambda_2(M) \geq \dots \geq \lambda_k(M)$.     

Given $T \in \mathbb{N}$, we let $\mathcal{S}_T$ denote the set of permutation maps on $\{1,\dots,T\}$ and $\mathcal{P}_T$ the set of corresponding permutation matrices. Sets will be usually denoted by calligraphic letters.

The notation $x \sim \mathcal{N}(\mu, \Sigma)$ specifies that $x$ is a Gaussian random vector with mean $\mu$ and covariance matrix $\Sigma$.  

For negative functions $f, g$, we will often write $f(x) = O(g(x))$ if there exists $c > 0$ and $x_0$ such that $f(x) \leq  c g(x)$ for all $x \geq x_0$. Moreover, we write $f(x) = \Omega(g(x))$ if $g(x) = O(f(x))$ holds, and write $f(x) = \Theta(g(x))$ if both $f(x) = O(g(x))$ and $g(x) = O(f(x))$ hold. Finally, we write $f(x) = o(g(x))$ if $\lim_{x \rightarrow \infty} \frac{f(x)}{g(x)} = 0$ and $f(x) = \omega(g(x))$ if $\lim_{x \rightarrow \infty} \frac{f(x)}{g(x)} = \infty$.

%
%
%---------------------------------------
% MLE and the linear assignment estimator 
%---------------------------------------
%
\section{MLE and the linear assignment estimator}\label{sec:mle_la_formulate}
%
%------------------------
% MLE derivation
%
\subsection{MLE for recovering \texorpdfstring{$\pi^\ast$}{pi*} and \texorpdfstring{$A^\ast$}{A*} when \texorpdfstring{$\sigma$}{sigma} is known}\label{sec:MLE_pi}

% \subsection{MLE for recovering \texorpdfstring{$\pi^*$}{pi*} and $A^*$ when $\sigma$ is known}\label{sec:MLE_pi}
%
    We first derive the MLE for $\pi^*$ and $A^*$, given the observations $ \big( (x_t)_{t \in [T]}, (x^\#_t)_{t \in [T]} \big) $ generated from \textbf{CVAR}($1$, \( d \), \( T \); \( A^\ast \), \( \pi^\ast \),~\( \sigma \)). Notice that although our main goal is to recover $\pi^*$, in principle, the parameters $A^*$ and $\sigma$ are also unobserved. For convenience, we will assume that $\sigma$ is known. As we will see shortly, this assumption only enters the picture for obtaining the MLE of $A^*$. If $A^*$ is known, then the MLE of $\pi^*$ does not require $\sigma$ to be known. In the work \cite{KuniskyWeed} where $A^* = 0$, the MLE for $\pi^*$ (which is a linear assignment problem) was obtained for  $\sigma$ unknown.  
    %\HT{lets focus on $\sigma$ known/fixed for now, we can derive the MLE when its unknown, separately, as it will be more relevant for the experiments part for this paper. I corrected typos etc in Lemma \ref{lem:MLE_sigma_known} and its proof}
%
\begin{lemma}[MLE for CVAR given $\sigma$]\label{lem:MLE_sigma_known}
    %
    %Let $\big( (x_t)_{t \in [T]}, (x^\#_t)_{t \in [T]} \big)$ be a correlated pair of time series generated from the {\cvar} model. 
    Given $\sigma$, the MLE for $(\pi^*,A^*)$ is found by solving 
    \begin{equation}\label{eq:MLE_sigma_known}
      \min_{\pi\in \calS_T, A\in\matR^{d\times d}  }\left\{ \sum^T_{t=1} \left(\|x_t-Ax_{t-1}\|^2_2+\frac1{\sigma^2}\|(x^\#_{\pi^{-1}(t)}-Ax^\#_{\pi^{-1}(t-1)})-(x_t-Ax_{t-1})\|^2_2 \right)\right\},
    \end{equation}
    where we set $\pi^{-1}(0), x^\#_0, x_0 \equiv 0$ for notational convenience.
\end{lemma}
\begin{proof}
   Let $(\xi_t)_{t\in[T]}, (\txi_t)_{t\in[T]}$ be two sequences of iid standard Gaussian vectors in dimension $d$. If $\big( (x_t)_{t \in [T]}, (x^\#_t)_{t \in [T]} \big)$ were generated according to the \cvarA model, we would have 
   \begin{align*}
     x_t = Ax_{t-1}+\xi_t \ \text{ and } \ 
     \tx_t = A\tx_{t-1}+\txi_t \quad \text{ for } t=2,\ldots,T,
   \end{align*}
   with $x_1=\xi_1$ and $\tx_1=\txi_1$.
On the other hand, $x'_t= x_t+\sigma \tx_t$ and $x^\#_t=x'_{\pi(t)}$. Define the negative log-likelihood function, including $\sigma$ as unobserved, by 
\begin{equation*}
   \calL(A,\pi,\sigma):= -\log{f_{A,\pi,\sigma}(x_1,x^\#_1,x_2,x^\#_2,\ldots,x_T,x^\#_T)},
\end{equation*}
where $f_{A,\pi,\sigma}$ denotes the joint density of $x_1,x^\#_1,x_2,x^\#_2,\ldots,x_T,x^\#_T$, under the \textbf{CVAR}($1,d,T;A,\pi,\sigma$) model. We use a similar notation to denote the density of marginals, e.g., $f_{A,\pi,\sigma}(x_t)$ denotes the (marginal) density of $x_t$.
First, note that 
\begin{align*}
    x'_t =x_t+\sigma \tx_t
    &=Ax_{t-1}+\xi_t+\sigma A\tx_{t-1}+\sigma \txi_t \\
    &=Ax'_{t-1}+\xi_t+\sigma \txi_t \\
    &=Ax'_{t-1}+x_t-Ax_{t-1}+\sigma \txi_t.
\end{align*}
Denoting\footnote{\rev{We keep the notation $f_{A,\pi,\sigma}$ for the density of $x_1,x'_1,\ldots,x_T,x'_T$.}} $f_{A,\pi,\sigma}(x_t,x'_t|x_{t-1},x'_{t-1})$ the density of $(x_t,x'_t)$ given $x_{t-1},x'_{t-1}$, we have for $t=2,\ldots,T$, 
\begin{align*}
    f_{A,\pi,\sigma}(x_t,x'_t|x_{t-1},x'_{t-1})&= \underbrace{f_{A,\pi,\sigma}(x_t|x_{t-1},x'_{t-1})}_{\sim \calN(Ax_{t-1},I_d)}\underbrace{f_{A,\pi,\sigma}(x'_t|x_t,x_{t-1},x'_{t-1})}_{\sim \calN(Ax'_{t-1}+x_t-Ax_{t-1},\sigma^2I_d)}\\
    &=  \frac{C}{\sigma^d}\exp\left(-\frac{\|x_t-Ax_{t-1}\|_2^2}{2}\right)\exp\left(-\frac{\|x'_t-Ax'_{t-1}-(x_t-Ax_{t-1})\|_2^2}{2\sigma^2}\right),
\end{align*}
where $C>0$ is a constant. The above is also valid for $t=1$, with the notation $x_0=x'_0=0$. With that in mind, we have $\rev{f_{A,\pi,\sigma}(x_1,x'_1,\dots,x_T,x'_T)} = \prod^T_{t=1}f_{A,\pi,\sigma}(x_t,x'_t|x_{t-1},x'_{t-1})$ which implies
\begin{align*}
-\log f_{A,\pi,\sigma}(x_1,x'_1,\ldots, x_T,x'_T)=& dT\log \sigma + \frac{1}{2} \sum^T_{t=1} \|x_t-Ax_{t-1}\|^2_2  + TC\\
&+\frac1{2\sigma^2}\sum^T_{t=1}\|(x'_t-Ax'_{t-1})-(x_t-Ax_{t-1})\|^2_2.
\end{align*}
Since $x^\#_t=x'_{\pi(t)}$, for $t=1,\ldots, T$, we obtain 
\begin{align*}
%-\log f(X_1,X^\#_1,\ldots, X_T,X^\#_T)
\calL(A,\pi,\sigma)=&dT\log \sigma + \frac{1}{2}\sum^T_{t=1}\|x_t-Ax_{t-1}\|^2_2 + TC \\
&+\frac1{2\sigma^2}\sum^T_{t=1}\|(x^\#_{\pi^{-1}(t)}-Ax^\#_{\pi^{-1}(t-1)})-(x_t-Ax_{t-1})\|^2_2.
\end{align*}
%
%Notice that $\calL(A,\pi,\sigma)$ is decreasing in $\sigma$, for $\sigma \in(0,\infty)$ then $\calL(A,\pi,\sigma)$ is minimized for $\sigma=\infty$, which is not an interesting case. 
For a given $\sigma$, we arrive at the formulation in \eqref{eq:MLE_sigma_known} for finding the optimal $A, \pi$.
\end{proof}
\begin{remark}[MLE when $\sigma=0$]
In the case $\sigma=0$, the negative log-likelihood simplifies to  
\begin{equation*}
    \mathcal{L}(A,\pi,0) = \sum_{t=1}^{T} \|x_t - A x_{t-1}\|_2^2 - \sum_{t=1}^{T} \log \indic_{\{x_t = x^\#_{\pi^{-1}(t)}\}}.
\end{equation*}
As a result, the optimization problem  
\begin{equation*}
    \min_{\pi, A} \mathcal{L}(A, \pi, 0)
\end{equation*}
is separable. Consequently, the maximum likelihood estimate (MLE) for $\pi^*$ can be determined independently of $A^*$ using the simple algorithmic approach described in Section \ref{subsec:corr_var_model}. On the other hand, the MLE for $A^*$ is obtained by solving  
\begin{equation*}
    \min_{A \in \mathbb{R}^{d \times d}} \sum_{t=1}^{T} \|x_t - A x_{t-1}\|_2^2,
\end{equation*}
which corresponds to the problem of estimating the system matrix from a single observed realization of a time series. This is a well-known \emph{system identification} problem, which has been extensively studied in the literature (see e.g., \cite{FaraUnstable18,Simchowitz18a, Sarkar19, Jedra20}). To obtain recovery guarantees for that problem, $A^*$ is commonly assumed to be stable, i.e., its spectral radius lies within the unit circle. This is ensured by the stricter condition $\| A^* \|_2 < 1$.  Such a condition will also be needed in our analysis later on, for the estimation of $\pi^*$.
%\EA{@ Hemant: could you revise this part, since you're more familiar with this topic.}  \HT{Will do}
%
\end{remark}
%
% HT: Removing this since it is not needed for notational reasons
%
%For notational convenience, we consider in the sequel the following %optimization problem 
%
%\begin{equation}\label{eq:MLE2_sigma_known}
      %
%      \min_{\pi\in \calS_T, A\in\matR^{d\times d}  }\left\{ \sum^T_{t=1} \left(\|x_t-Ax_{t-1}\|^2_2+\frac1{\sigma^2}\|(x^\#_{\pi(t)}-Ax^\#_{\pi(t-1)})-(x_t-Ax_{t-1})\|^2_2 \right)\right\},
%      %
%\end{equation}
%
%which given the one-to-one relation between $\pi\in\calS_T$ and its inverse, %is equivalent (for all practical purposes) to \eqref{eq:MLE_sigma_known}.
%

\paragraph{Notation.} In order to \rev{rewrite \eqref{eq:MLE_sigma_known} in} a more convenient form, the following notation will be useful.
\begin{itemize}
\item The shift operator $S\in \{0,1\}^{T\times T}$,
\[
S:=\begin{bmatrix}
0 & I_{T-1} \\
0 & 0 
\end{bmatrix}.
\]
\item The data matrices $X, X^\#\in \matR^{d\times T}$, where 
\begin{align*}
    X = [x_1 \ x_2 \ \ldots \ x_T] \ \text{ and } \ 
    X^\# = [x^\#_1 \ x^\#_2 \ \ldots \ x^\#_T].
\end{align*}
\item The permutation matrix $\Pi\in \calP_T$ corresponding to the map $\pi$,
\begin{equation*}
    \Pi=\begin{bmatrix}
      e^\top_{\pi(1)}\\
       e^\top_{\pi(2)}\\\
        \vdots\\
       \,e^\top_{\pi(T)}\,
    \end{bmatrix},
\end{equation*}
where we recall that $e_t$ corresponds to the $t$-th canonical (column) vector in $\matR^{T}$.
\end{itemize}
With this notation, \rev{\eqref{eq:MLE_sigma_known} can be} rewritten as
\begin{equation}\label{eq:MLE_2_sigma_known_mat}
    \min_{\Pi\in \calP_T,A\in \matR^{d\times d}}\left\{\|X-AXS\|^2_F+\frac1{\sigma^2}\|X^\#\Pi-AX^\#\Pi S-(X-AXS)\|^2_F\right\}.
\end{equation}
The following points are useful to note.
\begin{enumerate}
    \item  For a given $A$, observe from \eqref{eq:MLE_2_sigma_known_mat} that the optimal $\Pi$ is given by 
\begin{equation}\label{eq:MLE_pi_fixed_A_known_sigma}
    \est{\Pi}_{\operatorname{MLE}}(A) \in \argmin{\Pi\in\calP_T}\,\|X^\#\Pi-AX^\#\Pi S-(X-AXS)\|^2_F.
\end{equation}
This formulation defines an optimization problem with a convex quadratic objective function in $\Pi$ (as it corresponds to the squared norm of a linear function), subject to permutation constraints. Note that it does not require the knowledge of $\sigma$. While problem \eqref{eq:MLE_pi_fixed_A_known_sigma} is  combinatorial in nature, it is unclear whether it is NP-hard. Later on in Section \ref{sec:algos_pi_fixed_A}, we will consider solving its convex relaxations (see Algorithm \ref{alg:relaxMLE_round}) which can be solved efficiently and also perform well empirically. Note that if $A = 0$ then \eqref{eq:MLE_pi_fixed_A_known_sigma} reduces to the linear assignment problem (studied in \cite{KuniskyWeed}) which is efficiently solvable.
%\EA{I guess we should ideally comment on the `hardness' of this problem.}

\item %
 For a given $\Pi$,  observe from \eqref{eq:MLE_2_sigma_known_mat} that the optimal $A$ is given by 
\begin{equation}\label{eq:MLE_A_fixed_pi_known_sigma}
    \est{A}_{\operatorname{MLE}}(\Pi) \in \argmin{A\in \matR^{d\times d}}\left\{\|X-AXS\|^2_F+\frac1{\sigma^2}\|\hX\Pi-A\hX\Pi S-(X-AXS)\|^2_F\right\},
\end{equation}
which can be solved in closed form. This is proven in the following lemma, whose proof is deferred to Appendix \ref{app:proof_MLE_A_pi_fixed_sigma_known}. Note that \eqref{eq:MLE_A_fixed_pi_known_sigma} requires knowledge of $\sigma$.
\end{enumerate}
\begin{lemma}\label{lem:MLE_A_pi_fixed_sigma_known}
For a given $\Pi \in \calP_T$, define $\est{A}_{\operatorname{MLE}}(\Pi)$ as in \eqref{eq:MLE_A_fixed_pi_known_sigma}. Then, it holds,
\begin{align*}
    \est{A}_{\operatorname{MLE}}(\Pi)=&\left[X(XS)^\top+\frac1{\sigma^2}\left(X^\# \Pi-X\right)\left(X^\# \Pi S-XS\right)^\top\right]\\
    &\times \left[(XS)(XS)^\top+\frac1{\sigma^2}(X^\# \Pi S-XS)(X^\# \Pi S-XS)^\top\right]^\dagger.
\end{align*}
\end{lemma}
%
%
%\HT{I polished a bit the above section. Note the terminology in the above discussion (fixed $\Pi$ or fixed $A$...) - this seems more appropriate, especially if we do alternating minimization to solve \eqref{eq:MLE_2_sigma_known_mat}}
As discussed in Section \ref{sec:algos_iter}, one can formulate an efficient alternating minimization heuristic (see Algorithm \ref{alg:TS_matching_alternating}) for solving \eqref{eq:MLE_2_sigma_known_mat}, by iteratively solving (i) an efficiently solvable relaxation of \eqref{eq:MLE_pi_fixed_A_known_sigma} to first obtain $\est{\Pi}$, and (ii) then using $\est{\Pi}$ in \eqref{eq:MLE_A_fixed_pi_known_sigma} to obtain $\est{A}_{\operatorname{MLE}}(\est{\Pi})$.
%
%
%\HT{added remark below}
%
\begin{remark}[Unordered base time-series] \label{rem:unord_base_time_series}
Suppose that the temporal ordering of the base time series $(x_t)_{t=1}^T$ was unknown, which simply means that we are presented with a sequence $(y_t)_{t=1}^T$ where $y_{t} = x_{{\pi_1^*}(t)}$ for an unknown permutation $\pi_1^*$. Then, the unknown parameters are $\pi_1^*, \pi^* \in \calS_T$ and $A^*$. Denote $Y = [y_1 \cdots y_T]$, so that \rev{$Y = X (\Pi_1^*)^\top$}, with $\Pi_1^* \in \calP_T$ the permutation matrix corresponding to the map $\pi_1^*$. Then it is easy to verify that the MLE \eqref{eq:MLE_2_sigma_known_mat} changes to
\begin{equation}\label{eq:MLE_unord_base_series_sigma_known_mat}
    \min_{\substack{\Pi_1, \Pi \in \calP_T\\ A\in \matR^{d\times d}}}\left\{\|Y \Pi_1-AY\Pi_1 S\|^2_F+\frac1{\sigma^2}\|X^\#\Pi-AX^\#\Pi S-(Y \Pi_1-AY\Pi_1 S)\|^2_F\right\}.
\end{equation}
As before, we can attempt solving \eqref{eq:MLE_unord_base_series_sigma_known_mat} by alternating between updates to $A$ and $(\Pi, \Pi_1)$; note that for a given $A$ the objective in \eqref{eq:MLE_unord_base_series_sigma_known_mat} is convex in $(\Pi, \Pi_1)$. Clearly, the estimated maps $\est{\pi}_1$, $\est{\pi}$ can then be used to recover the correspondence between $(y_t)_{t=1}^T$ and $(x^{\#}_t)_{t=1}^T$.
\end{remark}
%-----------------------
% Linear assignment
%-----------------------
\subsection{Linear assignment approach for estimating  \texorpdfstring{$\pi^\ast$}{pi}}\label{sec:LA}
Given the difficulties of provably solving (and analyzing) the MLE given in \eqref{eq:MLE_2_sigma_known_mat}, we consider estimating $\Pi^*$ via linear assignment (LA). This consists in solving the (linear) optimization problem
\begin{equation}\label{eq:LA_pi}
    \est{\Pi}_{\operatorname{LA}}\in \argmax{\Pi\in\calP_T}\,\langle \hX\Pi,X\rangle_F=\langle \Pi,\underbrace{(\hX)^\top X}_{=:W}\rangle_F,
\end{equation}
which does not require $A^*$ or $\sigma$ to be known. Recall that even if $A^*$ was known, finding $\MLEpi(A^*)$ would involve solving the quadratic problem \eqref{eq:MLE_pi_fixed_A_known_sigma}, which in general appears to be a hard problem, as discussed earlier. In contrast, the problem \eqref{eq:LA_pi} can be efficiently solved, using the Hungarian method \cite{Kuhn1955}, for instance. The matrix $W$ can be viewed as a similarity matrix, whose entry $(t,t')$, defined as $\langle \hx_t,x_{t'}\rangle$, represents the similarity between $\hx_t$ and $x_{t'}$.

As remarked earlier, \eqref{eq:LA_pi} was recently used in the problem of matching point clouds (see for instance \cite{KuniskyWeed, schwengber2024geometricplantedmatchingsgaussian}) which corresponds to our setup with $A^* = 0$. In particular, \eqref{eq:LA_pi} is the MLE for $\pi^*$ in that case for a fixed (not necessarily known) $\sigma$. As we will see below, this method is also meaningful in our more general setup where $A^* \neq 0$ necessarily. This is not obvious as the temporal correlation in our problem distinguishes it from the uncorrelated (i.e., $A^*=0$) case. Also note that \eqref{eq:LA_pi} does not require knowledge of $A^*$. 
%\HT{I think \eqref{eq:LA_pi} can be obtained from \eqref{eq:MLE_pi_fixed_A_known_sigma} by upper bounding the objective in a way such that ``$A$ is removed'', if you know what I mean. Might be good to put that.}

In that sense, our main guiding question is 
\begin{quotation}
 \emph{``Under what conditions on $A^*$ and $\sigma$ can guarantees be established for the linear assignment estimator $\est{\Pi}$, defined in \eqref{eq:LA_pi}, to solve the VAR permutation regression problem?''}
\end{quotation}
We now give our main result regarding the recovery guarantees of the LA estimator. Following \cite{KuniskyWeed}, we distinguish three regimes of recovery, ordered from stronger to weaker: perfect recovery, constant error, and sublinear error.  
For the upper bounds presented here, the applicable regime depends on the assumptions imposed on the noise level~$\sigma$.
%
%
%
%\HT{added remark below}
\begin{remark}\label{rem:LA_agnostic_temporal}
Continuing Remark \ref{rem:unord_base_time_series}, note that LA is agnostic to the temporal nature of the data in terms of its formulation -- it simply finds a correspondence between the columns of $X$ and $X^\#$. This means that LA run on the matrices $Y$ and $X^\#$ (where $Y$ is an unknown column-shuffling of $X$) would find a matching between the columns of $Y$ and $X^\#$. It will not recover the underlying temporal ordering for the respective time-series (unless of course $\pi_1^*$ is known to be identity).
\end{remark}
\begin{theorem}\label{thm:main_upper_bound}
    Let $s_0 := 2^{1/d}$, and let $A^* \in \matR^{d \times d}$ be such that $\|A^*\|_2 < 1$.  
    Let $X,\hX \in \matR^{d\times T}$ be observed data from the \cvar\, model, and denote $\est{\pi}_{\operatorname{LA}}$, the permutation map corresponding to the linear assignment estimator defined in \eqref{eq:LA_pi}. Define, 
    \[
    \calE:=\{t\in[T]: \est{\pi}_{\operatorname{LA}}(t)\neq \pi^*(t)\}\,
    \]
    the set of mismatched indices by $\est{\pi}_{\operatorname{LA}}$. Then the following three statements hold.
    \begin{enumerate}[label=(\roman*)]
        \item If \[
        \sigma^2\leq \frac{(1-\|A^*\|_2)^5}{4(s^{\omega(1)}_0T^{4/d}-1)}.
        \]
        Then we have $\expec[|\calE|]\to 0$, when $T\to\infty$. In particular, $|\calE|=0$ with high probability.

        \item If \[
        \sigma^2\leq \frac{(1-\|A^*\|_2)^5}{4(s^{O(1)}_0T^{4/d}-1)}.
        \]
        Then, $\expec[|\calE|]=O(1)$. In particular, for any function $f(T)=w(1)$, we have $|\calE|\leq f(T)$ with high probability.

        \item If \[
        \sigma^2\leq \frac{(1-\|A^*\|_2)^5}{4(s^{\omega(1)}_0T^{2/d}-1)}
        \]
        Then, 
        \begin{align*}
            \expec[|\calE|]=\calO\left(\left(\frac{(1-\|A^*\|_2)^5}{4\sigma^2}+1\right)^{-\frac d2}T^2\right).
        \end{align*}
    \end{enumerate}
\end{theorem}
Interestingly, the theorem recovers the results of \cite{KuniskyWeed}, up to a multiplicative factor proportional to $(1-\norm{A^*}_2)^5$.  
%\HT{this is correct right?}.  
In particular, when $\|A^*\|_2 = 0$, parts~$(i)$ and~$(ii)$ yield the same noise condition as in their work, except for a factor of $4$ in the denominator.  
In the case of part~$(iii)$, our condition is also of the same order as theirs. 

We present the proof of Theorem \ref{thm:main_upper_bound} in the next section. 
%Compared to the approach in \cite{KuniskyWeed}, our method is arguably simpler and extends naturally to the time-series framework considered here, whereas their technique does not appear to adapt readily. 
As we will see, the time-series case introduces additional technical challenges, some of which lead to results that may be of independent interest.   
%
% \begin{theorem}[Upper bound]\label{thm:main_upper_bound}
% \EA{Here goes the upper bound result...}
% \end{theorem}
% %
% \begin{theorem}[Lower bound]\label{thm:main_lower_bound}
% \EA{Here goes the lower bound result...}
% \end{theorem}
%
\begin{remark}[On the factor $(1-\|A^*\|_2)^5$] \label{rem:spec_grap_factor}
 Theorem \ref{thm:main_upper_bound} requires $\normnew{A^*}_2 < 1$, which is stronger than requiring $\rho(A^*) < 1$. Such stability assumptions on $A^*$ are common for parameter estimation problems for VAR models. While recent results have tackled this under weaker conditions -- such as that of marginal stability wherein $\rho(A^*) \leq 1$ (see \cite{Simchowitz18a, Sarkar19}), or even unstable systems (e.g., \cite{FaraUnstable18, Sarkar19}) -- a number of prior results consist of estimation error bounds which tend to worsen as $\normnew{A^*}_2$ approaches one (e.g., \cite{Fang15, Loh12}). In that sense, we believe that the condition $\normnew{A^*}_2 < 1$ needed in our results is an artefact of the proof, and can be weakened as well. Establishing this is an interesting direction for future work.
\end{remark}

\begin{remark}[Gaussian assumption on $(\xi_t)_{t=1}^T$]
The assumption that $\xi_t \stackrel{\text{i.i.d}} {\sim}\calN(0,I_d)$ is mainly for convenience for the theoretical analysis in Section \ref{sec:analysis}, and also to be able to formulate the MLE as described already. This assumption was also made in \cite{KuniskyWeed} (for the i.i.d setting where $A^*= 0$), but was relaxed in \cite{schwengber2024geometricplantedmatchingsgaussian} to the more general class of sub-Gaussian distributions. We believe it should be similarly possible to extend our setup to the sub-Gaussian setting.
%
%
%    \HT{Need to mention about the assumption that $\xi_t$'s are standard Gaussians. If these are not Gaussian's then (i) the MLE will not be valid, and (ii) the analysis for LA will change at some points in a non-trivial (??) way (maybe the strategy of the paper of Oliviera  could then be used here?)}
\end{remark}
%
%
%------------------------------------------
% Analysis for the linear assignment estimator
%---------------------------------------------
\section{Analysis for the linear assignment estimator}\label{sec:analysis}
In this section, we present the key techniques used to prove Theorem \ref{thm:main_upper_bound}. Section \ref{sec:analysis_upper_bound} introduces the first moment method and the counting of augmenting paths, techniques commonly used in matching problems \cite{KuniskyWeed,schwengber2024geometricplantedmatchingsgaussian}. 
%
%We emphasize the key differences between the correlated (time series matching) and uncorrelated (point cloud matching) cases. %In Section \ref{sec:analysis_lower_bound}, we describe a lower bound technique based on random geometric graphs, originally developed for matching point clouds in the independent case \cite{schwengber2024geometricplantedmatchingsgaussian}. Adapting this approach to the correlated setting requires significant modifications, which we detail in our analysis.
%
\subsection{Augmenting cycles in the analysis of 
\texorpdfstring{\eqref{eq:LA_pi}}{(LA-pi)}
}\label{sec:analysis_upper_bound}
% \subsection{Augmenting cycles in the analysis of \eqref{eq:LA_pi}}\label{sec:analysis_upper_bound}
%
The technique used in \cite{KuniskyWeed,schwengber2024geometricplantedmatchingsgaussian} to obtain upper bounds on the error of the LA estimator $\est{\Pi}$ consists of counting the number of augmenting cycles. A cycle $C_t$ of length $t\leq T$ is a sequence $C_t=(i_1,\ldots,i_t)$ consisting of distinct indices $i_1,\ldots,i_t\in [T]$. We say that $C_t$ is an augmenting $t$-cycle, if and only if (recall $W = (\hX)^\top X$) 
\begin{equation}\label{eq:augmenting}
    \sum^t_{k=1}W_{i_ki_{k+1}}\geq \sum^t_{k=1}W_{i_ki_k}\,,
\end{equation}
%
% We define the set of miss-matched indices (or error set) 
% %
% \[
% \calE:=\{t\in[T]: \est{\pi}(t)\neq \pi^*(t)\}\,.
% \]
%
%
where $i_{t+1} := i_1$. To see the importance of the augmenting cycles for upper bounding\footnote{Recall the definition of $\calE$ as the set of mismatched indices in Theorem \ref{thm:main_upper_bound}, and note that $|\calE|$ corresponds to the Hamming distance between $\est{\pi}$ and $\pi^*$.} $|\calE|$, assume without loss of generality\footnote{This is common in the analysis of matching problems.} $\pi^*=\operatorname{id}$. It is easy to see that the elements of $\calE$ belong to a union of disjoint augmenting $t$-cycles of $\est{\pi}$, for different $t \in \set{2,\dots,T}$. 
%
%if \eqref{eq:augmenting} holds, for some cycle $C$ of length $t$, then $\est{\pi}$ has to mismatch all the elements in $C$ (otherwise this will contradict the ``maximality'' of $\est{\pi}$). 
This then implies  
\begin{equation*}
    |\calE|\leq \sum^T_{t=2}tN_t \quad \text{ where } \ N_t:=\sum_{(i_1,\ldots,i_t) \text{is $t$-cycle}}\indic_{\{(i_1,\ldots,i_t)\text{ is augmenting}\}}
\end{equation*}
represents the number of augmenting $t$-cycles. To guarantee perfect recovery, we will rely on the \emph{first moment method}, which involves bounding the expected value of the error. For that, we have 
\begin{align}
    \expec[|\calE|]\leq \sum^T_{t=2}t\expec[N_t]  \label{eq:first_moment}
    = \sum^T_{t=2}t\sum_{(i_1,\ldots, i_t) \text{is $t$-cycle}}\prob\big((i_1,\ldots,i_t)\text{ is augmenting}\big).
\end{align}
Therefore, a fundamental step is to bound the probability that a cycle $C_t$ is augmenting. Section \ref{subsec:analsysis_LA_2cycle} analyzes this step for the case $t=2$ to provide intuition for the general case in Section \ref{subsec:analsysis_LA_gen_cycle}.
%%Once a bound for $\expec[|\calE|]$ is stablish, the use of the first moment method 
%
%
\subsection{Warm-up: augmenting \texorpdfstring{$2$}{2}-cycles} \label{subsec:analsysis_LA_2cycle}
We begin by tackling the case of augmenting $2$-cycles, since it already contains the core of the argument. According to \eqref{eq:augmenting}, a $2$-cycle $C=(a\,, b)$, for a given pair $a,b\in[T]$, is augmenting if  
\begin{align}\label{eq:2cycle_1}
    W_{ab}+W_{ba}\geq W_{aa}+W_{bb},
\end{align}
or equivalently 
\begin{align}\label{eq:2cycle_2}
    (\hx_a)^\top x_b+(\hx_b)^\top x_a\geq (\hx_a)^\top + (\hx_b)^\top x_b.
\end{align}
Recall that
\begin{align*}
    \hx_a = A^*\hx_{a-1}+\xi_a+\sigma \txi_a \quad \text{ and } \quad x_a=A^*x_{a-1}+\xi_a,
\end{align*}
where we use the convention $\hx_0=x_0=0$. Hence, 
\begin{align*}
    (\hx_a)^\top x_b&=\left(A^*\hx_{a-1}+\xi_a+\sigma\txi_a\right)^\top (A^*x_{b-1}+\xi_b)\\
    (\hx_a)^\top x_a&=\left(A^*\hx_{a-1}+\xi_a+\sigma\txi_a\right)^\top (A^*x_{a-1}+\xi_a).
\end{align*}
It is easy to see that \eqref{eq:2cycle_2} it is equivalent to 
\begin{gather}\label{eq:2cycles_3}
\begin{aligned}
    \|\xi_a-\xi_b\|^2_2 &\leq \sigma \langle \txi_a-\txi_b,\xi_b-\xi_a\rangle+ \langle A^*\hx_{a-1}-A^*\hx_{b-1},\xi_b-\xi_a\rangle\\
                           & + \langle A^*\hx_{a-1}-A^*\hx_{b-1},A^*x_{b-1}-A^*_{a-1}\rangle\\
                           & +\langle \xi_b+\sigma\txi_b,A^*x_{a-1}-A^*x_{b-1}\rangle+\langle \xi_a+\sigma\txi_a,A^*x_{b-1}-A^*x_{a-1}\rangle.
\end{aligned}
\end{gather}
Notice that when $A^*=0$, we obtain $\|\xi_a-\xi_b\|^2_2\leq \sigma \langle \txi_a-\txi_b,\xi_b-\xi_a\rangle$, which appears in the upper bound argument in \cite{KuniskyWeed,schwengber2024geometricplantedmatchingsgaussian}. In comparison, here we have to deal with more complicated expressions.
Since $\hx_a=x_a+\sigma\tx_a$, we have (after some algebra) that \eqref{eq:2cycle_2} is equivalent to 
\begin{align}\label{eq:2cycle_4}
    \sigma \langle \underbrace{A^*(\tx_{a-1}-\tx_{b-1})+(\txi_a-\txi_b)}_{=:\ty_{ab}},\underbrace{A^*(x_{b-1}-x_{a-1})+\xi_b-\xi_a}_{=:y_{ab}}\rangle\geq \|A^*(x_{a-1}-x_{b-1})-(\xi_b-\xi_a)\|^2_2.
\end{align}
Assuming w.l.o.g that $a>b$, we have that conditioned on $\xi_1,\ldots, \xi_a$ (so that $y_{ab}$ is fixed) the left hand side of \eqref{eq:2cycle_4} is a Gaussian random variable. The following lemma more specifically characterizes this distribution. Its proof can be found in Appendix \ref{app:proof_Gaussian_2cycles}.
\begin{lemma}\label{lem:Gaussian_2cycles}
    Let $(x_t)_{t\in[a]}$,$(\tx_t)_{t\in[a]}$, $(\xi_t)_{t\in[a]}$ and $(\txi_t)_{t\in[a]}$ be as in \cvar. For $a>b$, define the variables
    \begin{align*}
        y_{ab} := A^*(x_{b-1}-x_{a-1})+(\xi_{b}-\xi_{a}),\\
        \ty_{ab} := A^*(\tx_{a-1}-\tx_{b-1})+(\txi_{a}-\txi_{b}).
    \end{align*}
    Then, conditional on $(\xi_i)^a_{i=1}$, we have
 $\sigma\langle \ty_{ab},y_{ab}\rangle 
  \sim \calN(0,\sigma^2\left(\tilde{\sigma}^2_1+\tilde{\sigma}^2_2)\right),$  
    where
    \begin{align*}\
        \tilde{\sigma}^2_1&=y^\top_{ab}\left(\sum^{b-1}_{i=0}(A^*)^i\left[(A^*)^{a-b}-I_d \right]\left[(A^*)^{a-b}-I_d \right]^\top({A^*}^\top)^i\right)y_{ab}\\
        \tilde{\sigma}^2_2&=y^\top_{ab}\left(\sum^{a-b}_{i=1}(A^*)^{a-b-i}({A^*}^\top)^{a-b-i}\right)y_{ab}.
    \end{align*}
    If $\normnew{A^*}_2 < 1$, we further have that 
    \begin{equation}\label{eq:sigma_bound}
        \sigma^2(\tilde{\sigma}^2_1+\tilde{\sigma}^2_2)\leq \frac{5\sigma^2\|y_{ab}\|^2_2}{1-\|A^*\|^2_2}.
    \end{equation}
\end{lemma}
%
%\HT{I edited somethings in the above lemma -- we were saying that the conditional expectation is gaussian which seems incorrect? I phrased it as written in the notes now. This means the density of $\sigma\langle \ty_{ab},y_{ab}\rangle$ conditional on ($\cdots$) is gaussian. Another correction: note that we are conditioning on $(\xi_i)^a_{i=1}$, not on $y_{ab}$ (its not the same thing!)} 
From \eqref{eq:2cycle_4} and Lemma \ref{lem:Gaussian_2cycles} we have that 
\begin{align*}\label{eq:cycle}
    &\prob\left(C=(a,b)\text{ is augmenting}\right) \\    &=\expec_{\xi_1,\ldots,\xi_a}\left[\prob\left(\sigma\langle \ty_{ab},y_{ab}\rangle\geq \|y_{ab}\|^2_2\,\Big|(\xi_i)^a_{i=1}\right)\right]\\
    &=\expec_{\xi_1,\ldots,\xi_a}\left[\prob\left(\underbrace{\frac{\langle \ty_{ab},y_{ab}\rangle}{\sqrt{\tilde{\sigma}^2_1+\tilde{\sigma}^2_2}}}_{=:g}\geq \frac{\|y_{ab}\|^2_2}{\sigma\sqrt{\tilde{\sigma}^2_1+\tilde{\sigma}^2_2}}\,\Bigg|(\xi_i)^a_{i=1}\right)\right]\\
    &\leq \expec_{\xi_1,\ldots,\xi_a}\left[\prob\left(g\geq \sqrt{\frac{1-\|A^*\|^2_2}{5\sigma^2}}\|y_{ab}\|^2_2\,\Bigg|(\xi_i)^a_{i=1}\right)\right] \tag{using \eqref{eq:sigma_bound}}\\
    &\leq \expec_{\xi_1,\ldots,\xi_a}\left[\exp\left(-\frac{(1-\|A^*\|^2_2)}{10\sigma^2}\|y_{ab}\|^2_2\right)\right] 
    \tag{conditional on $(\xi_i)^a_{i=1}$,  $g \sim \calN(0,1)$ \text{ by Lemma }\ref{lem:Gaussian_2cycles}}.
\end{align*}
Denoting $\xi_{1:a}:=(\xi_1^\top, \xi_2^\top,\dots, \xi_a^\top)^\top$, we have that 
$\|y_{ab}\|^2_2=\xi_{1:a}^\top BB^\top \xi_{1:a}$
where 
\begin{equation*}
    B:=\begin{bmatrix}
        \left((A^*)^{a-b}-I_d\right)^\top({A^*}^\top)^{b-1}\\
        \vdots\\
        \left((A^*)^{a-b}-I_d\right)^\top({A^*}^\top)\\
        \left((A^*)^{a-b}-I_d\right)^\top\\
        (A^*)^{a-b-1}\\
        \vdots\\
        {A^*}^\top\\
        I_d
    \end{bmatrix}.
\end{equation*}
Thus, 
\begin{align*}
    \expec_{\xi_{1:a}}\left[\exp\left(-\frac{(1-\|A^*\|^2_2)}{10\sigma^2}\xi_{1:a}^\top BB^\top \xi_{1:a}\right)\right]
    &=\frac{1}{(2\pi)^{\frac{ad}2}}\int_{\matR^{ad}}\exp\Big(-\frac{\xi^\top \xi}{2} \Big)\exp\left(-\underbrace{\frac{(1-\|A^*\|^2_2)}{10\sigma^2}}_{=:\alpha}\xi^\top BB^\top \xi\right)d\xi\\
    &=\frac1{(2\pi)^{\frac{ad}2}}\int_{\matR^{ad}}\exp\left(-\frac{\xi^\top}2(2\alpha BB^\top+I_{ad})\xi\right)d\xi\\
    &=\operatorname{det}(2\alpha BB^\top+I_{ad})^{-\frac12}
\end{align*}
and we obtain the bound 
\begin{equation} \label{eq:2cycle_fin_bound}
    \prob\left(C=(a,b)\text{ is augmenting}\right)\leq \operatorname{det}\left( \frac{(1-\|A^*\|^2_2)}{5\sigma^2}BB^\top+I_{ad}\right)^{-\frac12}.
\end{equation}
\subsection{General augmenting \texorpdfstring{$t$}{t}-cycles} \label{subsec:analsysis_LA_gen_cycle} 
The analysis of the general case $t\geq 2$ is similar  to that of $t=2$, but involves cumbersome calculations. The following proposition, whose proof can be found in Section  \ref{subsec:proof_Prop_tcycles} summarizes our findings in this case. For $t = 2$, note that the bound stated in Proposition \ref{prop:tcycles} is not necessarily always worse than that in \eqref{eq:2cycle_fin_bound}. 
%\HT{maybe explain a bit why this is the case? Or not.}
%\HT{I say $t \geq 2$ now since this applies for $t = 2$ as well. Probably, we should compare when the probability of augmenting bound for $t = 2$ (derived earlier) is strictly better than what we get in the general analysis. I think it depends on the regime in which $\norm{A^*}_2$ lies...} 
%
\begin{proposition}\label{prop:tcycles}
    For $t\geq 2$, let $C_t=(i_1,\ldots, i_t)$ be a $t$-cycle with $i_1$ the largest amongst the $i_k$'s (w.lo.g). Then if $\normnew{A^*}_2 < 1$, it holds that
    \begin{equation}\label{eq:augmenting_proba_bound}
        \prob(C_t\text{ is augmenting})\leq \operatorname{det}\left(\frac{(1-\|A^*\|_2)^3}{4\sigma^2}L+I_{i_1d}\right)^{-\frac12},
    \end{equation}
    where 
    \begin{equation*}
        L:=\sum^t_{k=1}B(\alpha_k,\beta_k) B^\top(\alpha_k,\beta_k),
    \end{equation*}
    with $B^\top(\alpha_k,\beta_k) \in \matR^{d \times (i_1 d)}$ defined as the matrix 
    \begin{equation*}
        \begin{bmatrix}
            (A^*)^{\alpha_k-1}-(A^*)^{\beta_k-1}& \ldots& (A^*)^{\alpha_k-\beta_k}-I_d&(A^*)^{\alpha_k-\beta_k-1}&\ldots &A^*\, &I_d & 0&\ldots& 0
        \end{bmatrix}.
    \end{equation*}
    Moreover,  $\alpha_t :=i_1, \, \beta_t:=i_t$, and 
    \begin{align*}
        \alpha_k &:=\max\{i_k,i_{k+1}\},\, \beta_k:=\min\{i_k,i_{k+1}\}, \text{ for } 1\leq k \leq t-1. 
    \end{align*}
\end{proposition}
%%%
%%%
%\HT{Proposition below will now change due to improvement!}
Equipped with Proposition \ref{prop:tcycles} and equation \eqref{eq:first_moment}, it remains only to bound the right-hand side of \eqref{eq:augmenting_proba_bound} in order to obtain an upper bound on the expected error $\expec [|\calE|]$. The next proposition %, whose proof is delayed to Appendix \ref{app:proof_det_bound}
establishes the required estimate. 
%under the additional assumption that $A^*$ is symmetric -- note that this was not needed in Proposition \ref{prop:tcycles}.
%
%
\begin{proposition}\label{prop:det_bound}
Under the assumptions and notation of Proposition~\ref{prop:tcycles}, 
%suppose additionally that $A^*$ is symmetric. 
we have that  
\begin{equation}\label{eq:prop_cycles_result}
\operatorname{det}\left(\frac{(1-\|A^*\|_2)^3}{4\sigma^2} L + I_{i_1 d}\right)^{-\frac{1}{2}} \leq \left(\frac{(1-\|A^*\|_2)^5}{4\sigma^2}+1\right)^{-(t-1)\frac{d}{2}}.
\end{equation}
\end{proposition}
%
% It is well known (see e.g., \cite{spielmanSGT}) that the eigenvalues of $C_t$ correspond to the (unordered) values 
% %
% \begin{align*}
%     2\left(1 - \cos\left(\frac{2\pi (t-k)}{t}\right)\right); \quad k=1,\dots,t.
% \end{align*}
% % 
%
The proof of Proposition \ref{prop:det_bound} is outlined in Section \ref{subsec:proof_prop_det_bound}. 
Theorem \ref{thm:main_upper_bound} is then proved by combining Propositions \ref{prop:tcycles} and \ref{prop:det_bound} with \eqref{eq:first_moment}. The reader is referred to Appendix \ref{app:proof_main_upper_bound} for a complete derivation.

\begin{remark}[Comparison with the analysis in \cite{KuniskyWeed}] \label{rem:comp_kunisky}  
Although our analysis for the LA estimator (in the VAR setting) is inspired by \cite{KuniskyWeed}, it is useful to elaborate on the technical challenges arising here as compared to \cite{KuniskyWeed}. 
\begin{enumerate}
    \item The proof of Proposition \ref{prop:tcycles}  follows, in spirit, along the same lines as that of \cite[Prop. 3.1]{KuniskyWeed}. However the underlying calculations in our case are considerably more involved and tedious, due to the VAR structure on the model.

    \item  Secondly, Proposition \ref{prop:det_bound} has a considerably easy analogue in \cite{KuniskyWeed} -- indeed, if $A^* = 0$, then the spectrum of $L$ is known in closed form and bounding the $\det$ term is straightforward. In our setting, however, the spectrum of $L$ is not known in closed form, and therefore, bounding the $\det$ term requires additional technical details. In particular, our argument relies only on information about the pseudo determinant of $L$.

    \item  Finally, once Propositions \ref{prop:det_bound} and \ref{prop:tcycles} are established, the remaining calculations needed to obtain Theorem \ref{thm:main_upper_bound} are essentially along the same lines as in \cite{KuniskyWeed}.
\end{enumerate}
%
%\HT{I rewrote this remark a bit to highlight the difference in the arguments more clearly.}
%
%
%\begin{remark}[On symmetry of $A^*$]
%The requirement that $A^*$ is symmetric is essentially made for %convenience and is needed to establish Lemma \ref{lem:det_PtP} 
%%(within the proof of Proposition \ref{prop:det_bound}). This %assumption can likely be removed by using the Jordan canonical %form of $A^*$ instead.
%\end{remark}
%
%
%its main advantage is that it extends naturally to the time-series framework considered here. The main limitation of their approach is that it requires, in principle, full knowledge of the spectrum of $L$, or at least precise bounds on each individual eigenvalue. By contrast, our argument relies only on information about the pseudo determinant. On the other hand, our bound comes with an extra constant factor. It is worth noting, however, that sharper constants could be obtained by making more optimal use of some of the intermediate bounds we derive. In this paper, we have not attempted to optimize these constants. \HT{I will re-write more carefully this part. Maybe put it a bit earlier. Note that the methods are the same (linear assignment) -- what we want to talk about is the analysis steps.}
\end{remark}
\subsection{Proof of Proposition \ref{prop:tcycles}} \label{subsec:proof_Prop_tcycles}
To ease notation, we sometimes use $A$ instead of $A^*$ in this proof; especially during a long sequence of calculations. Suppose $t\geq 2$ and $C_t=(i_1,i_2,\ldots, i_t)$, where we assume w.l.o.g that $i_1$ is the largest among the indices in $C_t$. In this case, $C_t$ is augmenting, if and only if
\begin{equation}\label{eq:augmenting_t_cycle}
    {\hx_{i_t}}^\top x_{i_1}+\sum^{t-1}_{k=1}{\hx_{i_k}}^\top x_{i_{k+1}}\geq \sum^t_{k=1}{\hx_{i_k}}^\top x_{i_k}.
\end{equation}
Recalling that 
\begin{align*}
    \begin{cases}
        \hx_{i_k}=A\hx_{i_{k}-1}+\xi_{i_k}+\sigma \txi_{i_k}\\
        x_{i_k}=Ax_{i_{k}-1}+\xi_{i_k},
    \end{cases}
\end{align*}
we have 
\begin{align*}
    {\hx_{i_k}}^\top x_{i_{k+1}}&=\langle A\hx_{i_{k-1}},Ax_{i_{k+1}-1}\rangle +\langle A\hx_{i_{k-1}},\xi_{i_{k+1}}\rangle+\langle \xi_{i_k}+\sigma \tilde{\xi}_{i_k},Ax_{i_{k+1}-1}\rangle+\langle \xi_{i_k}+\sigma \tilde{\xi}_{i_k},\xi_{i_{k+1}}\rangle.
\end{align*}
Then \eqref{eq:augmenting_t_cycle} is equivalent to 
\begin{align*}
    &\langle A\hx_{i_{t-1}},Ax_{i_1-1}\rangle +\langle A\hx_{i_t-1},\xi_{i_1}\rangle +\langle \xi_{i_t}+\sigma \tilde{\xi}_{i_t},Ax_{i_1-1}\rangle+ \langle \xi_{i_t}+\sigma \tilde{\xi}_{i_t},\xi_{i_1}\rangle \\
    & +\sum^{t-1}_{k=1}\langle A\hx_{i_{k-1}},Ax_{i_{k+1}-1}\rangle+\sum^{t-1}_{k=1} \langle A\hx_{i_k-1},\xi_{i_{k+1}}\rangle+\sum^{t-1}_{k=1}\langle \xi_{i_k}+\sigma \tilde{\xi}_{i_k},Ax_{i_{k+1}-1}\rangle+\sum^{t-1}_{k=1} \langle \xi_{i_k}+\sigma \tilde{\xi}_{i_k},\xi_{i_{k+1}}\rangle\\
    &\geq \langle A\hx_{i_{t}-1},Ax_{i_t-1}\rangle +\langle A\hx_{i_t-1},\xi_{i_t}\rangle +\langle \xi_{i_t}+\sigma \tilde{\xi}_{i_t},Ax_{i_t-1}\rangle+ \langle \xi_{i_t}+\sigma \tilde{\xi}_{i_t},\xi_{i_t}\rangle \\
    & +\sum^{t-1}_{k=1}\langle A\hx_{i_{k}-1},Ax_{i_k-1}\rangle+\sum^{t-1}_{k=1} \langle A\hx_{i_k-1},\xi_{i_{k}}\rangle+\sum^{t-1}_{k=1}\langle \xi_{i_k}+\sigma \tilde{\xi}_{i_k},Ax_{i_{k}-1}\rangle+\sum^{t-1}_{k=1} \langle \xi_{i_k}+\sigma \tilde{\xi}_{i_k},\xi_{i_{k}}\rangle.
\end{align*}
The previous inequality is equivalent to 
\begin{align*}
    &\langle A\hx_{i_t-1},A(x_{i_1-1}-x_{i_t-1})\rangle 
    + \sum^{t-1}_{k=1}\langle A\hx_{i_k-1},A(x_{i_{k+1}-1}-x_{i_k-1})\rangle 
    + \langle Ax_{i_t-1},\xi_{i_1}-\xi_{i_t}\rangle +\sum^{t-1}_{k=1}\langle A\hx_{i_k-1},\xi_{i_{k+1}}-\xi_{i_k}\rangle \\
    &+ \langle \xi_{i_t}+\sigma \txi_{i_t},A(x_{i_1-1}-x_{i_t-1})\rangle 
    + \sigma \langle \txi_{i_t},\xi_{i_1}-\xi_{i_t}\rangle 
    + \sigma\sum^{t-1}_{k=1}\langle \txi_{i_k},\xi_{i_{k+1}} -\xi_{i_k}\rangle 
    + \sum^{t-1}_{k=1} \dotprod{\xi_{i_k} + \sigma \txi_{i_k}}{A^*(x_{i_{k+1}-1} - x_{i_k-1})}
    \\
    &\geq \frac12\left(\|\xi_{i_t}-\xi_{i_1}\|^2_2+\sum^{t-1}_{k=1}\|\xi_{i_k}-\xi_{i_{k+1}}\|^2_2\right).
\end{align*}
Writing $\hx_k=x_k+\sigma\tx_k$, for $k\in[t]$, we obtain after some re-shuffling and simple algebra, that \eqref{eq:augmenting_t_cycle} is equivalent to
\begin{align*}
    &\sigma \langle A\tx_{i_t-1}+\txi_{i_t},A(x_{i_1-1}-x_{i_t-1})+\xi_{i_1}-\xi_{i_t}\rangle+\sigma \sum^{t-1}_{k=1}\langle A\tx_{i_k-1}+\txi_{i_k},A(x_{i_{k+1}-1}-x_{i_k-1})+\xi_{i_{k+1}}-\xi_{i_k}\rangle\\
    &\geq \frac12\left( \|A(x_{i_t-1}-x_{i_{1}-1})-(\xi_{i_{1}}-\xi_{i_t})\|^2_2+\sum^{t-1}_{k=1}\|A(x_{i_k-1}-x_{i_{k+1}-1})-(\xi_{i_{k+1}}-\xi_{i_k})\|^2_2 \right).
\end{align*}
Now, denoting
\begin{align*}
    y_{t}&:=A(x_{i_t-1}-x_{i_{1}-1})-(\xi_{i_{1}}-\xi_{i_t}),\\
    y_{k}&:=A(x_{i_{k+1}-1}-x_{i_{k}-1})-(\xi_{i_{k+1}}-\xi_{i_k}), \, \text{ for }k=1,\ldots, t-1,
\end{align*}
we have that \eqref{eq:augmenting_t_cycle} is equivalent to 
\begin{align} \label{eq:temp1_1_equiv}
    \sigma \sum^{t-1}_{k=1}\langle A\tx_{i_k-1}+\txi_{i_k},y_k\rangle \geq \frac12\left(\|y_t\|^2_2+\sum^{t-1}_{k=1}\|y_k\|^2_2\right).
\end{align}
Note that, conditioned on $\xi_1,\ldots,\xi_{i_1}$ the LHS above is a zero mean real Gaussian. To find its variance, we define $\pi:\{2,\ldots,t\}\rightarrow \{2,\ldots,t\}$ be the permutation that defines the ordering $i_1>i_{\pi(2)}>i_{\pi(3)}>\ldots>i_{\pi(t)}$, i.e., $i_{\pi(2)}$ and $i_{\pi(t)}$ are, respectively, the largest and the smallest indices among $i_2,i_3,\ldots,i_t$. With this, we have 
\begin{align*}
    g
    := \sigma \sum^t_{k=1}\langle A\tx_{i_k-1}+\txi_k,y_k\rangle&=\sigma \sum^t_{k=1}\sum^{i_k}_{j=1}y^\top_k A^{i_k-j}\txi_j\\
    &=\sigma \sum^t_{k=2}\sum^{i_{\pi(k)}}_{j=1}y^\top_{\pi(k)}A^{i_{\pi(k)}-j}\txi_j+\sigma \sum^{i_1}_{j=1}y^\top_1 A^{i_1-j}\txi_j \\
    &=\sigma\Bigg[\sum^{i_{\pi(t)}}_{j=1} \left(y^\top_1A^{i_1-j}+\sum^{t-2}_{k=0}y^\top_{\pi(t-k)}A^{i_{\pi(t-k)}-j}\right)\txi_j\\
    &+ \sum_{l=0}^{t-3} \left(\sum^{i_{\pi(t-l-1)}}_{j=i_{\pi(t-l)}+1} \left( y^\top_1 A^{i_1-j}+\sum_{k=l+1}^{t-2} y^\top_{\pi(t-k)}A^{i_{\pi(t-k)}-j} \right) \txi_j \right) \\
    &+\sum^{i_1-i_{\pi(2)}}_{k=1}y^\top_1A^{i_1-i_{\pi(2)}-k}\txi_{i_{\pi(2)+k}}\Bigg].
\end{align*}
Hence, $g\sim \calN(0,\sigma^2\bar{\sigma}^2)$, where 
\begin{align*}
    \bar{\sigma}^2& = \underbrace{\sum^{i_{\pi(t)}}_{j=1}\left\|y^\top_1A^{i_1-j}+\sum^{t-2}_{k=0}y^\top_{\pi(t-k)}A^{i_{\pi(t-k)}-j}\right\|^2_2}_{=:\bar{\sigma}^2_1}\\
    &+ \underbrace{\sum_{l=0}^{t-3} \sum^{i_{\pi(t-l-1)}}_{j=i_{\pi(t-l)}+1}\left\|y^\top_1A^{i_1-j}+\sum_{k=l+1}^{t-2} y^\top_{\pi(t-k)}A^{i_{\pi(t-k)}-j}\right\|^2_2}_{=:\bar{\sigma}^2_2}\\
    &+ \underbrace{\sum^{i_1-i_{\pi(2)}}_{k=1}\left\|y^\top_1A^{i_1-i_{\pi(2)}-k}\right\|^2_2.}_{=:\bar{\sigma}^2_3}
\end{align*}
We will now bound $\bar{\sigma}^2$, with the help of the following result. From now onwards, we will consider $\pi(1)=1$. 
%\EA{I added something here for $\pi(1)=1$, since it is not explained in the notes...} \HT{yes moved this here} \EA{I tried to write the part above a bit more concisely than the notes, where is written with the notation... instead of sums. The term $\bar{\sigma}_2$ needs to be checked.\HT{Yes, so I corrected that term, I think the above version is correct? There is an outer sum over $l$ and some indices changed inside}} 
%
\begin{lemma}
For $\bar{\sigma}$ as defined above, it holds that 
    \begin{equation*}
        \bar{\sigma}^2\leq \left(\frac1{1-\|A^*\|_2}\right)^3\sum^t_{k=1}\|y_{k}\|^2_2.
    \end{equation*}
\end{lemma}
\begin{proof}
The proof consists in bounding the terms $\bar{\sigma}^2_1,\bar{\sigma}^2_2,\bar{\sigma}^2_3$ separately, in three stages. 
\paragraph{Bound on $\bar{\sigma}^2_1$.} We define
\begin{equation*}
    M_0:=\begin{bmatrix}A^{i_{\pi(1)}-i_{\pi(t)}}& A^{i_{\pi(2)}-i_{\pi(t)}}&\ldots & \underbrace{A^{i_{\pi(t)}-i_{\pi(t)}}}_{=I_d}\end{bmatrix}.
\end{equation*}
With this, we can rewrite and bound $\bar{\sigma}^2_1$ as follows.
\begin{align*}
    \bar{\sigma}^2_1=\sum^{i_{\pi(t)}}_{j=1}\left\|A^{j-1} M_0\begin{bmatrix}y_{\pi(1)}\\y_{\pi(2)}\\\vdots\\y_{\pi(t-1)} 
    \\y_{\pi(t)}\end{bmatrix}\right\|^2_2
    \ &\leq \ \left(\sum^{i_\pi(t)}_{j=1}\|A\|^{2(j-1)}_2\left\|M_0\begin{bmatrix} y_{\pi(1)}\\y_{\pi(2)}\\\vdots\\y_{\pi(t-1)}\\ y_{\pi(t)}\end{bmatrix}\right\|^2_2\right)\\
    \ &\leq \ \frac1{1-\|A\|^2_2}\left\|M_0\begin{bmatrix}y_{\pi(1)}\\y_{\pi(2)}\\\vdots\\y_{\pi(t-1)}\\ y_{\pi(t)}\end{bmatrix}\right\|^2_2.
\end{align*}
\paragraph{Bound on $\bar{\sigma}^2_2$.} For $l = 1,\dots,t-2$, let 
\begin{equation*}
    M_l:=\begin{bmatrix}
        A^{i_{\pi(1)}-i_{\pi(t-l)}}&A^{i_{\pi(2)}-i_{\pi(t-l)}}& \ldots & A^{i_{\pi(t-l)}-i_{\pi(t-l)}}
    \end{bmatrix}. 
\end{equation*}
%
%Define, for $j=1,\ldots, t-2$,
%
Then we have that 
\begin{align*}
    \bar{\sigma}^2_{2,l} 
    \ := \ \sum\limits^{i_{\pi(t-l)}-i_{\pi(t-l+1)} -1}_{k=0}\left\|A^k M_l 
    \begin{bmatrix}
    y_{\pi(1)}\\y_{\pi(2)}\\\vdots\\y_{\pi(t-1)}\\ y_{\pi(t-l)}
    \end{bmatrix}\right\|^2_2
    \ \leq \ \frac1{1-\|A\|^2_2}\left\|
    M_l
    \begin{bmatrix}
    y_{\pi(1)}\\y_{\pi(2)}\\\vdots\\y_{\pi(t-1)}\\ y_{\pi(t-l)}
    \end{bmatrix}\right\|^2_2.
\end{align*}
Notice that $\sum^{t-2}_{l=1}\bar{\sigma}^2_{2,l}=\bar{\sigma}^2_2$. Hence, 
\[
\bar{\sigma}^2_2\leq \frac1{1-\|A\|^2_2}\sum^{t-2}_{l=1}\left\|M_l\begin{bmatrix}y_{\pi(1)}\\y_{\pi(2)}\\\vdots\\y_{\pi(t-1)}\\ y_{\pi(t-l)}\end{bmatrix}\right\|^2_2.
\]
\paragraph{Bounding $\bar{\sigma}^2_3$.} We have 
\begin{align*}
    \bar{\sigma}^2_3=\sum^{i_{\pi(1)}-i_{\pi(2)}}_{j=1}\left\|y^\top_1A^{i_{\pi(1)}-i_{\pi(2)}-j}\right\|^2_2
    &\leq \|y_1\|^2_2\sum^{i_{\pi(1)}-i_{\pi(2)}}_{j=1}\left\|A\right\|^{2(i_{\pi(1)}-i_{\pi(2)}-j)}_2\\
    &\leq \frac{\|y_1\|^2_2}{1-\|A\|^2_2}=\frac1{1-\|A\|^2_2}\|\underbrace{M_{t-1}}_{=I_d}y_1\|^2_2.
\end{align*}
Summarizing, we have 
\begin{equation*}
    \bar{\sigma}^2\leq \frac{1}{1-\|A\|^2_2}\sum^{t-1}_{k=0}\left\|M_k\begin{bmatrix}y_{\pi(1)}\\y_{\pi(2)}\\\vdots\\y_{\pi(t-1)}\\ y_{\pi(t)}\end{bmatrix}\right\|^2_2.
\end{equation*}
Now, 
\begin{align*}
    \sum^{t-1}_{k=0}\left\|M_k\begin{bmatrix}y_{\pi(1)}\\y_{\pi(2)}\\\vdots\\y_{\pi(t-1)}\\ y_{\pi(t)}\end{bmatrix}\right\|^2_2&=\left\|\underbrace{\begin{bmatrix}A^{i_{\pi(1)}-i_{\pi(t)}}&A^{i_{\pi(2)}-i_{\pi(t)}}&\cdots&\cdots&I_d\\
    A^{i_{\pi(1)}-i_{\pi(t-1)}}&A^{i_{\pi(2)}-i_{\pi(t-1)}}&\cdots &I_d &0\\
    \vdots& &\iddots & &0\\
    \vdots&\iddots & & &\vdots\\
    I_d &0&\cdots& &0\end{bmatrix}}_{:=\Gamma_1}\begin{bmatrix}y_{\pi(1)}\\y_{\pi(2)}\\\vdots\\\vdots\\y_{\pi(t)}\end{bmatrix}\right\|^2_2.
\end{align*}
Denote $j_1=i_{\pi(1)},j_2=i_{\pi(2)},\ldots,j_t=i_{\pi(t)}$. Then $\Gamma_1$ is a submatrix of $\Gamma_2$, where
\begin{equation*}
    \Gamma_2=\begin{bmatrix}A^{j_1-j_t}&A^{j_1-j_t-1}&\cdots&\cdots&I_d\\
    A^{j_1-j_t-1}&A^{j_1-j_t-2}&\cdots &I_d&0\\
    \vdots& &\iddots & &0\\
    \vdots&I_d & & &\vdots\\
    I_d&0&\cdots& &0\end{bmatrix}.
\end{equation*}
Hence, $\|\Gamma_1\|_2 \leq \|\Gamma_2\|_2$ holds. In addition, $\Gamma_2=\Pi \Gamma_3$, where 
\begin{equation*}
    \Gamma_3 = \begin{bmatrix}
        I_d& & & &\\
        A& I_d & & & \\
        \vdots & &\ddots & \\
        A^{j_1-j_t}&\cdots&\cdots&I_d
    \end{bmatrix}, 
    \quad 
    \Pi =  \begin{bmatrix}
        & & & & I_d\\
         &   & & I_d & \\
         & \iddots& & \\
        I_d& & & 
    \end{bmatrix}.
\end{equation*}
Thus, we have
\begin{align*}
    \|\Gamma_2\|_2 = \|\Gamma_3\|_2  
    \ \leq \ \sup_{x\in[0,1]}\left\|\sum^{j_1-j_t}_{s=0} A^s  e^{\iota 2\pi sx }\right\|_2 
    \ \leq \ \sum^{j_1-j_t}_{s=0}\left\|A^s\right\|_2 
    \ \leq \ \frac{1}{1-\|A\|_2},
\end{align*}
where the first inequality follows from a known result for banded Toeplitz matrices (see \cite[Lemma 5]{jedra2020finitetimeidentificationstablelinear} which in turn uses results from \cite[Chapter 6]{toeplitzbook}).
This means we have shown that 
\[
\|\Gamma_1\|_2\leq \frac{1}{1-\|A\|_2}
\]
which implies 
\begin{align*}
    \bar{\sigma}^2 \leq \frac{1}{1-\|A\|^2_2}\|\Gamma_1\|^2_2\sum^t_{k=1}\|y_{\pi(k)}\|^2_2
    &\leq\frac{1}{1-\|A\|^2_2}\left(\frac{1}{1-\|A\|_2}\right)^2\sum^t_{k=1}\|y_{\pi(k)}\|^2_2 \\
    &\leq \left(\frac{1}{1-\|A\|_2}\right)^3\sum^t_{k=1}\|y_{k}\|^2_2.
\end{align*}
\end{proof}
Recall from \eqref{eq:temp1_1_equiv} and the definition of $g$ that \eqref{eq:augmenting_t_cycle} is equivalent to 
$g\geq \frac12\left(\sum^t_{k=1}\|y_k\|^2_2\right).$ 
Hence, we have 
\begin{align*}
    \prob\left(g\geq \frac12\left(\sum^t_{k=1}\|y_k\|^2_2\right)\right)
    &=\prob\left(\frac{g}{\sigma\bar{\sigma}}\geq \frac1{2\sigma\bar{\sigma}}\sum^t_{k=1}\|y_k\|^2_2\right)\\
    &\leq \prob\left(g'\geq \frac1{2\sigma}\left(\sum^t_{k=1}\|y_k\|^2_2\right)^{\frac12}(1-\|A\|_2)^{\frac32}\right) \tag{since $\frac{g}{\sigma\bar{\sigma}} =: g' \sim \calN(0,1)$} \\
    &\leq \exp\left(-\frac1{8\sigma^2}\sum^t_{k=1}\|y_k\|^2_2(1-\|A\|_2)^3\right).
\end{align*}
Thus, 
\begin{equation*}
    \prob\left(C_t=(i_1,\ldots, i_t) \text{ is augmenting}\right)\leq \expec_{\xi_1,\ldots,\xi_{i_1}}\left[\exp\left(-\frac{(1-\|A\|_2)^3}{8\sigma^2}\sum^t_{k=1}\|y_k\|^2_2\right)\right].
\end{equation*}
Now, note that 
\begin{align*}
    y_t &= A(x_{i_1-1}-x_{i_t-1})+\xi_{i_1}-\xi_{i_t}\\
    &=\left(A^{i_1-1}-A^{i_t-1}\right)\xi_1+\ldots+\left(A^{i_1-i_t}-I\right)\xi_{i_t}+ (A^{i_1-i_t-1} \xi_{i_t + 1} + \cdots + A \xi_{i_1 - 1}) + \xi_{i_1},\\
    y_1 &= A\left(x_{i_2-1}-x_{i_1-1}\right)+\xi_{i_2}-\xi_{i_1}\\
    &=-\left(\left(A^{i_1-1}-A^{i_2-1}\right)\xi_1+\ldots+\left(A^{i_1-i_2}-I\right)\xi_{i_2}+\left(A^{i_1-i_2-1}\xi_{i_2+1}+\ldots+A\xi_{i_1-1}\right)+\xi_{i_1}\right).
\end{align*}
For $k=1,\ldots,t-1$, we can write $y_k$ as follows. Define  
\[
\alpha_k:=\max\{i_k,i_{k+1}\}, \ \beta_k:=\min\{i_k,i_{k+1}\}  \text{ and } 
s_k=\begin{cases}
    1 \text{ if }i_{k+1}>i_k\\
    -1 \text{ if }i_{k+1}<i_k\\
\end{cases}.
\]
Then, $y_k$ can be written as
\begin{align*}
    y_k&=A(x_{i_{k+1}-1}-x_{i_k-1})+\xi_{i_{k+1}}-\xi_{i_k}\\
    &=s_k\left(\left(A^{\alpha_k-1}-A^{\beta_k-1}\right)\xi_1+\ldots+\left(A^{\alpha_k-\beta_k}-I\right)\xi_{\beta_k}+\left(A^{\alpha_k-\beta_k-1}\xi_{\beta_{k+1}}+\ldots+A\xi_{\alpha_k-1}\right)+\xi_{\alpha_k}\right).
\end{align*}
In fact, defining $\alpha_t=i_1$, $\beta_t=i_t$ and $s_t=1$, the above expression holds for $k\in[t]$. Let
\begin{align*}
        &B^\top(\alpha_k,\beta_k) \\ 
        :=
        &\begin{bmatrix}
            (A^*)^{\alpha_k-1}-(A^*)^{\beta_k-1}& \ldots& (A^*)^{\alpha_k-\beta_k}-I&(A^*)^{\alpha_k-\beta_k-1}&\ldots &A^*\, &I &  \underbrace{0,\quad 0,\quad \cdots\quad 0, }_{i_1-\alpha_k \text{ times}}
        \end{bmatrix},
\end{align*}
and $\xi := (
    \xi_1^\top,  
    \xi_2^\top, \cdots, 
    \xi_{i_1}^\top)^\top
 \in \matR^{i_1 d}$. Then, we can write $y_k=s_k B^\top(\alpha_k,\beta_k)\xi$ and 
\[\implies \ 
\sum^t_{k=1}\|y_k\|^2_2=\xi^\top\left(\sum^t_{k=1}B(\alpha_k,\beta_k)B^\top(\alpha_k,\beta_k)\right)\xi = \xi^\top L \xi
\]
for $L$ as defined in Proposition \ref{prop:tcycles}. Hence we have (analogous to the case $t=2$), 
\begin{equation*}
    \expec\left[\exp\left(-\frac1{8\sigma^2} (\xi^\top L\xi) (1-\|A^*\|_2)^3\right)\right]=\operatorname{det}\left(\frac1{4\sigma^2}(1-\|A^*\|_2)^3 L+I_{i_1d}\right)^{-\frac12},
\end{equation*}
since $\xi\sim\calN(0,I_{i_1d})$.

%
%----------------------------
% Proof of determinant bound
%-----------------------------
%
\subsection{Proof of Proposition \ref{prop:det_bound}} \label{subsec:proof_prop_det_bound}
To prove Proposition~\ref{prop:det_bound}, we need to control the spectrum of the matrix $ L$. To this end, we begin by expressing $ L $ in a more convenient form. Observe that
\[
B(\alpha_k, \beta_k) = P_{\alpha_k} - P_{\beta_k},
\]
where each $ P_r \in \mathbb{R}^{i_1 d \times d} $ is a block matrix (with $i_1$  vertically stacked blocks of size $d\times d$) defined as
\[
P_r =
\begin{bmatrix}
    (A^*)^{r-1} \\
    (A^*)^{r-2} \\
    \vdots \\
    A^* \\
    I_d \\
    0 \\
    \vdots \\
    0
\end{bmatrix} \in \matR^{(i_1 d) \times d},
\]
with the final $ i_1 - r$ blocks consisting of zero matrices. Using this representation, we can write $L = M M^\top,$
where
\[
M :=
\begin{bmatrix}
    P_{\alpha_1} - P_{\beta_1} \ \ 
    P_{\alpha_2} - P_{\beta_2} \ \ 
    \cdots \ \ 
    P_{\alpha_t} - P_{\beta_t}
\end{bmatrix}.
\]
%
%\HT{Note that this is how $M$ is defined in the notes as well, what was written here earlier was different from the notes} 

%%%
Assume, without loss of generality (see Remark \ref{rem:assumption_cycle}), that the indices in the cycle $C_t = (i_1, i_2, \ldots, i_t)$ are ordered such that $i_1 > i_2 > \cdots > i_t$. In this case, we can write
\[
M=P(D\otimes I_d),
\]
where $P:=\begin{bmatrix}
    P_{i_1}&P_{i_2}&\ldots &P_{i_t}
\end{bmatrix}$ and $D$ is the incidence matrix of the directed cycle $\overrightarrow{C_t}$, which corresponds to $C_t$ with the orientation $i_1\rightarrow i_2\rightarrow \ldots \rightarrow i_t$ and $i_1\rightarrow i_t$. With this representation, we obtain
\begin{align*}
    L = M M^\top 
      = P \left((D D^\top) \otimes I_d\right) P^\top 
      = P(\underbrace{L_{C_t} \otimes I_d}_{=:\tilde{L}_{C_t}}) P^\top,
\end{align*}
where $L_{C_t}$ denotes the graph Laplacian of the cycle $C_t$.
This decomposition enables us to view $L$ as a multiplicative perturbation of $\tilde{L}_{C_t}$, whose spectrum is explicitly known (it coincides with that of $L_{C_t}$, up to multiplicity). 
Indeed, it is well known (see e.g., \cite{spielmanSGT}) that the eigenvalues of $C_t$ correspond to the (unordered) values 
\begin{align*}
    2\left(1 - \cos\left(\frac{2\pi (t-k)}{t}\right)\right); \quad k=1,\dots,t.
\end{align*}
% % The following lemma provides bounds on the eigenvalues of $L$, in terms of those of $\tilde{L}_{C_t}$ and $PP^\top$. 
%
Denote 
\[s:=\operatorname{rank}(L),\]
that is, $L$ has $s$ non-zero eigenvalues (later, we will determine $s$). The following general lemma helps us to lower-bound the quantity
\[
\log\operatorname{det}\left(\frac{(1-\|A^*\|_2)^3}{4\sigma^2} L + I_{i_1 d}\right).
\]

%the logarithm of the quantity in \eqref{eq:det_to_bound_gen}. 
%
\begin{lemma}\label{lem:bound_logarithm_quant}
    Let $Z\in \matR^{q\times q}$ be symmetric and p.s.d with rank $p\leq q$, and let $\lambda_1(Z),\ldots,\lambda_p(Z)$ be its non-zero eigenvalues. %\HT{I added ``symmetric and p.s.d'', its correct?}.
    Then, for any $\gamma>0$,
    \begin{equation*}
        \log{\prod^p_{k=1}\left(\gamma \lambda_k(Z)+1\right)}\geq p\log{\left(\gamma \left(\prod^p_{k=1}\lambda_k(Z)\right)^{\frac1p}+1\right)}.
    \end{equation*}
\end{lemma}
To prove this, we need the following elementary classic result, whose proof can be found in Appendix \ref{app:proof_sup_geo_means}.
\begin{lemma}[Super-additivity of geometric means]\label{lem:superadd_geo_means}
    Let $(a_k)_{1\leq k\leq n},(b_k)_{1\leq k\leq n}$ be two sequences of non-negative real numbers, then 
    \[
    \left(\prod^n_{k=1}a_k\right)^{1/n}+\left(\prod^n_{k=1}b_k\right)^{1/n}\leq \left(\prod^n_{k=1}(a_k+b_k)\right)^{1/n}.
    \]
\end{lemma}
Lemma~\ref{lem:bound_logarithm_quant} follows directly from Lemma~\ref{lem:superadd_geo_means}, 
by considering the sequences $(a_k)_{1\leq k\leq p}$ and $(b_k)_{1\leq k\leq p}$ defined as
\[
    a_k = \gamma \lambda_k(Z), 
    \quad 
    b_k = 1, \quad \text{  for }k\in \{1,\ldots,p\},
\]
and then applying the logarithm.
Applying Lemma \ref{lem:bound_logarithm_quant} with $Z=L$ (so $p=s$) and $\gamma=\frac{(1-\|A^*\|_2)^3}{4\sigma^2}$, noting  that $\lambda_k(L)\geq 0$, for all $k\in [i_1d]$, we get 
\begin{equation} \label{eq:logdet_interm_bd_1}
     \log\det\left(\frac{(1-\|A^*\|_2)^3}{4\sigma^2}L+I_{i_1d}\right)\geq s\log{\left(\frac{(1-\|A^*\|_2)^3}{4\sigma^2} \left(\prod^s_{k=1}\lambda_k(L)\right)^{\frac1s}+1\right)}.
\end{equation}
To complete the proof of the lower bound, two ingredients are required. 
First, we must determine the value of $s$. 
Second, we need to obtain a lower bound for
\[
   \det\nolimits^{*}(L):=\!\prod_{k=1}^s \lambda_k(L),
\]
known as the \emph{pseudo determinant} of $L$.

The following lemma gives the value of $s$ and the rank of $P$ (which will be needed later). 
\begin{lemma}\label{lem:rank_P_L}
   Let $P$ and $L$ as defined above, and $s=rank(L)$. Then $s=(t-1)d$ and $\rank(P)=td$.
\end{lemma}
\begin{proof}
The claim $\rank(P)=td$ follows directly from the structure of the matrices $P_{i_1},\ldots,P_{i_t}$. 
Indeed, for each $k \in [t]$, the matrix $P_{i_k}$ has column rank $d$, 
since it is a tall matrix containing $I_d$ as one of its blocks. 
Moreover, the $I_d$ blocks corresponding to different $P_{i_k}$ are disjoint.

To see $s = (t-1)d$, recall $L = M M^\top$ where $M = P(D \otimes I_d)$. Since $P \in \matR^{(i_1 d) \times (td)}$ is full column-rank (as $i_1 \geq t$) and $D \otimes I_d \in \matR^{(td) \times (t-1) d}$ is also full column-rank, hence it readily follows that $s = (t-1)d$. 
%\HT{I added details about $s = (t-1)d$ claim here. The other claim is clear from the explanation provided.}
%\EA{ Idk if the proof it's enough to convince the reader. As we discussed, this should be simple. If not enough, one can be more formal, and it should be similar to a proof I wrote for the case $A^*=\theta I_d$. It is in the old material, if needed...}
\end{proof}
From Lemma \ref{lem:rank_P_L} and \eqref{eq:logdet_interm_bd_1}, it follows that
\begin{equation}\label{eq:det_bound_2}
     \log\det\left(\frac{(1-\|A^*\|_2)^3}{4\sigma^2}L+I_{i_1d}\right)\geq (t-1)d\log{\left(\frac{(1-\|A^*\|_2)^3}{4\sigma^2} \left(\prod^{(t-1)d}_{k=1}\lambda_k(L)\right)^{\frac1{(t-1)d}}+1\right)}.
\end{equation}
From the previous lemma, it also follows that $\operatorname{rank}(P^\top P)=td$, which will be used in the lower bound for $\operatorname{det}^*(L)$.
\paragraph{Bound for the pseudo determinant of $L$.}%$\operatorname{det}^*(L)$.}
To obtain a bound on $\det\nolimits^{*}(L)$, we require a sequence of auxiliary lemmas. 
The first, stated in a more general form, shows that the pseudo determinant of $L = P(L_{C_t}\otimes I_d )P^\top$ 
can be factorized into the product of two terms, each depending exclusively on $L_{C_t}$ or $P$, respectively. The proof can be found in Appendix \ref{app:proof_pseudodet_gen_PLPt}.
\begin{lemma}\label{lem:pseudodet_gen_PLPt}
    Let $Z\in \matR^{q\times q}$ be a symmetric p.s.d matrix of rank $p\leq q$, and let $W\in \matR^{q'\times q}$, with $q'\geq q$. Assume that $W$ has rank $q$. Then, 
    \begin{equation*}
        \operatorname{det}^*(W Z W^\top)=\operatorname{det}^*(Z)\operatorname{det}\left((W U_p)^\top(W U_p)\right),
    \end{equation*}
    where $U_p \in \matR^{q \times p}$ is the matrix whose columns are the eigenvectors of $Z$ associated to its non-zero eigenvalues. 
\end{lemma}
Given Lemma \ref{lem:rank_P_L}, the assumptions of Lemma \ref{lem:pseudodet_gen_PLPt} are satisfied with $Z=(L_{C_t}\otimes I_d)$, $W=P$, $p=(t-1)d$, $q=td$, $q'=i_1d$. Consequently, by these lemmas, we get 
\begin{align*}
    \operatorname{det}^*\left(P(L_{C_t}\otimes I_d) P^\top\right)&=\operatorname{det}^*(L_{C_t}\otimes I_d)\operatorname{det}\left((PU_{(t-1)d})^\top(P U_{(t-1)d})\right)\\
    &={\operatorname{det}^*(L_{C_t})}^d\operatorname{det}\left((PU_{(t-1)d})^\top(P U_{(t-1)d})\right)
\end{align*}
By the Kirchoff's matrix tree theorem \cite[Lemma 13.2.4]{godsil2001algebraic}, we know that 
\[
\frac{1}{t}\operatorname{det}^*(L_{C_t})=|\{\text{spanning trees in }C_t\}|.
\]
Since the number of spanning trees in the $t$-cycle graph is equal to $t$, we obtain
\begin{align*}
   \operatorname{det}^*(L_{C_t})=t^2.
\end{align*}
From this, we obtain, 
\begin{equation} \label{eq:pseudo_det_exp_1}
     \operatorname{det}^* (L) = \operatorname{det}^*\left(P(L_{C_t}\otimes I_d) P^\top\right)=t^{2d}\operatorname{det}\left((PU_{(t-1)d})^\top(P U_{(t-1)d})\right).
\end{equation}
The following lemma helps us bound the right-hand side of the previous expression. Its proof is deferred to Appendix \ref{app:proof_bound_PUtPU}.
\begin{lemma}\label{lem:bound_PUtPU}
    We have 
    \begin{equation*}
        \det\left((PU_{(t-1)d})^\top(PU_{(t-1)d})\right)\geq \frac{\det\left(P^\top P\right)}{\prod^d_{k=1}\lambda_k\left(P^\top P\right)}.
    \end{equation*}
\end{lemma}
By the previous lemma, in order to obtain a lower bound on $\operatorname{det}^*(PL P^\top)$, there are two ingredients left: a lower bound on $\det(P^\top P)$, and an upper bound on $\prod^d_{k=1}\lambda_k(P^\top P)$. For the latter, we will use that
\[
\prod^d_{k=1}\lambda_k(P^\top P)\leq {\lambda_1(P^\top P)}^d,
\]
for which an upper bound on $\lambda_1(P^\top P)$ suffices.
\paragraph{Determinant of $P^\top P$.} The following lemma, whose proof is given in Appendix \ref{app:proof_lem:det_PtP}, provides a formula of the explicit value of $\det\left(P^\top P\right)$, which could be of independent interest. The proof relies on decomposing the matrix $P^\top P$ into the product of square block lower triangular matrices and a Gram matrix whose determinant is straightforward to compute.
%\HT{I have updated the statement of the lemma below, please check!}
% 
%
%
\begin{lemma}\label{lem:det_PtP}
   %Assume $A^*$ to be symmetric. Then 
   For any cycle $C_t=(i_1,\ldots,i_t)$, where $t\geq 2$, it holds for $P=\begin{bmatrix}
        P_{i_1}&P_{i_2}&\ldots&P_{i_t}
    \end{bmatrix}$ that
    \begin{equation*}
        \det\left(P^\top P\right)
        = \left(\prod_{k=1}^{t-1}\det\Big(\sum^{i_k-i_{k-1}-1}_{l=0}((A^{*})^{l})^\top (A^*)^l\Big) \right)\left(\det\Big(\sum^{i_t-1}_{l=0} ((A^*)^{l})^\top (A^*)^l \Big) \right) \geq 1.
         %   \left(\prod^{t-1}_{k=1}\prod^d_{j=1}\,\sum^{i_k-i_{k-%1}-1}_{l=0}{\lambda_{j}
         %(A^*)}^{2l}\right)\left(\prod^d_{j=1}\sum^{i_t-1}_{l=0}%{\lambda_{j}(A^*)}^{2l}\right)  .  
    \end{equation*}
\end{lemma}
From the previous lemma, we obtain $\det(P^\top P) \geq 1.$ Although simple, this bound is already nontrivial from the definition of $P$. 
Moreover, it allows us to identify the cases of equality: indeed, the bound is tight 
when $\|A^*\|_2 = 0$ or when the cycle is $C_t = (t, t-1, \ldots, 2, 1)$.

\paragraph{Upper bound on $\lambda_1\left(P^\top P\right)$.} We will use a generalization of Gershogorin's theorem for block matrices (see \cite[Theorem 1.13.1]{Tretter2008}) to bound the largest eigenvalue of $P^\top P$. The proof of the next result is deferred to Appendix \ref{app:proof_gerschgorin_PtP}.
\begin{lemma}\label{lem:gerschgorin_PtP}
    For $P$ defined as above, it holds 
    \[
    \lambda_1\left(P^\top P\right)\leq \frac1{(1-\|A^*\|_2)^2}.
    \]
\end{lemma}
\paragraph{Putting it together.} 
From \eqref{eq:pseudo_det_exp_1}, and Lemmas \ref{lem:bound_PUtPU}, \ref{lem:det_PtP} and \ref{lem:gerschgorin_PtP} we deduce %\HT{added $t^{2d}$ factor below}
\begin{equation*}
    {\det}^*(L)\geq t^{2d}(1-\|A^*\|_2)^{2d}, 
\end{equation*}
which together with \eqref{eq:det_bound_2} gives
\begin{align}
     \log\det\left(\frac{(1-\|A^*\|_2)^3}{4\sigma^2}L+I_{i_1d}\right)^{-\frac12}&\leq -(t-1)\frac d2\log{\left(\frac{(1-\|A^*\|_2)^3}{4\sigma^2} \left(t^{2d}(1-\|A^*\|_2)^{2d}\right)^{\frac1{(t-1)d}}+1\right)}\nonumber\\
     %\label{eq:det_bound_3} 
     &\leq -(t-1)\frac d2\log{\left(\frac{(1-\|A^*\|_2)^5}{4\sigma^2} +1\right)}.\nonumber
\end{align}
In the last line we used that $t^{\frac{2}{t-1}}\geq 1$ for $t\geq 2$, and that $(1-\|A^*\|_2)^{\frac2{t-1}}\geq (1-\|A^*\|_2)^2$, for $t\geq2$. From this, Proposition \ref{prop:det_bound} follows. 
\begin{remark}[About the assumption $i_1>i_2>\ldots>i_t$]\label{rem:assumption_cycle}
    In the proof of Proposition~\ref{prop:det_bound}, we assumed \( i_1 > i_2 > \ldots > i_t \) for convenience. More generally, for any \( t \)-cycle \( C_t = (i_1, i_2, \ldots, i_t) \) with \( i_1 \) being the largest element, we always have \( (\alpha_1, \beta_1) = (i_1, i_2) \) and \( (\alpha_t, \beta_t) = (i_1, i_t) \). This allows us to consistently orient the edges \( i_1 \rightarrow i_2 \) and \( i_1 \rightarrow i_t \). The orientation of the remaining edges—namely, \( \{i_2, i_3\}, \{i_3, i_4\}, \ldots, \{i_{t-1}, i_t\} \)—depends on the relative ordering of \( i_2, i_3, \ldots, i_t \). For the purposes of the analysis, this implies that we can express \( M = P(D \Pi \otimes I_d) \), where \( \Pi \) is a permutation matrix. Clearly, the matrix \( L = M M^\top \) remains unchanged compared to the case where \( i_1 > i_2 > \ldots > i_t \).
\end{remark}
%
%
%----------------------------------
% Algorithms for solving the MLE
%----------------------------------
\section{Algorithms for solving the MLE}\label{sec:algos} 
%
%In this section, we propose several algorithms to efficiently solve the permutation recovery problem from \textbf{CVAR} data. We consider two main strategies: one based on the linear assignment estimator, $\hat{\Pi}_{\operatorname{LA}}$, defined in \eqref{eq:LA_pi} and analyzed in Section \ref{sec:analysis}, and another based on maximum likelihood estimation (MLE). The description of the linear assignment estimator is straightforward, and we defer implementation details to Section~\ref{sec:experiments}. We now describe the MLE-based strategy.
%
Given the hardness of solving the joint optimization problem \eqref{eq:MLE_2_sigma_known_mat}, we now describe an alternating minimization based heuristic for solving \eqref{eq:MLE_2_sigma_known_mat} consisting of the following two main steps, which will be iteratively applied.
\begin{itemize}
    \item \textbf{Step 1: Estimation of $A^*$ for fixed $\Pi$}. For a fixed $\Pi$, we can estimate $A^*$ by solving \eqref{eq:MLE_A_fixed_pi_known_sigma}. This has a closed form solution outlined in Lemma \ref{lem:MLE_A_pi_fixed_sigma_known}.
    
    \item \textbf{Step 2: Estimation of $\Pi^*$ for fixed $A$}. For a fixed \( A \), we consider the problem \eqref{eq:MLE_pi_fixed_A_known_sigma}. Since the objective function is convex, we propose several relaxations of the permutation constraints, resulting in convex optimization problems. %In addition, we consider the linear assignment estimator discussed in previous sections.

\end{itemize}
We start by introducing algorithms for estimating $\Pi^*$, for a fixed $A$, in Section \ref{sec:algos_pi_fixed_A}. Later, in Section \ref{sec:algos_iter} we introduce an iterative algorithm based on alternating optimization. Recall that we assume $\sigma$ is known.
%
%
%-------------------------------------------------------
% Relaxed MLE strategy for estimating $\Pi$ given $A$
%-------------------------------------------------------
\subsection{Relaxed MLE strategy for estimating \texorpdfstring{$\Pi^*$}{Pi} given \texorpdfstring{$A$}{A}}\label{sec:algos_pi_fixed_A}
Recall the estimator for estimating $\Pi^*$, for a fixed $A$, introduced in \eqref{eq:MLE_pi_fixed_A_known_sigma} -- the MLE
\begin{equation*}
    \est{\Pi}_{\operatorname{MLE}}(A) \in \argmin{\Pi\in\calP_T}\,\|X^\#\Pi-AX^\#\Pi S-(X-AXS)\|^2_F,
\end{equation*}
% %
% and, the linear assignment estimator
% %
% \begin{equation*}
% %
%     \hat{\Pi}_{\operatorname{LA}}\in \argmax{\Pi\in\calP_T}\,\langle \hX\Pi,X\rangle_F.%=\langle \Pi,\underbrace{(\hX)^\top X}_{=:W}\rangle_F.
% %
% \end{equation*}
% %
where, $S$ denotes the shift operator introduced in Section~\ref{sec:MLE_pi}. %Note that the linear assignment estimator, $\hat{\Pi}_{\operatorname{LA}}$, can be computed, for example, using the Hungarian algorithm, which in this case has a complexity of~$\calO(T^3)$. In contrast, it is not clear whether $\hat{\Pi}_{\operatorname{MLE}}$ can be found efficiently in the general case. For this reason, we study a two-step strategy: (1) solving a \emph{relaxed MLE} problem, and (2) rounding the relaxed solution to a valid permutation. 
Since a general quadratic program with permutaion constraints is hard in the worse case, it is not clear whether $\hat{\Pi}_{\operatorname{MLE}}$ can be found efficiently. For this reason, we study a two-step strategy: (1) first solving a \emph{relaxed MLE} problem, and (2) then rounding the relaxed solution to a valid permutation. 
\paragraph{Relaxed MLE.}The objective in \eqref{eq:MLE_pi_fixed_A_known_sigma} is convex (the square norm of a linear function in $\Pi$), but the set of constraints is discrete. We will consider a convex set $\calK\subseteq \matR^{T\times T}$, containing the set of permutations matrices, to obtain the relaxed convex optimization problem 
\begin{equation}\label{eq:relaxed_general}
    \min_{\Pi\in \calK}\,\underbrace{\|X^\#\Pi-AX^\#\Pi S-(X-AXS)\|^2_F}_{=:f(\Pi;X,\hX,A,S)}.
\end{equation}
In what follows, we will consider specific relaxations induced by particular choices of ~$\calK$. Although these relaxations can be solved using general-purpose convex optimization methods (e.g., interior point methods), for the sake of efficiency we will implement gradient-based algorithms in the numerical experiments (see Section \ref{sec:experiments} for details).%propose  gradient-based algorithms, as described below.
\begin{enumerate}
%
    % \item \textbf{Orthogonal matrices}. Here, we take $\calK=\{Z\in \matR^{T\times T}:Z^\top Z=I_T\}$, the set of orthogonal matrices, which is a non-convex relaxation. This relaxation is inspired by Umeyama algorithm for graph matching \cite{umeyama1991least}. In the case of graph matching, this relaxation leads to a closed form spectral solution. \EA{What do we obtain in our case?TRS?}
    % %
    \item \textbf{Hyperplane}. This corresponds to the choice (where $\mathbbm{1}$ denotes the all ones vector) 
    $$\calK=\{Z\in \matR^{T\times T}:\mathbbm{1}^\top Z\mathbbm{1}=T\},$$ which is  a hyperplane constraint in the space $\matR^{T^2}$ (with the obvious identification of a matrix with a vector). The main motivation for this relaxation comes from its success in the graph matching problem, as studied in \cite{fan2022spectralI} (with an additional regularization term) where performance guarantees were obtained under specific planted matching models. 
    %Similar to \cite{fan2022spectralI}, we use a gradient descent algorithm (initialized with the all-ones matrix and a constant learning rate~$\gamma > 0$) to solve an unconstrained version of~\eqref{eq:relaxed_general}, where the hyperplane constraint defining~$\calK$ is incorporated as a penalty term.%, we propose using a gradient descent algorithm initialized with the all-ones matrix and a constant learning rate~$\gamma > 0$.
    %
    \item \textbf{Simplex}. A tighter relaxation than the hyperplane constraint is given by the simplex 
$$\mathcal{K} = \{Z \in \mathbb{R}^{T \times T} : \mathbbm{1}^\top Z \mathbbm{1} = T,\; Z \geq 0\}.$$ The addition of positivity constraints, in comparison to the simplex relaxation, has proven beneficial in the context of graph matching, as demonstrated recently in \cite{ArayaTyagi_GM_fods}. There, this relaxation was shown to outperform the hyperplane formulation experimentally. %Although it does not admit a closed-form solution, this simplex relaxation can be solved efficiently using an iterative method based on (entropic) mirror descent, which leads to a multiplicative weights update algorithm (see, for example, \cite[Sec. 9]{beck2017first}).
\item \textbf{Birkhoff polytope.} The tightest convex relaxation of \eqref{eq:MLE_pi_fixed_A_known_sigma} is given by the Birkhoff polytope (the set of doubly stochastic matrices) $$\calK=\{Z\in\matR^{T\times T}: \mathbbm{1}^\top Z=\mathbbm{1}^\top, Z\mathbbm{1}=\mathbbm{1},Z\geq 0\}.$$ Indeed, by the Birkhoff–von Neumann theorem, the Birkhoff polytope is precisely the convex hull of the set of permutation matrices. %To solve this problem, one can use projected gradient descent with projection onto the Birkhoff polytope at each step. In practice, this projection is often approximated using the Sinkhorn-Knopp algorithm, which performs iterative row and column normalization. Alternatively, the alternating direction method of multipliers (ADMM) can be used to handle the constraints more efficiently. %To solve it, we can use a strategy based, for instance, on the ADMM algorithm \cite{boyd2011distributed}.
\end{enumerate}
%
%\EA{Will it make sense to consider, for example, a projected gradient descent, to directly tackle \eqref{eq:MLE_pi_fixed_A_known_sigma}? Do we need a 'warm start' or random initilization will work as well?}.
%
%

% A direct calculation yields the following matrix expression for the gradient of $f$, the objective function in the relaxed MLE, which is required for the gradient-based methods,
% %
% \begin{align}
%     \nabla f(\Pi)=2&\Big({\hX}^\top\hX\Pi-{\hX}^\top A^\top\hX\Pi S^\top-{\hX}^\top A \hX\Pi S+{\hX}^\top A^\top A \hX\Pi S^\top \nonumber\\ \label{eq:grad_f}
%     &-{\hX}^\top(X-AXS)+{\hX}^\top A^\top(X-AXS)S^\top\Big).
% \end{align}
%
%
\paragraph{Rounding procedure.}The solutions of the relaxed problem~\eqref{eq:relaxed_general} are, in general, not guaranteed to be permutation matrices. Therefore, an additional rounding step is required. This can be achieved by solving a linear assignment problem using the relaxed solution (denoted by~$\est{\Pi}_{\operatorname{rel}}$) as the cost matrix; that is, we solve
\begin{equation}\label{eq:LA_rounding}
    \max_{\Pi\in\calP_T}\langle \Pi,\est{\Pi}_{\operatorname{rel}}\rangle_F.
\end{equation}
As previously mentioned, the linear assignment problem can, in general, be solved by the Hungarian algorithm with cubic complexity.
\begin{algorithm}
\caption{Relaxed MLE + LA rounding (\texttt{RelaxMLE-Round})}
\label{alg:relaxMLE_round}
\begin{algorithmic}[1]
\Require Time series matrices $X, X^\# \in \matR^{d \times T}$, system matrix $A\in \matR^{d\times d}$, %, max iterations $K\in\matN$, learning rate $\gamma>0$, initial point $\est{\Pi}^{(0)}_{\operatorname{rel}}$, 
a convex set $\calK\supseteq \calP_T$
\State Set $S=\begin{bmatrix}
    0&I_{T-1}\\
    0&0
\end{bmatrix}$.
\State Solve the relaxed MLE problem in \eqref{eq:relaxed_general} \[\est{\Pi}_{\operatorname{rel}}=\argmin{\Pi\in\calK}f(\Pi;X,\hX,A,S).\]
\State Round with linear assignment 
\[
\est{\Pi}=\argmax{\Pi\in\calP_T}\langle \Pi, \est{\Pi}_{\operatorname{rel}}\rangle_F.
\]
% \For{$k = 1$ to $K-1$}
%     %\State $\est{\Pi}^{(k)}_{\operatorname{rel}}=\est{\Pi}^{(k-1)}_{\operatorname{rel}}-\gamma \nabla f(\est{\Pi}^{(k-1)}_{\operatorname{rel}};X,\hX,A,S)$, \, \, where $\nabla f$ is defined in \eqref{eq:grad_f}
% \EndFor
\State \Return $\est{\Pi}$%, a solution of \eqref{eq:LA_rounding} with cost $\est{\Pi}^{(K)}_{\operatorname{rel}}$
\end{algorithmic}
\end{algorithm}
\begin{remark}[Other relaxations]
        We choose to present Algorithm~\ref{alg:relaxMLE_round} in a general form, which allows, in principle, the use of other convex relaxations. Our implementation in Section~\ref{sec:experiments} will be based on the choices of~$\calK$ discussed above, which are motivated by their success in other quadratic optimization problems over the set of permutations, as well as by the availability of efficient algorithms based on gradient descent, mirror descent, and the alternating direction method of multipliers (ADMM). It is worth noting that in related problems, such as graph matching, non-convex relaxations have also been considered. For example, in the classic work~\cite{umeyama1991least}, relaxing the graph matching problem to the set of orthogonal matrices yields a closed-form solution based on spectral information. In the case of~\eqref{eq:relaxed_general}, however, it is not clear that a closed-form solution exists when~$\calK$ is taken to be the set of orthogonal matrices.

\end{remark}

\subsection{Iterative algorithm for estimating \texorpdfstring{$\Pi^*$}{Pi*}}\label{sec:algos_iter}
To recover $\Pi^*$, we propose an iterative algorithm that alternates between {\bf Steps 1 and 2} discussed at the beginning of this section. In Algorithm~\ref{alg:TS_matching_alternating}, we present the proposed procedure in the general case, where any of the strategies based on convex relaxation for estimating~$\Pi$ given~$A$, as discussed in Section~\ref{sec:algos_pi_fixed_A}, can be used. We choose to write it in this abstract form, using \texttt{RelaxMLE-Round}, which serves as a subroutine for estimating~$\Pi$ given~$A$.   % \EA{I wrote this way, but we should stress that each $\texttt{PiGivenA}$ receives a common input (mainly the data matrices) and maybe some algorithms have specific parameters (e.g., learning rate, etc).}
\begin{algorithm}
\caption{Alternating minimization method for matching VAR time series}
\label{alg:TS_matching_alternating}
\begin{algorithmic}[1]
\Require Time series matrices $X, X^\# \in \matR^{d \times T}$, noise parameter $\sigma>0$, initial estimate $\Pi^{(0)}\in \calP_T$, max iterations $K$, a convex set $\calK\supseteq \calP_T$. %algorithm $\texttt{PiGivenA}$ to recover $\Pi$ given $A$.
\State Set $S=\begin{bmatrix}
    0&I_{T-1}\\
    0&0
\end{bmatrix}$.
\For{$k = 1$ to $K$}
    \State \textbf{A-update:} Set 
    \[
    \begin{aligned}A^{(k)}=&\left[X(XS)^\top+\frac1{\sigma^2}\left(X^\# \Pi^{(k-1)}-X\right)\left(X^\# \Pi^{(k-1)} S-XS\right)^\top\right] \\&\times \left[(XS)(XS)^\top+\frac1{\sigma^2}(X^\# \Pi^{(k-1)} S-XS)(X^\# \Pi^{(k-1)} S-XS)^\top\right]^\dagger
    \end{aligned}
    \]
    \State \textbf{$\Pi$-update:} Use sub-routine $\texttt{RelaxMLE-Round}$ to obtain
    \[
    \Pi^{(k)} := \texttt{RelaxMLE-Round}(X,X^\#,A^{(k)},\calK)
    \]
\EndFor
\State \Return $\Pi^{(K)}$
\end{algorithmic}
\end{algorithm}

\paragraph{Initialization.}Algorithm~\ref{alg:TS_matching_alternating} requires an initial estimate $\Pi^{(0)}$ of the optimal matching. Several strategies can be used for this initialization. One option is to set the initial permutation to the linear assignment estimator, i.e., $\Pi^{(0)} = \widehat{\Pi}_{\operatorname{LA}}$. In this case, Algorithm~\ref{alg:TS_matching_alternating} can be seen as an iterative refinement of the linear assignment solution. This raises the natural question of how much improvement the algorithm provides over linear assignment—a question we explore experimentally in Section~\ref{sec:experiments}. Another approach is to initialize randomly, for example by drawing $\Pi^{(0)}$ uniformly from the set $\mathcal{P}_T$, which we also investigate empirically. A central question is the extent to which the initial estimate of the permutation influences the outcome.

\begin{remark}[Another strategy: estimate $A^*$ first]\label{rem:estimate_A_first}
Notice that Algorithm~\ref{alg:TS_matching_alternating} does not require prior knowledge of the system matrix $A^*$, as it is updated iteratively (specifically in line~4). If enough data is available—that is, sufficiently long time series—an alternative strategy is to first estimate the system matrix using only the time series  $X$, and then address the problem of estimating the permutation $\Pi$, without further updating the estimate of $A^*$. The intuition is that the estimate of $A^*$ would not change significantly, since the amount of information about $A^*$ contained in $(X, \widehat{X})$ should be asymptotically similar to that contained in $X$ alone. On the other hand, an interesting question is wether one can consistently estimate the matching (or achieve non-trivial recovery), even if the system matrix cannot be consistently estimated. 
\end{remark}

\begin{remark}[Complexity]\label{rem:complexity}
The computational complexity of the $A$-update step in Algorithm~\ref{alg:TS_matching_alternating} is dominated by matrix multiplications involving matrices of sizes \(d\times d\), \(d\times T\), and \(T\times T\). In the worst case, this yields a complexity of $
O(\max\{d, T\}^{\omega})$, where $\omega \le 3$ (see e.g. \cite{mat_mult}) denotes the matrix multiplication exponent. The complexity of the $\Pi$-update step depends on the specific optimization method employed. For instance, when using a gradient-based method for the hyperplane relaxation with $\log T$ iterations, the cost is approximately $O(\max\{d, T\}^{\omega} \log T)$, as the most expensive operation---gradient computation---also reduces to matrix multiplications. The rounding step incurs an additional $O(T^3)$ cost in the worst case when solved via the Hungarian algorithm \cite{Kuhn1955}. Therefore, the overall computational complexity of the alternating scheme is $
O(\max\{d, T\}^{\omega} K \log T)$, 
where $K$ denotes the number of outer iterations. This complexity can be reduced in practice under structural assumptions. For example, if some of the matrices involved are sparse, matrix multiplications become more efficient. Alternatively, replacing the linear assignment rounding step with a greedy method reduces the rounding cost to $O(T^2)$, see \cite[Algorithm 1]{ArayaTyagi_GM_fods}, for example.

% The complexity of the $A$-update step in Algorithm \ref{alg:TS_matching_alternating} is mainly given by matrix multiplication of $d\times d$, $d\times T$ and $T\times T$ matrices. In the worst case, its complexity is $O(\max\{d,T\}^\omega)$, where $\omega\leq 3$, is the constant bounding the complexity in matrix multiplication. Since the $\Pi$- updated step depends on the algorithm used, we will give a rough estimate. For instance, using a gradient-based method for $\log{T}$ steps to solve the hyperplane relaxation will take $O(\max\{d,T\}^\omega\log{T})$, since the most expensive step (computing gradients) reduces to matrix multiplication. The rounding step is $O(T^3)$ in the worst case, using the Hungarian algorithm. With this the overall complexity is $O(\max{d,T}^{\omega}K\log T)$. This can be improved for instance, if some of the matrices involved are sparse, so that matrix multiplication is more efficient. Another way to improved is to replace the rounding by linear assignment by a greedy method, in which case the rounding is $O(T^2)$.
\end{remark}
\section{Numerical experiments}\label{sec:experiments}
%
%\EA{I'll just put figures here with some very preliminary text}

%
% \paragraph{$A^*$ known.} First, we evaluate the \texttt{RelaxMLE-Round} subroutine for a given matrix $A^*$. We fix take matrices of the form $A^*=\theta I_d$, and sample time series with distribution \cvar\ for different values of $\sigma$. To better visualize the influence of $\theta$ and $\sigma$, we fix one of them and show the influence of the other in the estimation of $\Pi^*$. In Figure \ref{fig:recovery_vs_scale} we show the influence of the scale in the estimation. We consider two cases, one with $d=5\,, T=50$ and the other with $d=50\,,T=5$ (the idea is to capture cases where $d$ is small compared to $T$ and vice versa). Interestingly, when $d=5$ and $T=50$, the scale does not seem to affect the estimation. In other word, at least for the chosen algorithm, the problem does not seem to be harder with larger values of $\theta$. On the contrary, in the case $d=5$ and $T=50$ larger scales produce worse estimates for the same $\sigma$. Additionally, it worsen the larger the $\sigma$.  
%
In this section, we empirically test the recovery algorithms discussed in Section \ref{sec:algos}, and the linear assignment estimator analyzed in Section \ref{sec:analysis}. We focus on synthetic data generated under the \textbf{CVAR} model, introduced in Section \ref{sec:mle_la_formulate}. In Section \ref{sec:details}, we provide some details about the implementation of the relaxed-MLE strategy (Algorithm \ref{alg:relaxMLE_round}) for different choices of $\calK$. In Section \ref{sec:exp_rel-MLE-knownA}, we test the relaxed MLE algorithms under the assumption that $A^*$ is known. In Section \ref{sec:algos-unknownA}, we assume that $A^*$ is unknown and evaluate the alternating optimization approach of Algorithm \ref{alg:TS_matching_alternating}, as well as the linear assignment (LA) estimator. \rev{The python code for the experiments can be found at \url{https://github.com/ErnestoArayaV/Matching-VAR-time-series}.}

%\subsection{Relaxed MLE}\label{sec:exp_rel-MLE}
%
%\paragraph{Algorithmic implementation details.} 
\subsection{Algorithmic implementation details}\label{sec:details}
To solve the relaxed MLE problem \eqref{eq:relaxed_general}, we use different approaches depending on the convex set $\calK$ considered. Although this problem could be solved with general purpose convex optimization algorithms, for the sake of efficiency, we focus on first-order methods as described below. 
\begin{itemize}
    \item \textbf{Hyperplane. } In the case $\calK=\{Z\in \matR^{T\times T}:\mathbbm{1}^\top Z\mathbbm{1}=T\}$, we will use a \emph{projected gradient descent} (PGD) strategy to optimize. More specifically, we consider a learning rate $\gamma_k>0$, and an initial vector ${\est{\Pi}_{\operatorname{rel}}}^{(0)}=\frac 1T J_T$, where $J_T$ is the $T\times T$ all-ones matrix. This choice can be considered agnostic, since $J_T$ is at the same distance with respect to all the permutation matrices. Each iteration is described by
    \[
    \est{\Pi}_{\operatorname{rel}}^{(k)}=\mathscr{P}_{\calK}\left(\est{\Pi}_{\operatorname{rel}}^{(k-1)}-\gamma_k \nabla f\left(\est{\Pi}_{\operatorname{rel}}^{(k-1)};X,\hX,A,S\right)\right), \text{ for }k\geq 1,
    \]
    where $f$ is the relaxed MLE objective defined in \eqref{eq:relaxed_general}, $S$ is the shift matrix defined in Section~\ref{sec:mle_la_formulate}, and $\mathscr{P}_{\calK}$ is the Euclidean projection onto $\calK$. We use an adaptive learning rate strategy, given by 
    \begin{equation}\label{eq:learning_rate_pgd}
    \gamma_k=\gamma\frac{\log{(k+1)}}{\left(\left\|\nabla f\left(\est{\Pi}_{\operatorname{rel}}^{(k-1)}\right)\right\|_2\vee 10^{-4}\right)\sqrt{k+1}},
    \end{equation}
    where $\gamma > 0$ is a user-specified constant, and the small term $10^{-4}$ is included arbitrarily to prevent numerical blow-up. This strategy is commonly used in practice; see \cite[Chapter~8]{beck2017first} for this and other related learning rate schemes.
    \item \textbf{Simplex. } When $\mathcal{K} = \{Z \in \mathbb{R}^{T \times T} : \mathbbm{1}^\top Z \mathbbm{1} = T,\; Z \geq 0\}$, we use the \emph{Entropic Mirror Descent} algorithm (see \cite[Chapter 9]{beck2017first}), which results in a multiplicative weights update algorithm. As in the case of the simplex, we use a learning rate $\gamma_k>0$ and ${\est{\Pi}_{\operatorname{rel}}}^{(0)}=\frac 1T J_T$ as the initial point. The iterative step is, for each $k\geq 1$,
    \[
    \est{\Pi}_{\operatorname{rel}}^{(k)}=T\frac{\est{\Pi}_{\operatorname{rel}}^{(k-1)}\odot \exp{\left(-\gamma_k\nabla f\left(\est{\Pi}_{\operatorname{rel}}^{(k-1)};X,\hX,A,S\right)\right)}}{\left\|\est{\Pi}_{\operatorname{rel}}^{(k-1)}\odot \exp{\left(-\gamma_k\nabla f\left(\est{\Pi}_{\operatorname{rel}}^{(k-1)};X,\hX,A,S\right)\right)}\right\|_1},
    \]
    where, for a matrix $A\in \matR^{N\times N}$, $\|A\|_1:=\sum_{i,j\in[N]}|A_{ij}|$. The learning rate is chosen as follows,
    \begin{equation}\label{eq:learning_rate_MD}
        \gamma_k=\gamma\frac{\log{(k+1)}}{\left(\left\|\nabla f\left(\est{\Pi}_{\operatorname{rel}}^{(k-1)}\right)\right\|_\infty\vee 10^{-4}\right)\sqrt{k+1}}.
    \end{equation}
    \item \textbf{Birkhoff Polytope. } Consider $\calK=\{Z\in\matR^{T\times T}: \mathbbm{1}^\top Z=\mathbbm{1}^\top, Z\mathbbm{1}=\mathbbm{1},Z\geq 0\}$. To solve the relaxed MLE problem on the Birkhoff polytope, we employ a projected gradient descent strategy (which we call Birkhoff PGD). We consider the same initialization as the previous algorithms ${\est{\Pi}_{\operatorname{rel}}}^{(0)}=\frac 1T J_T$, and the iterations are of the form 
    \[
    \est{\Pi}_{\operatorname{rel}}^{(k)}=\mathscr{P}_{\calK}\left(\est{\Pi}_{\operatorname{rel}}^{(k-1)}-\gamma_k \nabla f\left(\est{\Pi}_{\operatorname{rel}}^{(k-1)};X,\hX,A,S\right)\right),\, \text{ for }k\geq 1.
    \]
    Similar as before, here $f$ is the relaxed MLE objective defined in \eqref{eq:relaxed_general}, $S$ is the shift matrix defined in Section~\ref{sec:mle_la_formulate}, and $\mathscr{P}_{\calK}$ is the Euclidean projection onto $\calK$ (the Birkhoff polytope). We use the Dykstra algorithm \cite{Dykstra} to approximate this projection. The learning rate follows~\eqref{eq:learning_rate_pgd}, with the constant $\gamma$ 
    adjusted as needed. We refer to this algorithm as Birkhoff PGD.
    % the ADMM algorithm. The iterative steps are as follows
    % %
    % \begin{align}
    %     \est{\Pi}_{\operatorname{rel}}^{(k)}&=\argmin{\Pi}\left\{f(\Pi)+\frac{\rho}2\left\|\Pi-Z^{(k-1)}+U^{(k-1)}\right\|^2_F\right\}\, ,\nonumber \\ \label{eq:Z-update}
    %     Z^{(k)}&=\argmin{Z}\left\{\iota_{\calK}(Z)+\frac{\rho}2\left\|\est{\Pi}_{\operatorname{rel}}^{(t)}-Z+U^{(t-1)}\right\|^2_F\right\}\,,\\
    %     U^{(t)}&=U^{(t-1)}+\est{\Pi}_{\operatorname{rel}}^{(t)}-Z^{(t)},\nonumber
    % \end{align}
    % %
    % where $\iota_{\calK}(Z)$ is $0$ if $Z\in \calK$ and $+\infty$ if $Z\notin \calK$. In practice, the $Z$-update in \eqref{eq:Z-update} requires projecting onto the Birkhoff polytope, a step that is computationally demanding in high dimensions. To alleviate this cost, we instead employ the Sinkhorn–Knopp algorithm, widely used in the optimal transport literature, to compute an efficient approximate (oblique) projection.
    %
    
\end{itemize}
\paragraph{Error metrics.} To quantify the success of the proposed methods in the recovery of $\Pi^*\in\calP_T$ we consider the recovery fraction, defined for any matrix $\Pi\in\calP_T$ as 
\[
\operatorname{Recovery}\, \operatorname{fraction}(\Pi):=\left\langle \Pi,\Pi^*\right\rangle_F/T.
\]
Notice that in the previous definition the ground truth $\Pi^*$ is implicit.
Even if our goal is mainly recovery of $\Pi^*$, the alternating optimization method proposed in Algorithm \ref{alg:TS_matching_alternating} allows us to jointly recover $\Pi^*$ and $A^*$. We measure the error for the estimation of $A^*\in \matR^{d\times d}$ with the MSE, defined for a matrix $A\in \matR^{d\times d}$ as follows, 
\[
\operatorname{MSE}(A) := \|A-A^*\|_F^2/d.
\]
\paragraph{Parametric assumption on $A^*$.} 
%\subsubsection{Relaxed MLE with known $A^*$}\label{sec:exp_rel-MLE-knownA}
%
Throughout our experiments, we consider a random $A^*$, constructed as follows. First, we sample a $A'$ with iid standard Gaussian entries. Then we set 
\begin{equation}\label{eq:parametric_A*}
A^*=\theta \frac{A'}{\|A'\|_2}\,,
\end{equation}
where the parameter $\theta\in \matR$ helps to control $\|A^*\|_2$.
%
% \begin{itemize}
%     \item We sample a $A'$ with iid Gaussian entries. 
%     \item We
% \end{itemize}
%
% For simplicity, we restrict our attention, throughout this section, to matrices of the form 
% %
% \begin{equation}\label{eq:parametric_A*}
% A^* = \theta\left(\alpha I_d+(1-\alpha)\frac{\mathbbm{1}\mathbbm{1}^\top}{n}\right),
% \end{equation}
% %
% for $\alpha,\,\theta\in (0,1)$. This choice of $A^*$ interpolates between two interesting regimes: when $\alpha=1$ the coordinates of the time series evolve independently; when $\alpha=0$ the coordinates are fully mixed. Notice that the parameter $\theta$ helps to control $\|A^*\|_2$.

\subsection{Relaxed MLE with known \texorpdfstring{$A^*$}{A*}}\label{sec:exp_rel-MLE-knownA}
We begin by evaluating the performance of the \texttt{RelaxMLE-Round} subroutine, see Algorithm \ref{alg:relaxMLE_round}, when the system matrix $A^*$ is known. We generate time series following the \cvar\ model under different noise levels $\sigma$, with $A^*$ satisfying the parametric assumption in \eqref{eq:parametric_A*}.
To unravel the influence of $\theta$ and $\sigma$, we fix one parameter and vary the other, reporting the recovery accuracy of $\Pi^*$. We choose $\Pi^*=I_T$, and for all methods we initialize with $\est{\Pi}_{\operatorname{rel}}^{(0)}=\frac{1}{T}J_T$. In addition, we consider two contrasting regimes: $(d,T)=(5,50)$, where the ambient dimension is small relative to the number of observations, and $(d,T)=(50,5)$, where the opposite holds. 

Figure~\ref{fig:recovery_vs_scale_A_known} illustrates the effect of the scale parameter $\theta$ on recovery. 
Interestingly, in both regimes considered for $d,T$ the estimation error remains essentially unaffected by $\theta$, suggesting that for these algorithms the problem does not become more difficult at larger scales. It should be noted that, in the case $d=50,\,T=5$, the considered $\sigma$ is higher, since for smaller values of $\sigma$ perfect recovery is achieved for all the algorithms for most random realizations. 
Intuitively, the problem becomes easier in this regime, as discussed in more detail below. In terms of performance, the relaxation to the Birkhoff polytope generally outperforms the simplex and hyperplane relaxations, particularly in high-noise settings (e.g., in the case $d=50,\, T=5$ considered here). Interestingly, the overall differences between the relaxations are not drastic—the simplex relaxation performs comparably to the Birkhoff relaxation across most conditions. One advantage of the hyperplane and simplex relaxations is their comparatively lower computational complexity relative to the Birkhoff relaxation.

In Figure~\ref{fig:recovery_vs_sigma_A_known} we plot the recovery fraction versus the noise level to highlight the effect of $\sigma$ on the recovery level for the same pairs $(d,T)$. This is a complementary plot to Figure \ref{fig:recovery_vs_scale_A_known}, which allows us to read how the recovery decays with the noise level. We obtain similar conclusions: the scale does not seem to affect the recovery fraction and while Birkhoff PGD performs slightly better, the difference in performance remains moderate (expect for high noise scenarios in $d=50\,,T=5$). In addition, we notice that the performance of Birkhoff PGD is very close to the LA estimator. This is expected for small scales, since for $A=0$ solving \eqref{eq:relaxed_general} on the Birkhoff polytope is equivalent to the linear assignment solution in \eqref{eq:LA_pi}. On the other hand, for larger scales it is not obvious from \eqref{eq:relaxed_general} that both approaches would have a similar performance.  

These plots suggest that case $d=50,T=5$ is easier for all algorithms considered, compared to case $d=5,T=50$ (notice that we considered higher levels of noise in case $d=50,T=5$). This is expected since in the case $d=50,T=5$ we have few points in a high dimension, which implies a higher separation between them. More surprising is the fact that the LA estimator performs on par with the best-performing MLE relaxation: the one on the Birkhoff polytope. It seems that the model information, used by the MLE relaxations, does not lead to a noticeable advantage in terms of the matching performance. A natural question is whether LA is optimal in terms of recovery for time series matching, at least in the regime $\|A^*\|_2 < 1$. We explore the case $\|A^*\|_2 \geq 1$ below.
%\EA{This negative result is still interesting, because it is not obvious from the problem formulation that LA should perform on par or better.}
% %

\begin{figure}[htbp]
  \centering
  \setlength{\floatsep}{2pt} % Reduce space between floats
  \setlength{\textfloatsep}{5pt} % Reduce space around the float
  \setlength{\intextsep}{5pt} % Reduce space when using [h] placement
  
  %---- Row 1 ----
  \begin{subfigure}[b]{0.45\textwidth}
    \includegraphics[width=0.95\linewidth,height=0.3\textheight,keepaspectratio]{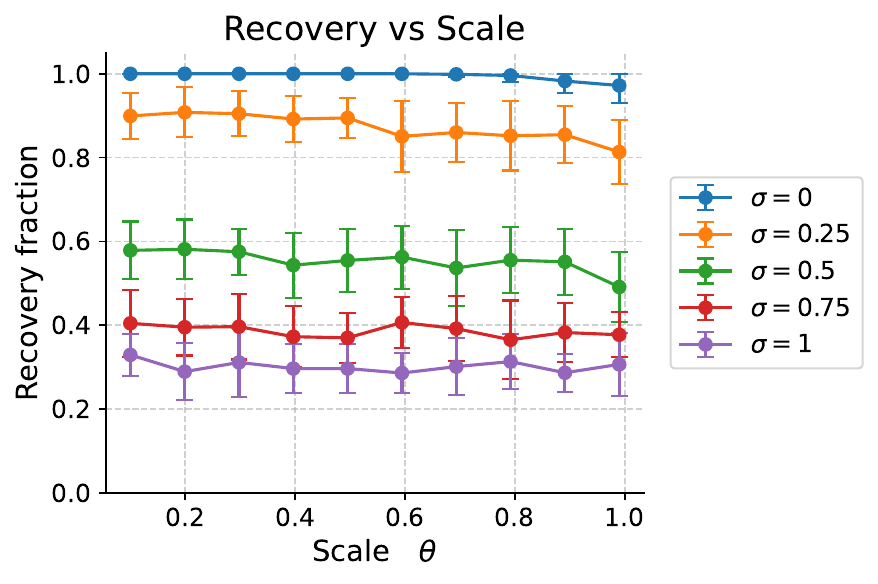}
    \caption{Hyperplane relaxed MLE $d=5,\,T=50$.}
  \end{subfigure}
  \hfill
  \begin{subfigure}[b]{0.45\textwidth}
    \includegraphics[width=0.95\linewidth,height=0.3\textheight,keepaspectratio]{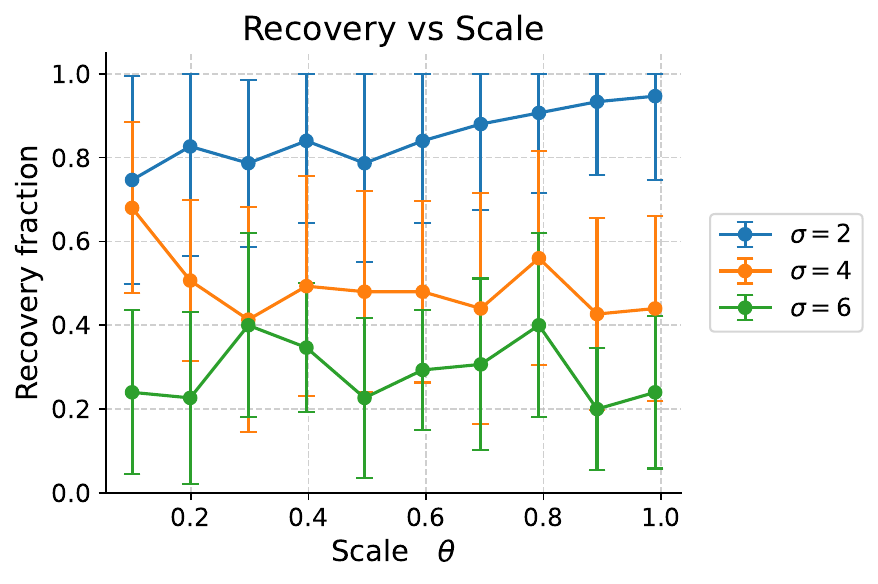}
    \caption{Hyperplane relaxed MLE $d=50,\,T=5$.}
  \end{subfigure}
  
  %\vspace{-2pt} % Reduce space between rows
  %---- Row 2 ----
  \begin{subfigure}[b]{0.45\textwidth}
    \includegraphics[width=0.95\linewidth,height=0.3\textheight,keepaspectratio]{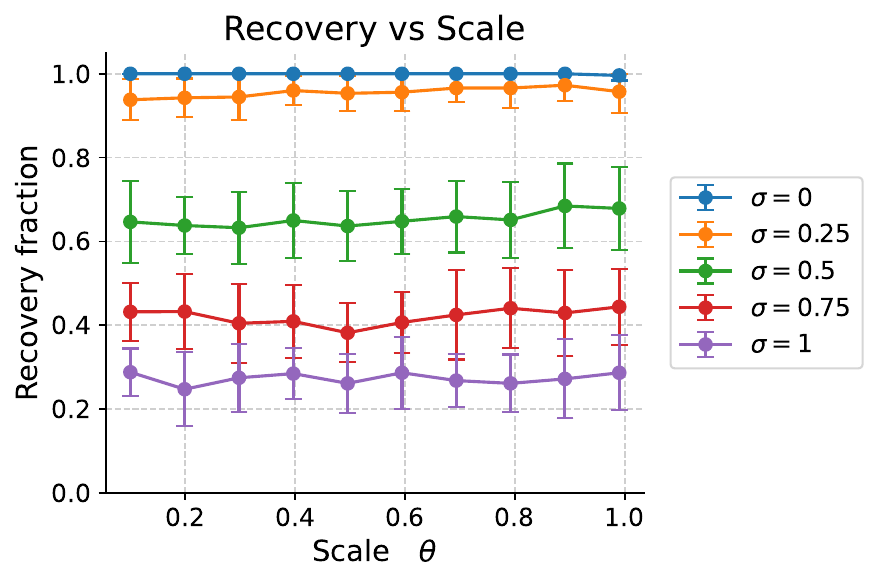}
    \caption{Simplex relaxed MLE $d=5,\,T=50$.}
  \end{subfigure}
  \hfill
  \begin{subfigure}[b]{0.45\textwidth}
    \includegraphics[width=0.95\linewidth,height=0.3\textheight,keepaspectratio]{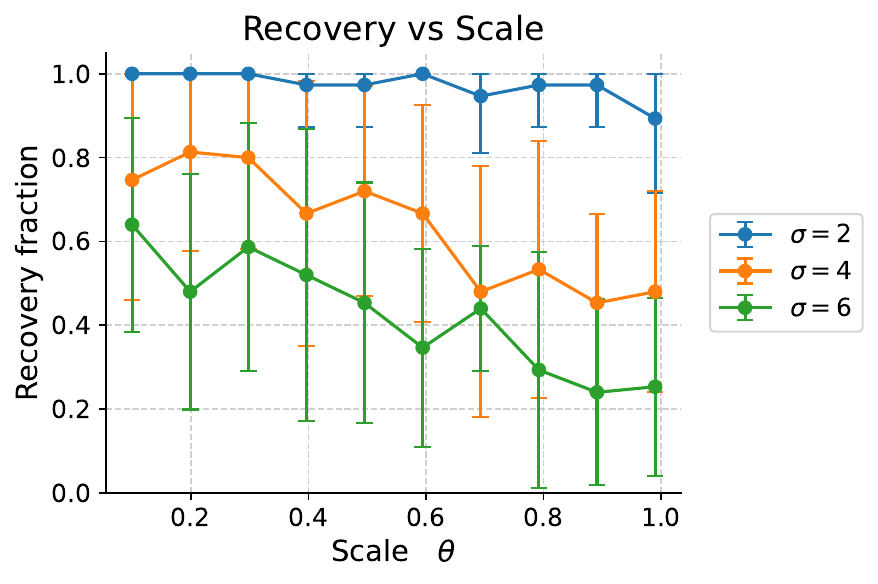}
    \caption{Simplex relaxed MLE $d=50,\,T=5$.}
  \end{subfigure}

  %\vspace{-2pt} % Reduce space between rows
  %---- Row 3 ----
  \begin{subfigure}[b]{0.45\textwidth}
    \includegraphics[width=0.95\linewidth,height=0.3\textheight,keepaspectratio]{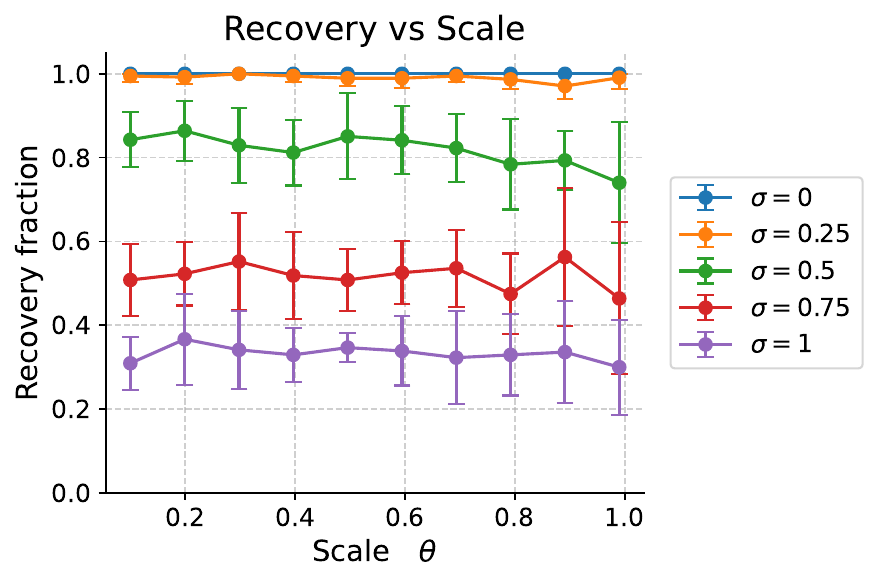}
    \caption{Linear assignment estimator $d=5,\,T=50$.}
  \end{subfigure}
  \hfill
  \begin{subfigure}[b]{0.45\textwidth}
    \includegraphics[width=0.95\linewidth,height=0.3\textheight,keepaspectratio]{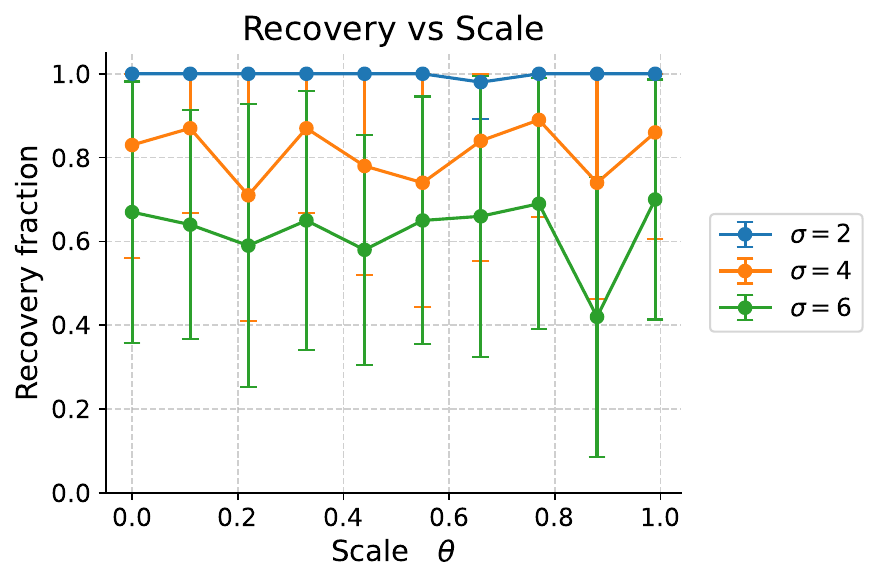}
    \caption{Linear assignment estimator $d=50,\,T=5$.}
  \end{subfigure}

  %\vspace{-2pt} % Reduce space between rows
  %---- Row 4 ----
  \begin{subfigure}[b]{0.45\textwidth}
    \includegraphics[width=0.95\linewidth,height=0.3\textheight,keepaspectratio]{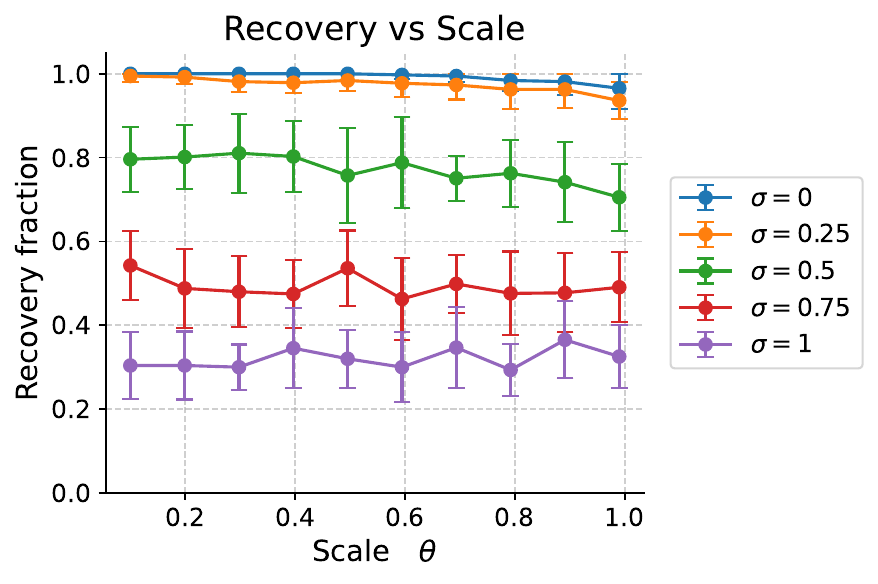}
    \caption{Birkhoff relaxed MLE $d=5,\,T=50$.}
  \end{subfigure}
  \hfill
  \begin{subfigure}[b]{0.45\textwidth}
    \includegraphics[width=0.95\linewidth,height=0.3\textheight,keepaspectratio]{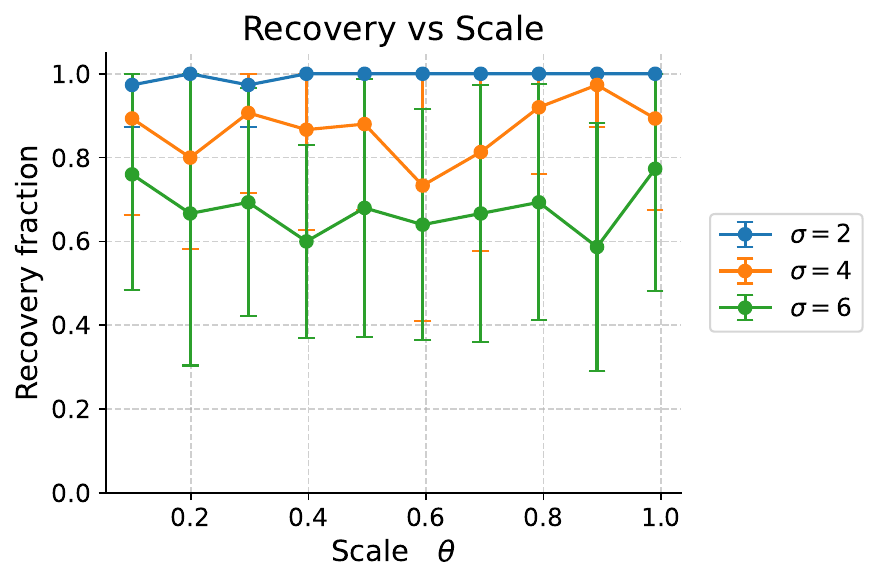}
    \caption{Birkhoff relaxed MLE $d=50,\,T=5$.}
  \end{subfigure}
  %\vspace{-2pt} % Reduce space between rows
  \caption{Recovery fraction vs. scale $\theta$ using Algorithm \ref{alg:relaxMLE_round}. We assume known $A^*$ of the form \eqref{eq:parametric_A*}, and consider different values for $\theta$. For each $(\theta,\sigma)$ pair, the plotted value corresponds to the average over $30$ Monte Carlo samples of the \cvar\ model. The error bars reflects one standard deviation above and below the mean.}
  \label{fig:recovery_vs_scale_A_known}
\end{figure}

\begin{figure}[htbp]
  \centering
  %---- Row 1 ----
  \begin{subfigure}[b]{0.45\textwidth}
    \includegraphics[width=0.95\linewidth,height=0.3\textheight,keepaspectratio]{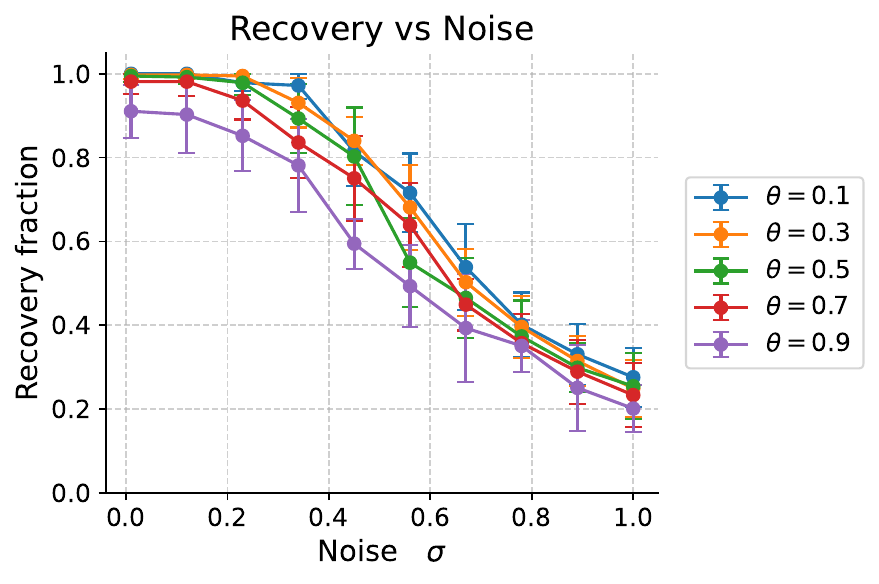}
     \caption{Hyperplane relaxed MLE $d=5,\,T=50$.}
  \end{subfigure}
  \hfill
  \begin{subfigure}[b]{0.45\textwidth}
    \includegraphics[width=0.95\linewidth,height=0.3\textheight,keepaspectratio]{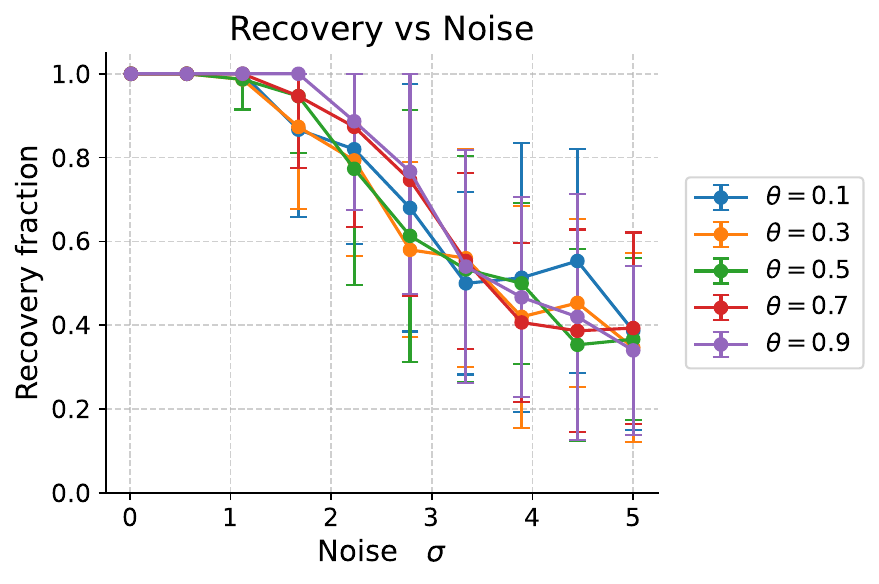}
     \caption{Hyperplane relaxed MLE $d=50,\,T=5$.}
  \end{subfigure}

  %---- Row 2 ----
  \begin{subfigure}[b]{0.45\textwidth}
    \includegraphics[width=0.95\linewidth,height=0.3\textheight,keepaspectratio]{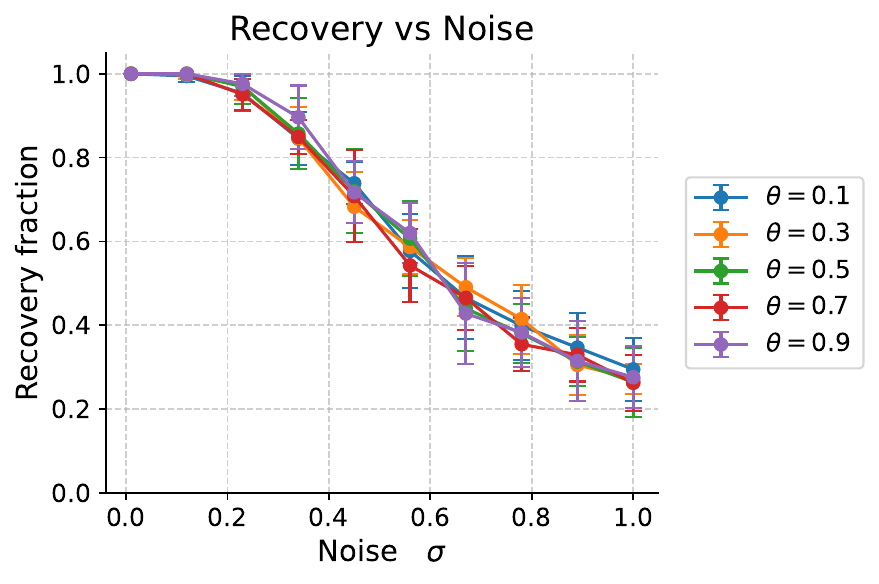}
    \caption{Simplex relaxed MLE $d=5,\,T=50$.}
  \end{subfigure}
  \hfill
  \begin{subfigure}[b]{0.45\textwidth}
    \includegraphics[width=0.95\linewidth,height=0.3\textheight,keepaspectratio]{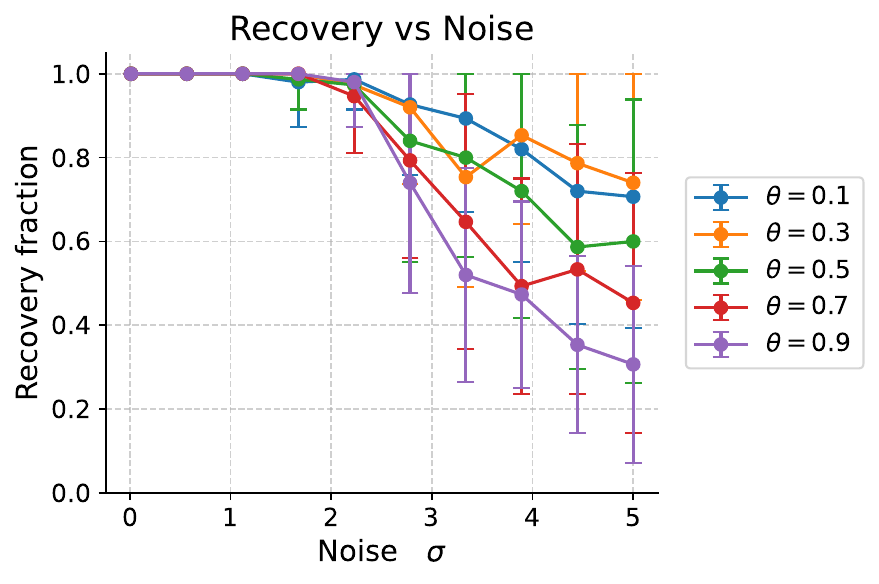}
    \caption{Simplex relaxed MLE $d=50,\,T=5$.}
  \end{subfigure}

   %---- Row 3 ----
  \begin{subfigure}[b]{0.45\textwidth}
    \includegraphics[width=0.95\linewidth,height=0.3\textheight,keepaspectratio]{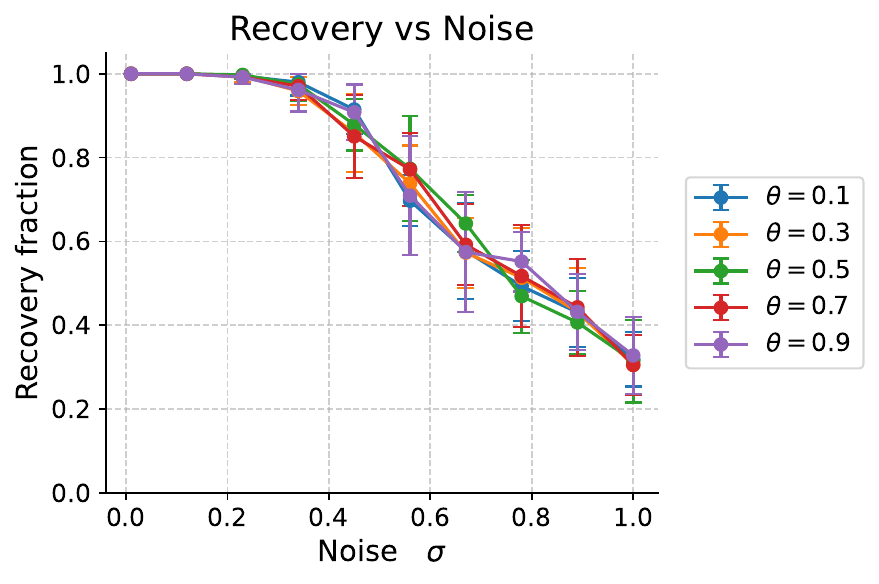}
    \caption{Linear assignment estimator $d=5,\,T=50$.}
  \end{subfigure}
  \hfill
  \begin{subfigure}[b]{0.45\textwidth}
    \includegraphics[width=0.95\linewidth,height=0.3\textheight,keepaspectratio]{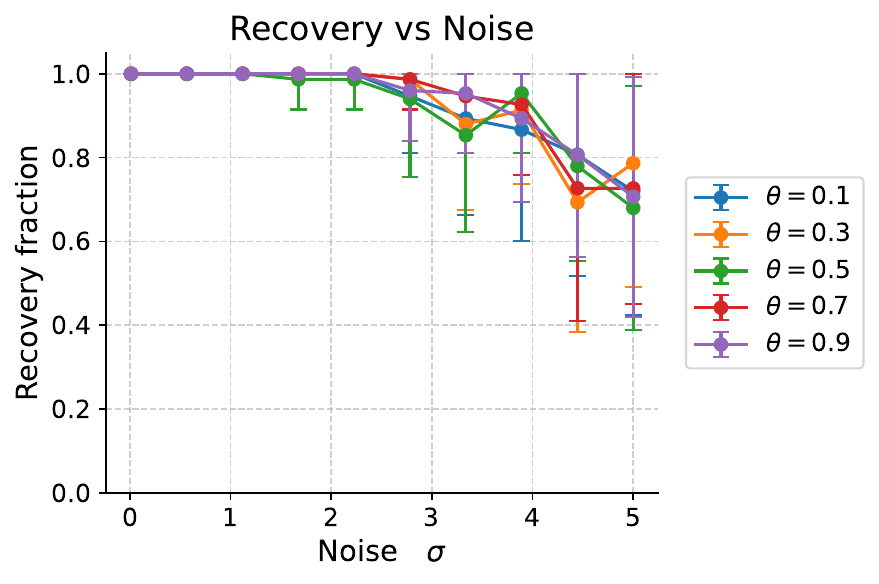}
    \caption{Linear assignment estimator $d=50,\,T=5$.}
    \end{subfigure}

  %---- Row 4 ----
  \begin{subfigure}[b]{0.45\textwidth}
    \includegraphics[width=0.95\linewidth,height=0.3\textheight,keepaspectratio]{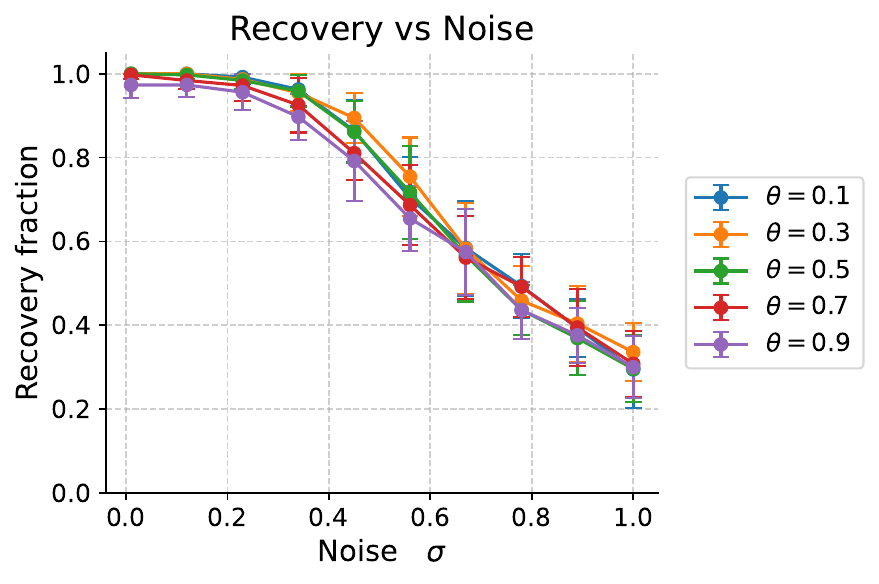}
    \caption{Birkhoff relaxed MLE $d=5,\,T=50$.}
  \end{subfigure}
  \hfill
  \begin{subfigure}[b]{0.45\textwidth}
    \includegraphics[width=0.95\linewidth,height=0.3\textheight,keepaspectratio]{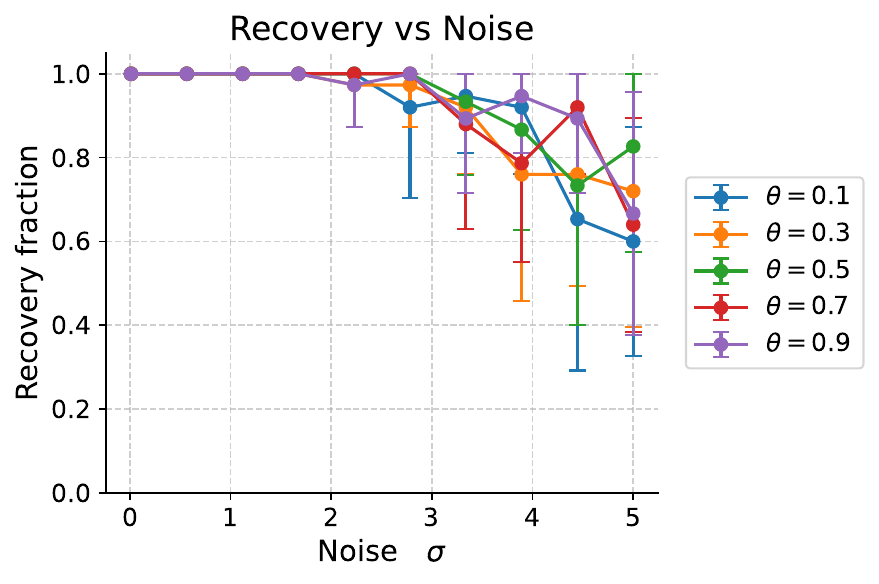}
    \caption{Birkhoff relaxed MLE $d=50,\,T=5$.}
  \end{subfigure}

  \caption{Recovery fraction vs. noise $\sigma$ using Algorithm \ref{alg:relaxMLE_round}. The setting is analogous to Fig.\ref{fig:recovery_vs_scale_A_known}.}
  \label{fig:recovery_vs_sigma_A_known}
\end{figure}

\subsection{Algorithms with unknown \texorpdfstring{$A^*$}{A*}}\label{sec:algos-unknownA}
We now evaluate the proposed alternating optimization procedure in Algorithm \ref{alg:TS_matching_alternating} and the linear assignment estimator in \eqref{eq:LA_pi}, in terms of $\Pi^*$ recovery. In these experiments, we assume that $\sigma$ is known, while $A^*$ remains unknown. We initialize $\Pi^{(0)}$ uniformly at random and use $K=5$ alternating minimization steps. In Appendix \ref{sec:additional_experiments} we include some experiments for the strategy of estimating $A^*$ first, discussed in Remark \ref{rem:estimate_A_first}.
In Figure~\ref{fig:recovery_vs_scale_A_unknown}, we present the recovery fraction as a function of the scale parameter~$\theta$ for $(d,T)\in\{(5,50),(50,5)\}$. As in the case where $A^*$ is known, the recovery fraction remains fairly stable across scales. In this case, all algorithms based on MLE relaxations have a similar performance in the case $d=5,\,T=50$, while the LA estimator slightly outperform them. On the other hand, linear assignment outperforms relaxed MLE-based algorithms in the case $d=50,\,T=5$, achieving perfect recovery for the considered values of $\sigma$. This suggests that random initialization performs poorly for estimating $A^*$, particularly with small $T$. It should be noted that the variance, reflected in the error bars, is relatively high for this choice of parameters $d,T$. Interestingly, for $d=5$ and $T=50$, the Birkhoff relaxation achieves an average recovery performance comparable to linear assignment. This is somewhat surprising, as one might expect that an initial random permutation would adversely affect the $A$-update step in Algorithm \ref{alg:TS_matching_alternating}. Similar conclusions can be obtained from the complementary plots, in Figure~\ref{fig:recovery_vs_sigma_A_unknown}, where the recovery fraction is plotted against the noise level.

\begin{figure}[htbp]
  \centering
  %---- Row 1 ----
  \begin{subfigure}[b]{0.45\textwidth}
    \includegraphics[width=0.95\linewidth,height=0.3\textheight,keepaspectratio]{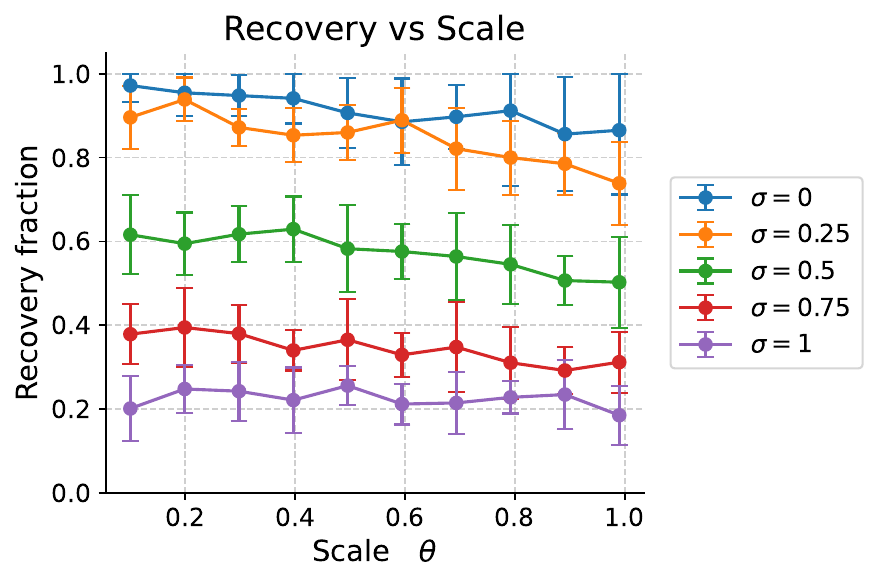}
    \caption{Hyperplane relaxed MLE $d=5,\,T=50$.}
  \end{subfigure}
  \hfill
  \begin{subfigure}[b]{0.45\textwidth}
    \includegraphics[width=0.95\linewidth,height=0.3\textheight,keepaspectratio]{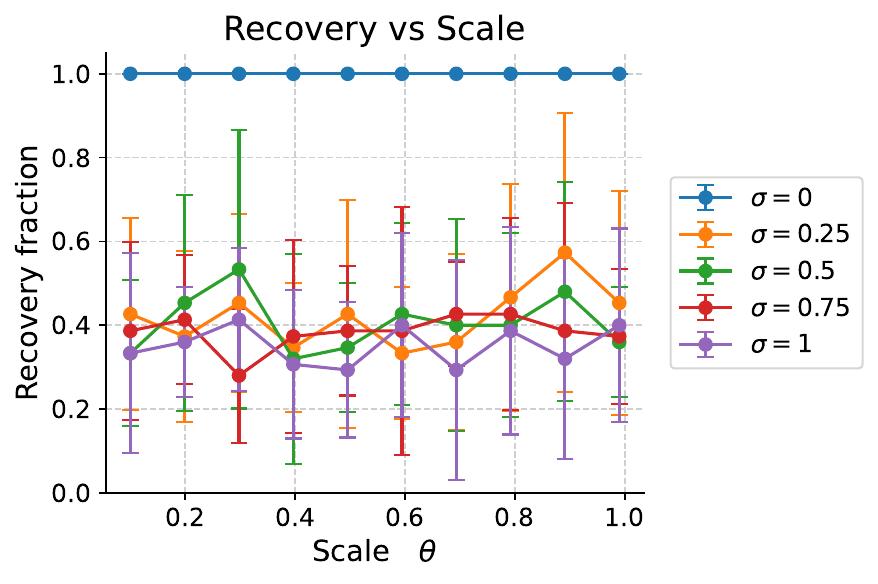}
    \caption{Hyperplane relaxed MLE $d=50,\,T=5$.}
  \end{subfigure}

  %---- Row 2 ----
  \begin{subfigure}[b]{0.45\textwidth}
    \includegraphics[width=0.95\linewidth,height=0.3\textheight,keepaspectratio]{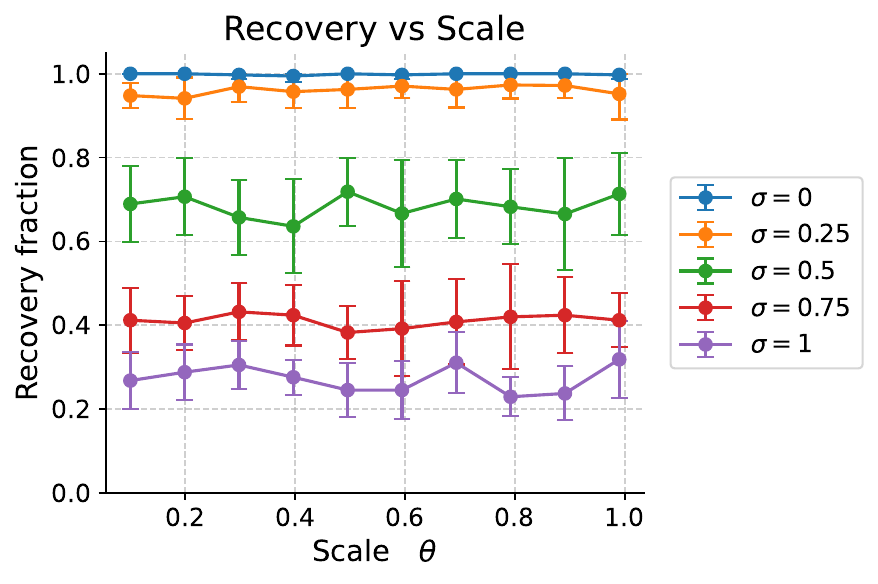}
    \caption{Simplex relaxed MLE $d=5,\,T=50$.}
  \end{subfigure}
  \hfill
  \begin{subfigure}[b]{0.45\textwidth}
    \includegraphics[width=0.95\linewidth,height=0.3\textheight,keepaspectratio]{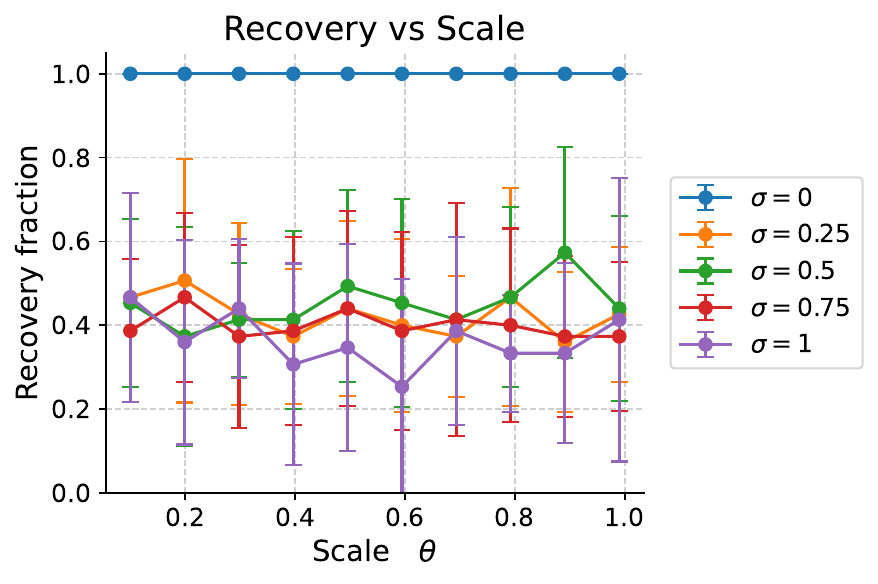}
    \caption{Simplex relaxed MLE $d=50,\,T=5$.}
  \end{subfigure}

  %---- Row 3 ----
  \begin{subfigure}[b]{0.45\textwidth}
    \includegraphics[width=0.95\linewidth,height=0.3\textheight,keepaspectratio]{Figures/A_unknown_AltMin/NEW_recovery_vs_scale_id_22_d_5_T_50_alpha_0.5_alg_LA_A_False.pdf}
    \caption{Linear assignment estimator $d=5,\,T=50$.}
  \end{subfigure}
  \hfill
  \begin{subfigure}[b]{0.45\textwidth}
    \includegraphics[width=0.95\linewidth,height=0.3\textheight,keepaspectratio]{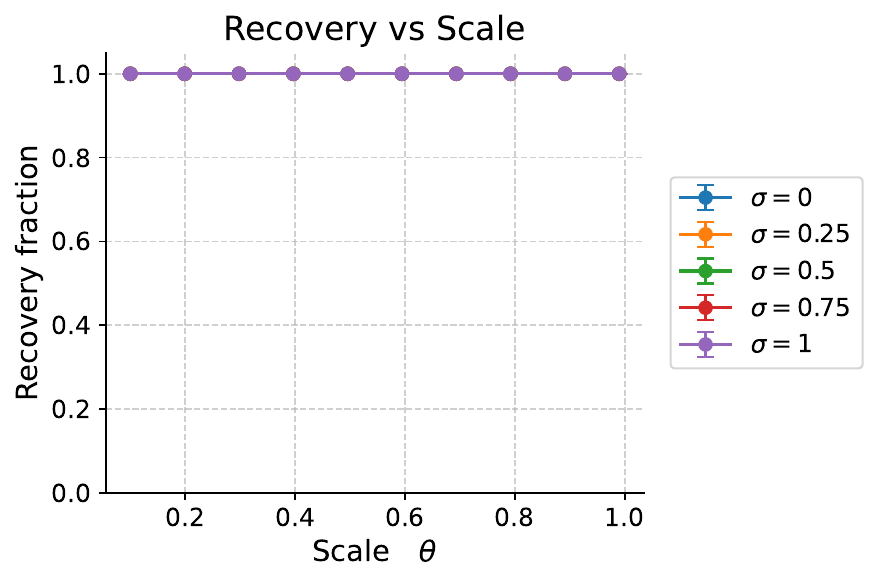}
    \caption{Linear assignment estimator $d=50,\,T=5$.}
  \end{subfigure}

  %---- Row 4 ----
  \begin{subfigure}[b]{0.45\textwidth}
    \includegraphics[width=0.95\linewidth,height=0.3\textheight,keepaspectratio]{Figures/A_unknown_AltMin/NEW_recovery_vs_scale_id_22_d_5_T_50_alpha_0.5_alg_PGD_A_False.pdf}
    \caption{Birkhoff relaxed MLE $d=5,\,T=50$.}
  \end{subfigure}
  \hfill
  \begin{subfigure}[b]{0.45\textwidth}
    \includegraphics[width=0.95\linewidth,height=0.3\textheight,keepaspectratio]{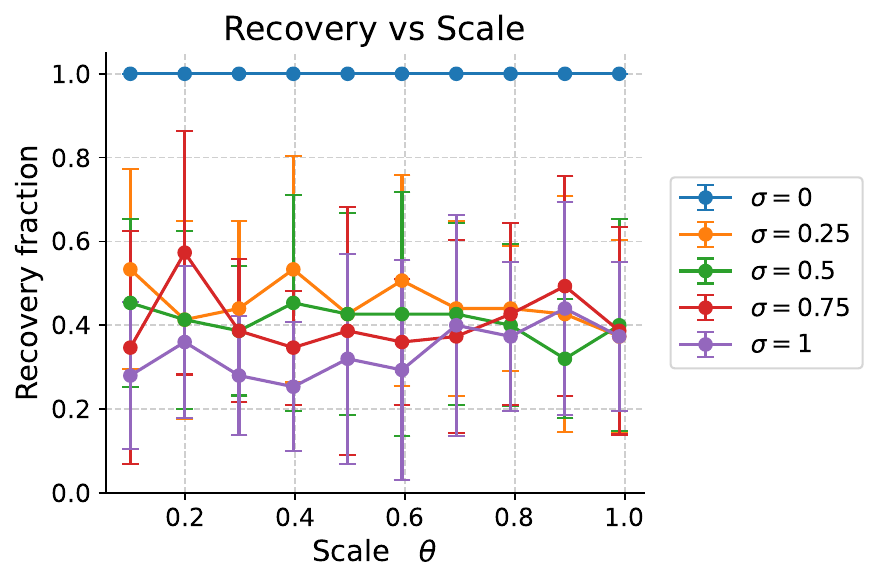}
    \caption{Birkhoff relaxed MLE $d=50,\,T=50$.}
  \end{subfigure}

  \caption{Recovery fraction vs. scale $\theta$ using Algorithm \ref{alg:TS_matching_alternating} with $K=5$. $A^*$ (unknown) is of the form \eqref{eq:parametric_A*}. We average $30$ Monte Carlo samples of the \cvar\ model. The error bars reflects one standard deviation above and below the mean.}
  \label{fig:recovery_vs_scale_A_unknown}
\end{figure}
\begin{remark}[Number of alternating optimization iterations]
     We observe that often after one or two alternating optimization iterations in Algorithm \ref{alg:TS_matching_alternating}, the estimator for the permutation converges. This suggest that estimating $A^*$ first, as discussed in Remark~\ref{rem:estimate_A_first}, is a viable alternative. We evaluate this in Appendix~\ref{sec:est_A_first}.
\end{remark}
\begin{figure}[htbp]
  \centering
  %---- Row 1 ----
  \begin{subfigure}[b]{0.45\textwidth}
    \includegraphics[width=0.95\linewidth,height=0.3\textheight,keepaspectratio]{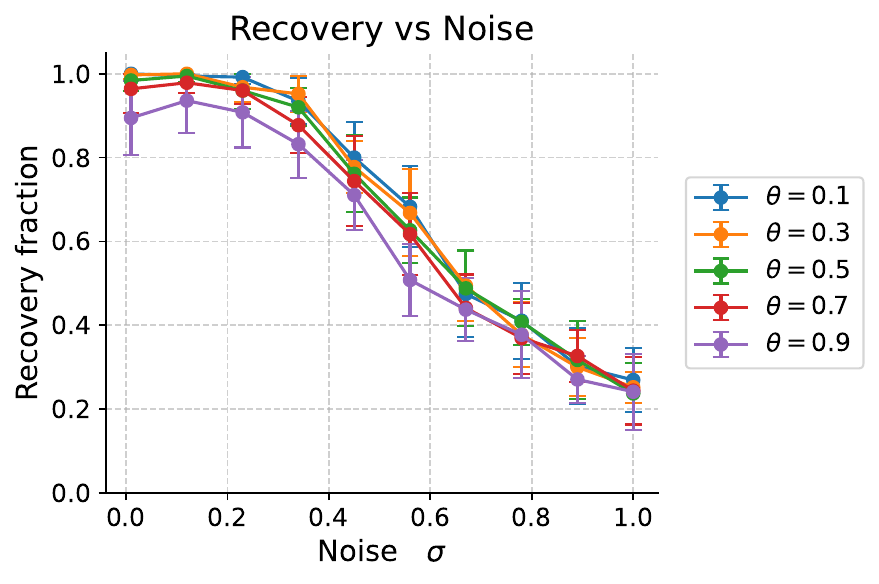}
    \caption{Hyperplane relaxed MLE $d=5,\,T=50$.}
  \end{subfigure}
  \hfill
  \begin{subfigure}[b]{0.45\textwidth}
    \includegraphics[width=0.95\linewidth,height=0.3\textheight,keepaspectratio]{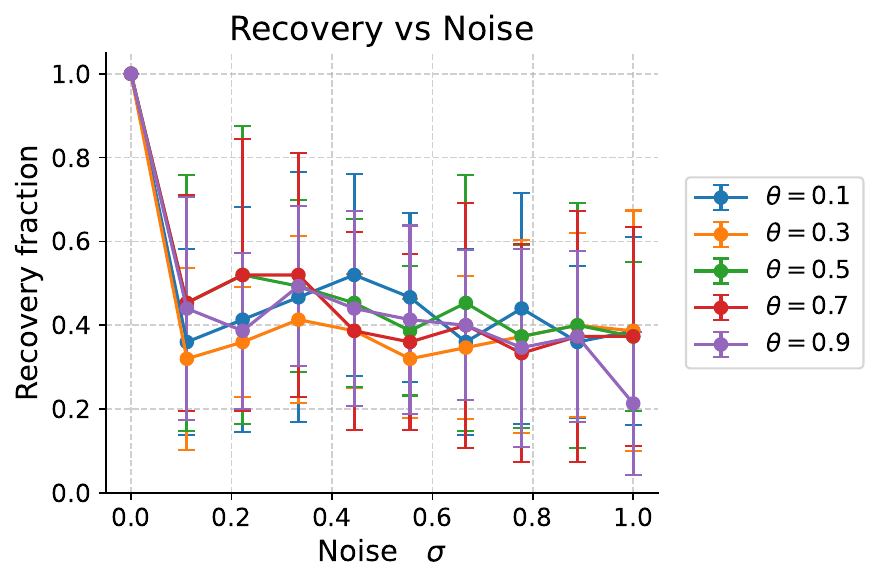}
    \caption{Hyperplane relaxed MLE $d=50,\,T=5$.}
  \end{subfigure}

  %---- Row 2 ----
  \begin{subfigure}[b]{0.45\textwidth}
    \includegraphics[width=0.95\linewidth,height=0.3\textheight,keepaspectratio]{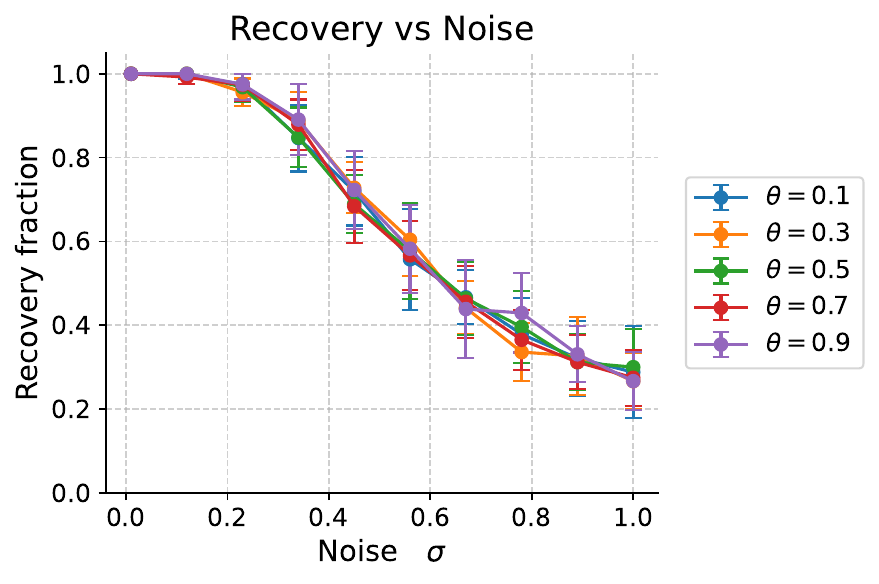}
    \caption{Simplex relaxed MLE $d=5,\,T=50$.}
  \end{subfigure}
  \hfill
  \begin{subfigure}[b]{0.45\textwidth}
    \includegraphics[width=0.95\linewidth,height=0.3\textheight,keepaspectratio]{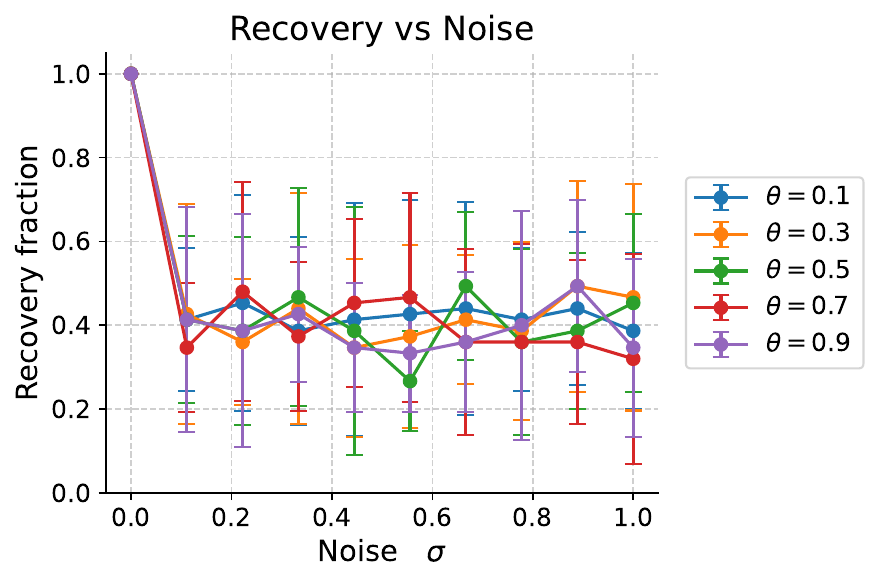}
    \caption{Simplex relaxed MLE $d=50,\,T=5$.}
  \end{subfigure}

  %---- Row 3 ----
  \begin{subfigure}[b]{0.45\textwidth}
    \includegraphics[width=0.95\linewidth,height=0.3\textheight,keepaspectratio]{Figures/A_unknown_AltMin/NEW_recovery_vs_sigma_id_17_d_5_T_50_alpha_0.5_alg_LA_A_False.pdf}
   \caption{Linear assignment estimator $d=5,\,T=50$.}
  \end{subfigure}
  \hfill
  \begin{subfigure}[b]{0.45\textwidth}
    \includegraphics[width=0.95\linewidth,height=0.3\textheight,keepaspectratio]{Figures/A_unknown_AltMin/NEW_recovery_vs_sigma_id_11_d_50_T_5_alpha_0.5_alg_LA_A_False.pdf}
    \caption{Linear assignment estimator $d=50,\,T=5$.}
  \end{subfigure}

  %---- Row 4 ----
  \begin{subfigure}[b]{0.45\textwidth}
    \includegraphics[width=0.95\linewidth,height=0.3\textheight,keepaspectratio]{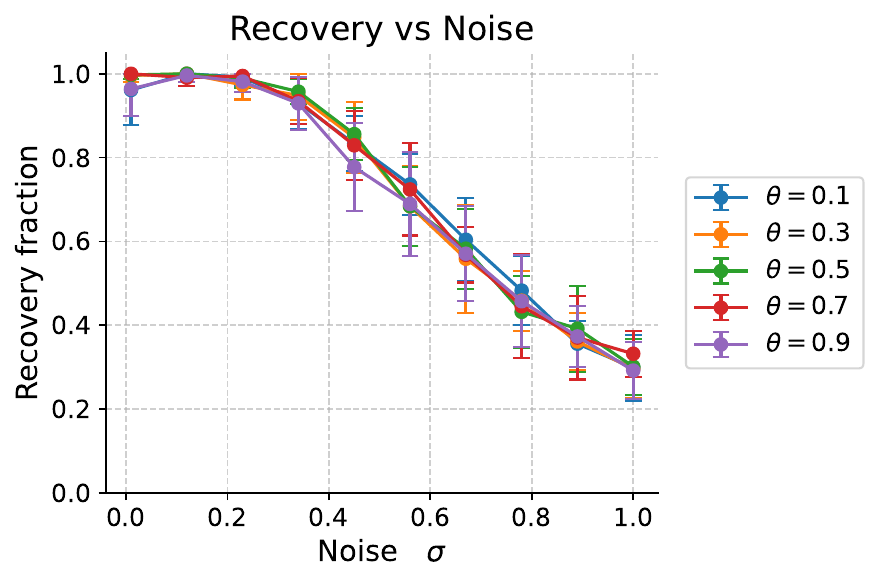}
    \caption{Birkhoff relaxed MLE $d=5,\,T=50$.}
  \end{subfigure}
  \hfill
  \begin{subfigure}[b]{0.45\textwidth}
    \includegraphics[width=0.95\linewidth,height=0.3\textheight,keepaspectratio]{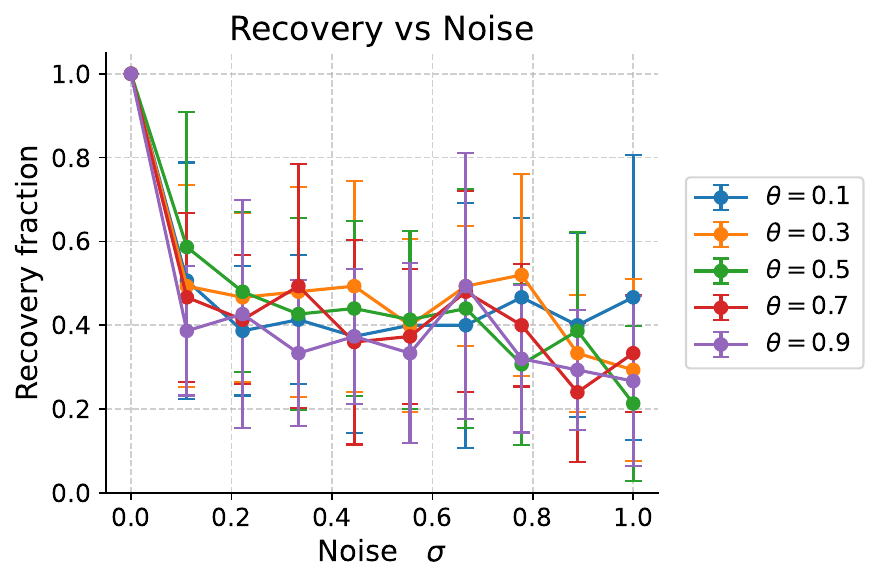}
    \caption{Birkhoff relaxed MLE $d=50,\,T=5$.}
  \end{subfigure}

  \caption{Recovery fraction vs. noise $\sigma$ with $A^*$ unknown. This figure is complementary to Fig. \ref{fig:recovery_vs_scale_A_unknown}, under an analogous setting.}
  \label{fig:recovery_vs_sigma_A_unknown}
\end{figure}
\paragraph{Estimation of $A^*$.}Although the estimation of $A^*$ is not our primary objective, Algorithm \ref{alg:TS_matching_alternating} jointly estimates both $\Pi^*$ and $A^*$. To assess the performance of the algorithm for the estimation of $A^*$, we fix $d=5$, $\theta=0.5$, and vary the time horizon $T\in\{10,20,30,50,100\}$. We plot the performance of the Birkhoff-based relaxation, but in our experiments the hyperplane and the simplex performs similarly for this task. We report the estimation error for $\sigma=0.5$, since other values produce qualitatively similar results. In Figure~\ref{fig:MSE_A}, we plot the MSE for the estimation of $A^*$ for different values of $T$. As expected, $\mathrm{MSE}(A)$ decreases as $T$ increases, indicating that Algorithm~\ref{alg:TS_matching_alternating} successfully estimates $A^*$. Figure~\ref{fig:scatter_A_Pi} shows a scatterplot of the estimation errors of $A^*$ and $\Pi^*$ (in terms of the recovery fraction) over $50$ samples with $T = 100$. Although there is a tendency for better estimation of $A^*$ to correspond to improved recovery of $\Pi^*$, similar estimation errors for $A^*$ can still lead to markedly different recovery fractions.

\begin{figure}[htbp]
  \centering
  %---- Row 1 ----
  \begin{subfigure}[b]{0.45\textwidth}
    \includegraphics[width=0.95\linewidth,height=0.3\textheight,keepaspectratio]{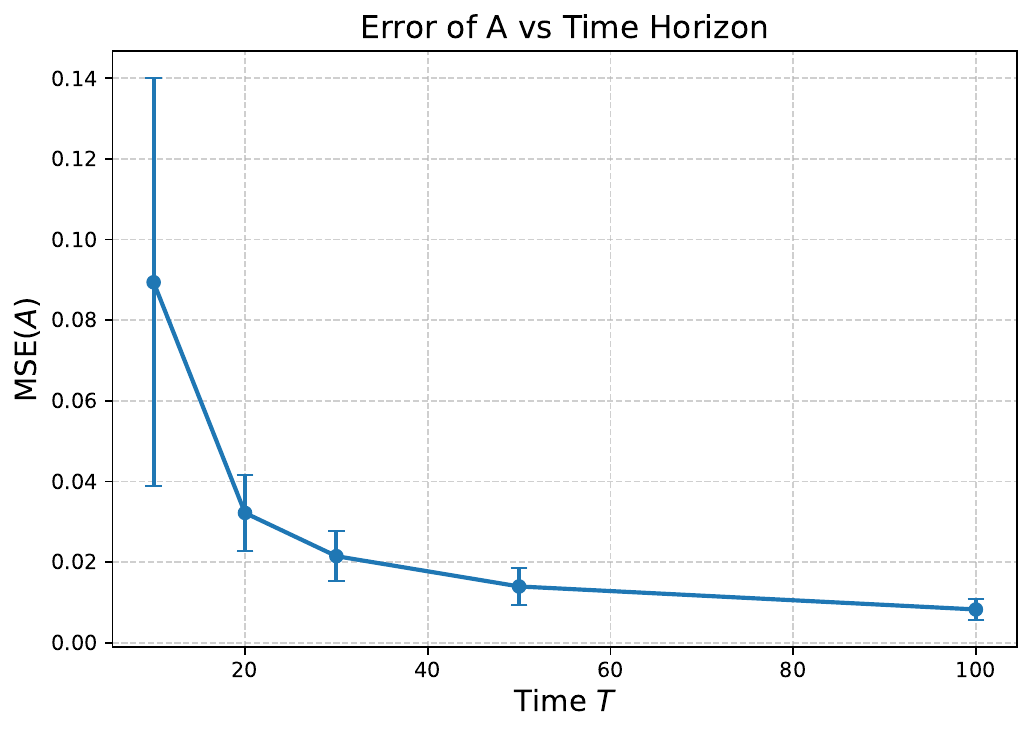}
    \caption{Estimation MSE for $A^*$ }
    \label{fig:MSE_A}
  \end{subfigure}
  \hfill
  \begin{subfigure}[b]{0.45\textwidth}
    \includegraphics[width=0.95\linewidth,height=0.22\textheight,keepaspectratio]{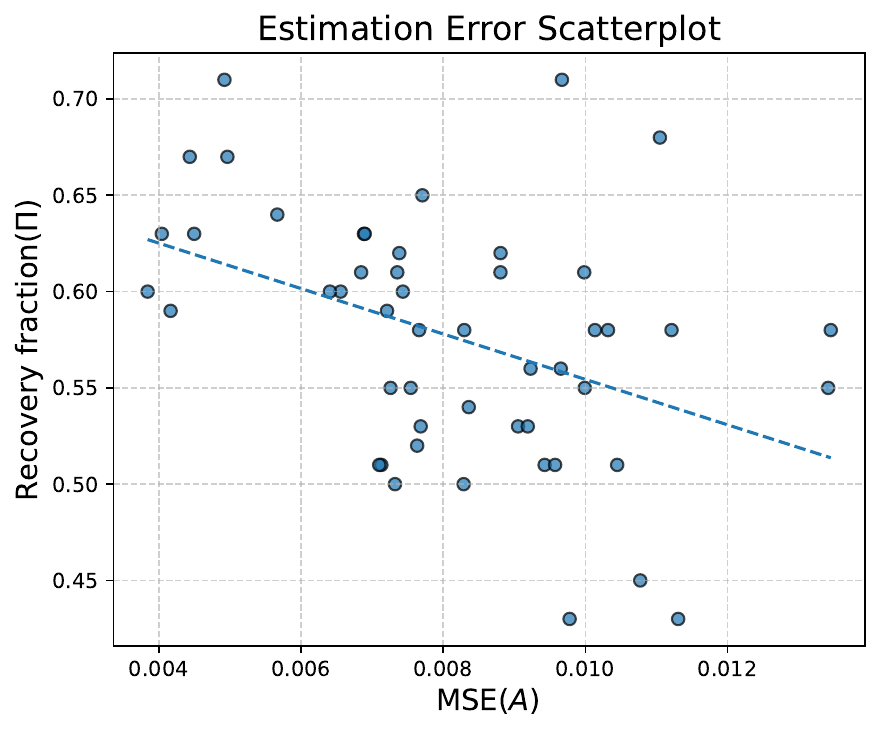}
    \caption{Error for estimating $\Pi^*$ and $A^*$}
    \label{fig:scatter_A_Pi}
  \end{subfigure}
  \caption{Estimation error for $A^*$. We fix $d=5$, $\sigma=0.5$, $\theta=0.5$. In Fig.\ref{fig:MSE_A} we plot MSE($A$) for $T\in\{10,20,30,50,100\}$ averaged over $30$ Monte Carlo samples (the error bars reflect one standard deviation above and below the mean). Fig.\ref{fig:scatter_A_Pi} is a scatter plot, over $50$ samples, for the error of estimating $\Pi^*$ and $A^*$. The dashed line represent the linear trend.}
  \label{fig:recovery_A}
\end{figure}
\paragraph{Recovery vs. $T$.}We examine the recovery performance of the LA estimator and Birkhoff PGD in function of the time horizon $T$. In Figure \ref{fig:recovery_vs_T_stable} we show the average recovery of these estimators for the scale $\theta = 0.5$ and $d\in\{5,25\}$. We observe that the performance of both estimators is very similar for both dimensions. It should be noted that while we report only the scale $\theta=0.5$, similar results were obtained for other values $\theta<1$. This supports the hypothesis that LA is already optimal, or near optimal, in this regime. In addition, we observe that the recovery worsens as $T$ grows, as predicted by Theorem \ref{thm:main_upper_bound}. As we will see next, the situation is slightly different for $\theta>1$.
\begin{figure}[ht]
    \centering
    \begin{subfigure}[b]{0.45\textwidth}
        \centering
        \includegraphics[width=0.95\linewidth,height=0.3\textheight,keepaspectratio]{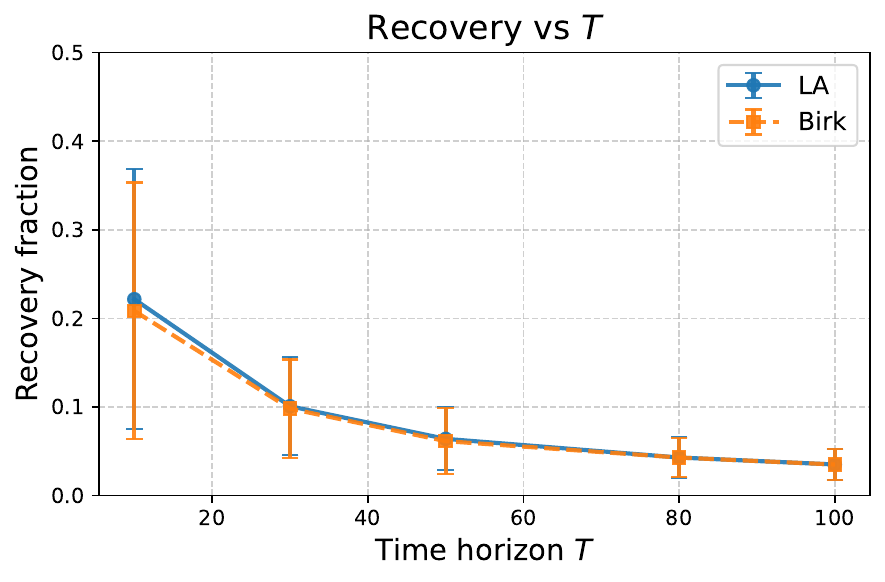}
        \caption{$\theta = 0.5$, $d=5$}
    \end{subfigure}
    \begin{subfigure}[b]{0.45\textwidth}
        \centering
        \includegraphics[width=0.95\linewidth,height=0.3\textheight,keepaspectratio]{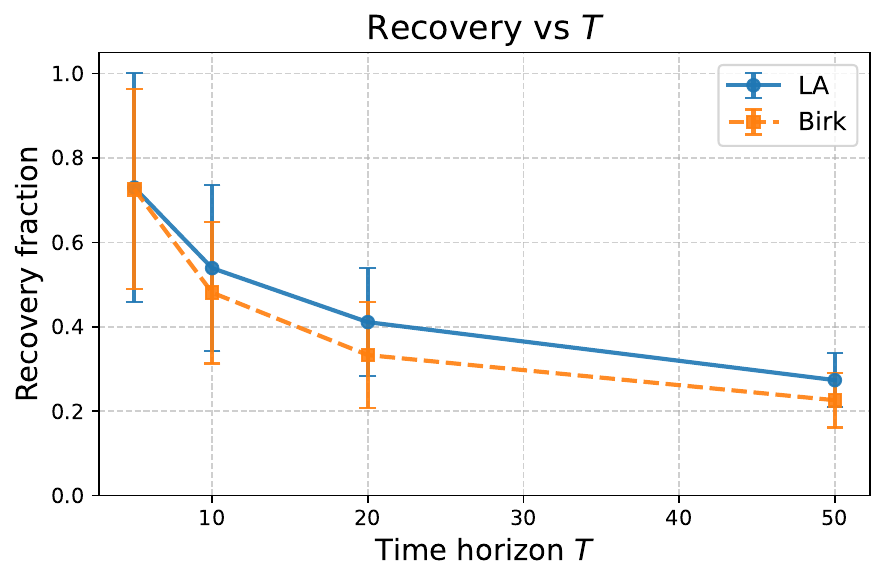}
        \caption{$\theta = 0.5$, $d=25$}
    \end{subfigure}
 \caption{Comparison of recovery performance for the LA and Birkhoff PGD estimators for different time horizons. The recovery fraction is the average over $30$ samples, and the error bars reflect one standard deviation above and below the average. }
    \label{fig:recovery_vs_T_stable}
\end{figure}

\paragraph{The case $\theta\geq 1$.} This case is interesting, since the processes $(x_t)_{t\in[T]},(x^\#)_{t\in[T]}$ become unstable, from the dynamical systems perspective. 
%The analysis of this regime in the problem of estimating $A^*$ has posed technical challenges. \EA{@Hemant: maybe you can rephrase/complement this. E.g., are any results pointing that it is a harder regime from the point of view of the estimation of $A^*$.} 
A natural question is if its possible to non-trivially recover the hidden permutation $\Pi^*$ in this regime, and does the problem becomes easier from the matching perspective. Intuitively, the problem might become easier if the individual points are more separated. We evaluate this by considering the scales $\theta\in\{1.5,2,2.5,3\}$.

In Figure~\ref{fig:recovery_vs_T_unstable}, we plot the average recovery fraction for the estimated permutation as a function of $T$. The comparison includes the performance of the LA estimator and the Birkhoff PGD method. We observe a transition in recovery performance as the scale parameter $\theta$ increases: for $\theta = 1.5$, the Birkhoff-based estimator slightly outperforms LA, whereas for $\theta = 3$, LA achieves better recovery. A similar trend appears when varying $T$: for smaller $T$, Birkhoff PGD performs slightly better, while for larger $T$, LA tends to outperform it in average. It should be noted that the LA estimator has larger variance at the considered scales. Intuitively, as $T$ grows, the underlying unstable processes evolve for longer periods, leading to more separated trajectories and hence an easier matching task. The same argument applies for larger scales, as the separation between points increases quickly.
\begin{figure}[ht]
    \centering
    \begin{subfigure}[b]{0.45\textwidth}
        \centering
        \includegraphics[width=0.95\linewidth,height=0.3\textheight,keepaspectratio]{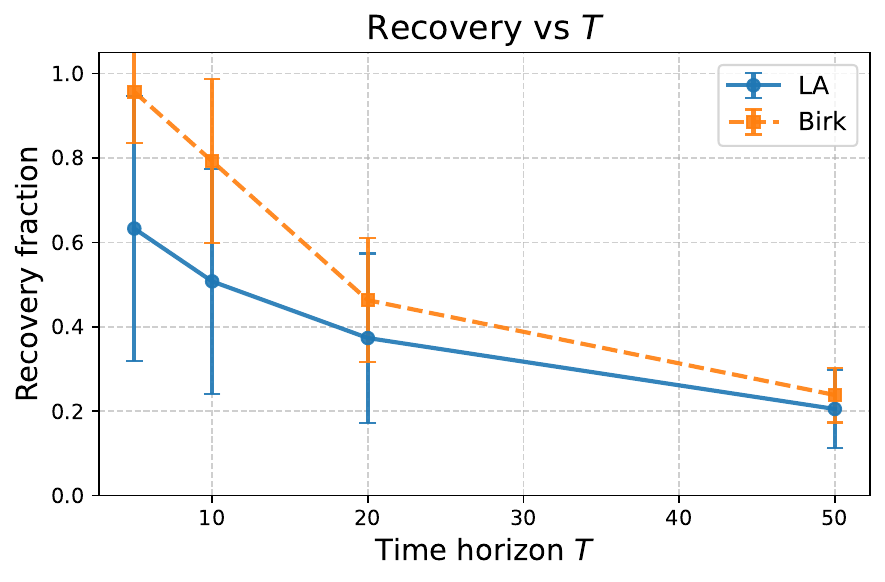}
        \caption{$\theta = 1.5$}
    \end{subfigure}
    \begin{subfigure}[b]{0.45\textwidth}
        \centering
        \includegraphics[width=0.95\linewidth,height=0.3\textheight,keepaspectratio]{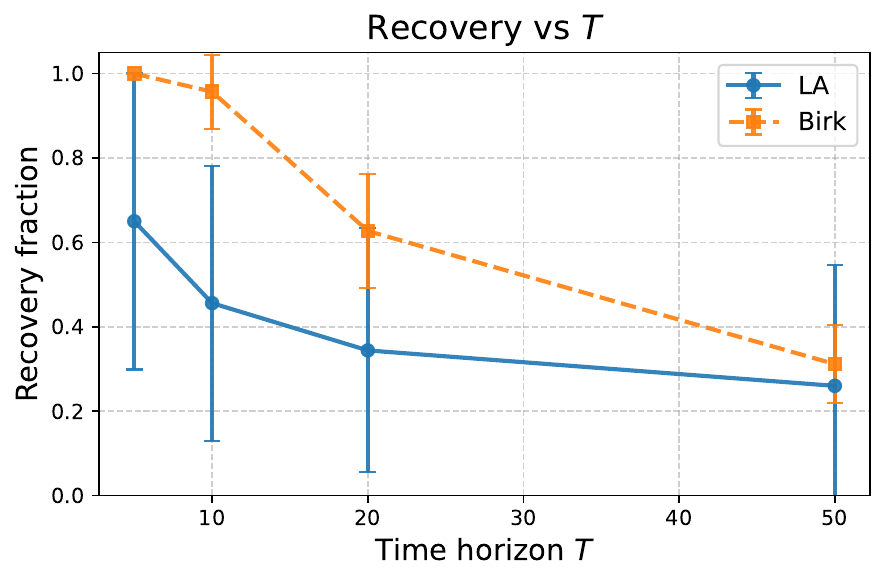}
        \caption{$\theta = 2$}
    \end{subfigure}
    \begin{subfigure}[b]{0.45\textwidth}
        \centering
        \includegraphics[width=0.95\linewidth,height=0.3\textheight,keepaspectratio]{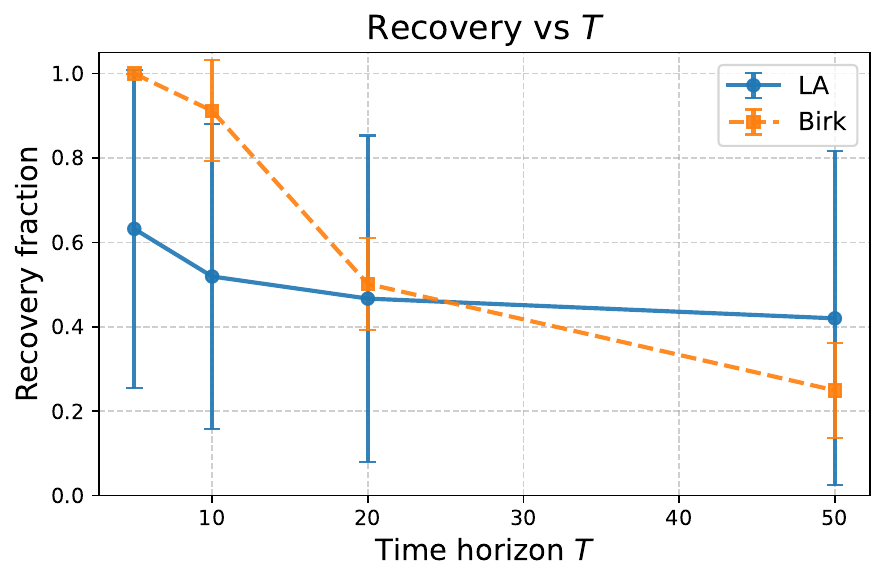}
        \caption{$\theta = 2.5$}
    \end{subfigure}
    \begin{subfigure}[b]{0.45\textwidth}
        \centering
        \includegraphics[width=0.95\linewidth,height=0.3\textheight,keepaspectratio]{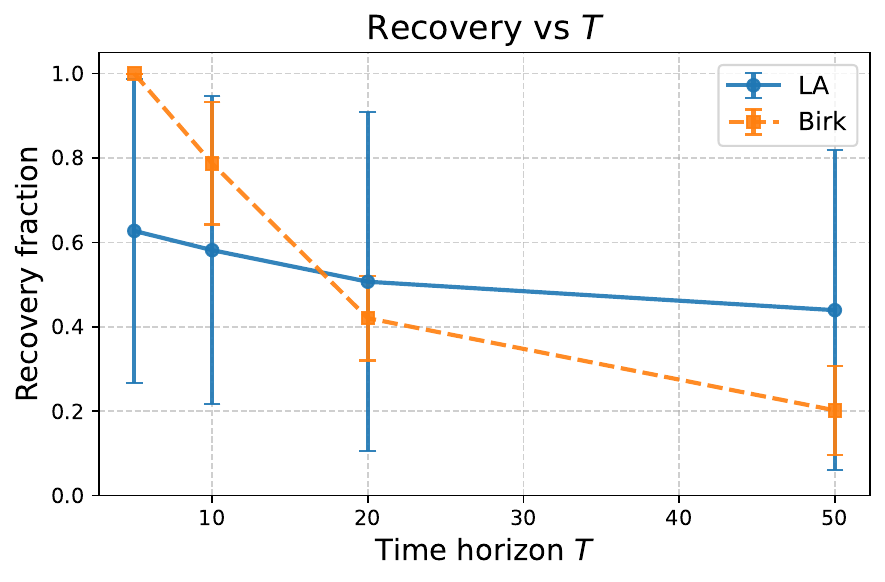}
        \caption{$\theta = 3$}
    \end{subfigure}
    \caption{Comparison of recovery performance across different values of $\theta$. Here $d=25$ and the recovery fraction is the average over $30$ samples. The error bars reflect one standard deviation above and below the average. }
    \label{fig:recovery_vs_T_unstable}
\end{figure}
\section{Conclusion and open questions}\label{sec:conclusion}
%
%\HT{Need to add stuff here. Probably we should not reveal the connection with shuffled linear regression as thats something we will look at next :)}
In this paper, we studied the problem of matching correlated VAR time series, extending the point-cloud matching framework of \cite{KuniskyWeed} to a temporal setting. We introduced a model in which a base VAR process $(x_t)_{t\in[T]}$ is perturbed by a $\sigma$-scaled independent copy and then permuted by an unknown $\pi^*$, yielding $(\hx_t)_{t\in[T]}$.

We derived the MLE for recovering $\pi^*$, from the observation of $\left((x_t)_{t\in[T]},(\hx_{t\in[T]})\right)$, and showed that it leads to a quadratic assignment problem, which is NP-hard in general. To obtain tractable alternatives, we theoretically analyzed the linear assignment estimator and established conditions---expressed as thresholds on $\sigma$---under which perfect or partial recovery is guaranteed. We also developed an alternating minimization based framework for solving the MLE, where the latent permutation is iteratively estimated by solving a suitable convex relaxation of the set of permutation matrices, thus enabling efficient first-order algorithms. Finally, we evaluated both the linear assignment estimator and the MLE relaxations on synthetic datasets, demonstrating their practical performance.

There are several promising directions for future work.

\begin{itemize}
    \item \textbf{Information-theoretic limits.} A natural open question is to determine the fundamental limits for recovering $\pi^*$. While related bounds are known for point-cloud matching \cite{schwengber2024geometricplantedmatchingsgaussian}, it remains unclear whether those techniques extend to settings with temporal correlations.  

    \item \textbf{Extending the analysis.} It would be interesting to extend our results to the regime $\rho(A^*) \le 1$, or even $\rho(A^*) > 1$, where the spectral radius $\rho(\cdot)$ characterizes the stability of linear dynamical systems. In the system identification literature for learning VAR models from a single trajectory, these two regimes have been studied recently for the setting $\rho(A^*) \le 1$ \cite{Simchowitz18a,Sarkar19} where non-asymptotic error bounds for recovering $A^*$ were obtained. These were also studied for the case $\rho(A^*) > 1$ in \cite{Sarkar19} where the results hold for ``regular'' matrices\footnote{Matrices for which the geometric multiplicity of eigenvalues lying outside the unit circle, is one.} $A^*$. 

    \item \textbf{Alternative problem formulations.} Our analysis focused on permutations of time indices, i.e., column permutations of the $d\times T$ matrix $X'$ (whose columns are the elements of $(x'_t)_{t\in[T]}$, defined in \eqref{eq:def_Xprime}). Another meaningful variant permutes the rows of $X'$, effectively shuffling the coordinates of the time series. This has potential applications in problems such as dynamic time warping and time-series alignment, and may require different analytical tools.
\end{itemize}

%
%\clearpage
%\newpage
\bibliographystyle{plain}
\bibliography{refs}

%\clearpage 

\appendix 
\section{Proof of Lemma \ref{lem:MLE_A_pi_fixed_sigma_known}}\label{app:proof_MLE_A_pi_fixed_sigma_known}
The solution to \eqref{eq:MLE_A_fixed_pi_known_sigma} can be found in closed form. Indeed, note that
\begin{align*}
    \|X-AXS\|^2_F&=\|\operatorname{vec}(X)-\left((XS)^\top\otimes I_d\right)\operatorname{vec}(A)\|^2_2, \\
    \|\hX\Pi-A\hX \Pi S-(X-AXS)\|^2_F&=\|\operatorname{vec}(\hX\Pi-X)-\left((\hX \Pi S)^\top\otimes I_d\right)\operatorname{vec}(A)+\left((XS)^\top\otimes I_d\right)\operatorname{vec}(A)\|^2_2
\end{align*}
Define, 
\begin{align*}
 v&:=\operatorname{vec}(X),\\
 V&:=(XS)^\top\otimes I_d,\\
 u&:=\operatorname{vec}(\hX\Pi-X),\\
 U&:=(\hX \Pi S)^\top\otimes I_d+(XS)^\top\otimes I_d,\\
 a&:=\operatorname{vec}(A).
\end{align*}
So \eqref{eq:MLE_A_fixed_pi_known_sigma} can be rewritten as 
\begin{equation*}
    \operatorname{vec}\left(\est{A}_{\operatorname{MLE}}(\Pi)\right)=\argmin{a\in \matR^{d^2}}\left\{\|v-Va\|^2_2+\frac1{\bar{\sigma}^2}\|u-Ua\|^2_2\right\}.
\end{equation*}
Define $g(a):=\|v-Va\|^2_2+\frac1{\bar{\sigma}^2}\|u-Ua\|^2_2$. Then, 
\begin{align*}
    \nabla g(a)&=2V^\top Va - 2 V^\top v + \frac1{\bar{\sigma}^2}\left(2U^\top U a-2 U^\top u\right)\\
    &=2\left((V^\top V + \frac1{\bar{\sigma}^2}U^\top U)a-(V^\top v+\frac1{\bar{\sigma}^2}U^\top u)\right).
\end{align*}
Hence, the condition $\nabla g(a)=0$ is equivalent to 
\begin{equation*}
    \left(V^\top V + \frac1{\bar{\sigma}^2}U^\top U\right)a=V^\top v+\frac1{\bar{\sigma}^2} U^\top u,
\end{equation*}
Now, 
\begin{equation*}
V^\top V+\frac1{\bar{\sigma}^2}U^\top U=\left((XS)(XS)^\top+\frac1{\bar{\sigma}^2} (\hX\Pi S-XS)(\hX\Pi S-XS)^\top\right)\otimes I_d.  
\end{equation*}
Also,
\begin{align*}
    V^\top v&=\left((XS)\otimes I_d \right)\operatorname{vec}(X)\\
    U^\top u&=\left((\hX\Pi S-XS)\otimes I_d \right)\operatorname{vec}(\hX\Pi-X).
\end{align*}
This translates to
\begin{align*}
    \left(V^\top V+\frac1{\bar{\sigma}^2}U^\top U\right)\operatorname{vec}(A)=A\left((XS)(XS)^\top+\frac1{\bar{\sigma}^2} (\hX\Pi S-XS)(\hX\Pi S-XS)^\top\right),
\end{align*}
and 
\begin{align*}
    V^\top v+ \frac1{\bar{\sigma}^2}U^\top u = X(XS)^\top +\frac1{\bar{\sigma}^2}\left(\hX\Pi-X\right)\left(\hX\Pi S-XS\right)^\top.
\end{align*}
So, if $\est{A}$ satisfies \eqref{eq:MLE_A_fixed_pi_known_sigma}, then 
\begin{equation*}
    \est{A}\left((XS)(XS)^\top+\frac1{\bar{\sigma}^2} (\hX\Pi S-XS)(\hX\Pi S-XS)^\top\right)=X(XS)^\top +\frac1{\bar{\sigma}^2}\left(\hX\Pi-X\right)\left(\hX\Pi S-XS\right)^\top.
\end{equation*}
From this, the result follows. 
%\EA{I was checking the result in the notes and we have the $2\bar{\sigma}^2$ factor, in the lemma there is just $\bar{\sigma}^2$.} \HT{Yes I think we shouldn't have a $2$ factor, as seen from the proof of Lemma \ref{lem:MLE_sigma_known}. So I edited this in the proof above.}
%
\section{Proof of Lemma \ref{lem:Gaussian_2cycles}}\label{app:proof_Gaussian_2cycles}
To ease notation, we use $A$ instead of $A^*$. Note that, for $a,b\in \matN$, with $a>b$, we have  
\begin{align*}
    \tilde{x}_a&=A^{a-1}\txi_1+A^{a-2}\txi_2+\ldots+A^{a-b}\txi_b+A^{a-b-1}\txi_{b+1}+\ldots+\txi_a, \\
    \tilde{x}_b&=A^{b-1}\txi_1+A^{b-2}\txi_2+\ldots+\txi_b,
\end{align*}
from which we obtain 
\begin{align*}
A(\tx_{a-1}-\tx_{b-1})=&A^{b-1}\left(A^{a-b}-I_d \right)\txi_1+A^{b-2}\left(A^{a-b}-I_d \right)\txi_2+\ldots+ A\left(A^{a-b}-I_d\right)\txi_{b-1}\\
&+\left(A^{a-b}\txi_b+\ldots+A\txi_{a-1}\right).
\end{align*}
With some algebra, we get 
\begin{align*}
  \ty_{ab} := A(\tx_{a-1}-\tx_{b-1})+\txi_a-\txi_b&=\sum^{b-1}_{i=0}A^i\left(A^{a-b}-I_d\right)\txi_{b-i}+\sum^{a-b}_{i=1}A^{a-b-i}\txi_{b+i}.
\end{align*}
Recalling the definition of $y_{ab}, \ty_{ab}$, it is clear that 
\begin{align*}
    \langle \ty_{ab},y_{ab}\rangle=\underbrace{\left\langle y_{ab},\sum^{b-1}_{i=0}A^i\left(A^{a-b}-I_d\right)\txi_{b-i}\right\rangle}_{=:\zeta_1}+\underbrace{\left\langle y_{ab},\sum^{a-b}_{i=1}A^{a-b-i}\txi_{b+i}\right\rangle}_{=:\zeta_2}.
\end{align*}
Conditioned on $\xi_1,\dots,\xi_a$, note that $\zeta_1$ and $\zeta_2$ are (independent) Gaussians (being linear combination of jointly Gaussian variables) of the form $\zeta_1\sim \calN(0,\tilde{\sigma}^2_1)$, $\zeta_2\sim \calN(0,\tilde{\sigma}^2_2)$, with
\begin{align*}
    \tilde{\sigma}^2_1&=y^\top_{ab}\left(\sum^{b-1}_{i=0}A^i\left(A^{a-b}-I_d\right) \rev{(A^{a-b}-I_d)^\top} {A^i}^\top\right)y_{ab}\\
     \tilde{\sigma}^2_2&=y^\top_{ab}\left(\sum^{a-b}_{i=1}A^{a-b-i}{A^{a-b-i}}^\top\right)y_{ab}.
\end{align*}
Furthermore, it holds 
\begin{align*}
    \tilde{\sigma}^2_1 \ \leq \ \|y_{ab}\|^2_2\|A^{a-b}-I_d\|^2_2\sum^{b-1}_{i=0}\|A^i\|^2_2
    \ &\leq \ \|y_{ab}\|^2_2\|A^{a-b}-I_d\|^2_2\sum^{b-1}_{i=0}\|A\|^{2i}_2 \\
    &\leq \frac{\|y_{ab}\|^2_2\|A^{a-b}-I_d\|^2_2}{1-\|A\|^2_2},
\end{align*}
and also
\begin{align*}
    \tilde{\sigma}^2_2 \ \leq \ \|y_{ab}\|^2_2 \norm{\sum^{a-b}_{i=1}A^{a-b-i}{A^{a-b-i}}^\top}_2
    \ &\leq \ \|y_{ab}\|^2_2\sum^{a-b}_{i=1}\|A\|^{2(a-b-i)}_2 \\
    &= \|y_{ab}\|^2_2\frac{1}{1-\|A\|^2_2}.
\end{align*}
From this, we deduce that if $\norm{A}_2 < 1$, then 
\begin{align*}
    \sigma^2(\tilde{\sigma}^2_1+\tilde{\sigma}^2_2) \ \leq \ \frac{\sigma^2\|y_{ab}\|^2_2}{1-\|A\|^2_2}\left(1+\underbrace{\|A^{a-b}-I_d \|^2_2}_{\leq 4}\right)
    \ \leq \ \frac{5\sigma^2\|y_{ab}\|^2_2}{1-\|A\|^2_2}.
\end{align*}
This completes the proof. 
%
%

%\end{proof}
%
% \section{Proof of Proposition \ref{prop:det_bound}}\label{app:proof_det_bound}
% %
%  By lemmas \ref{lem:eigs_bounds} and \ref{lem:bound_eigs_PPtop}, and recalling the definition of $\gamma(A^*)$, we have for $1 \leq j \leq td$, 
%     %
%     \begin{equation*}
%         \lambda_j(L)\in \left[\lambda_j(\tilde{L}_{C_t})(1-\gamma(A^*)),\lambda_j(\tilde{L}_{C_t})(1+\gamma(A^*))\right].
%     \end{equation*}
%     %
%     Recalling the definition of $\theta(A^*,\sigma)$, we have 
%     \begin{align*}
%         \operatorname{det}\left(\theta(A^*,\sigma)L+I_{i_1d}\right)^{-1/2}&\leq \prod^{td}_{j=1}\left(\theta(A^*,\sigma)(1-\gamma(A^*))\lambda_j(\tilde{L}_{C_t})+1\right)^{-1/2}\\
%         &=\prod^{t}_{k=1}\Big(\theta(A^*,\sigma)(1-\gamma(A^*))\lambda_k(L_{C_t})+1\Big)^{-d/2}\\
%         &=\exp\left(-\frac{d}{2} \sum_{k=1}^{t} \log\Big(1 + \theta(A^*, \sigma)(1 - \gamma(A^*))\lambda_k(L_{C_t})\Big)\right),
%     \end{align*}
%     where the penultimate equality follows from the fact that $\tilde{L}_{C_t} = L_{C_t} \otimes I_d$ has the same eigenvalues as $L_{C_t}$, each repeated $d$ times.
% %
%
%
%
%
\section{Auxiliary lemmas for proof of Proposition \ref{prop:det_bound}}
\subsection{Proof of Lemma \ref{lem:superadd_geo_means}}\label{app:proof_sup_geo_means}

    We have 
    \begin{align*}
        \left(\frac{\prod^n_{k=1}a_k}{\prod^n_{k=1}(a_k+b_k)}\right)^{1/n}+\left(\frac{\prod^n_{k=1}b_k}{\prod^n_{k=1}(a_k+b_k)}\right)^{1/n}&=\left(\prod^n_{k=1}\frac{a_k}{a_k+b_k}\right)^{1/n}+\left(\prod^n_{k=1}\frac{b_k}{a_k+b_k}\right)^{1/n}\\
        &\leq \frac1n\sum^n_{k=1}\left(\frac{a_k}{a_k+b_k} \right) + \frac1n\sum^n_{k=1} \left(\frac{b_k}{a_k+b_k} \right) \\
        &=1,
    \end{align*}
    where we used the AM-GM inequality in the second line. Multiplying both sides by $\left(\prod^n_{k=1}(a_k+b_k)\right)^{1/n}$ gives the result.
\subsection{Proof of Lemma \ref{lem:pseudodet_gen_PLPt} }\label{app:proof_pseudodet_gen_PLPt}
    Consider the eigenvalue decomposition of $Z$, 
    \[
    Z=U\begin{bmatrix}
        \Lambda & 0\\
        0 & 0
    \end{bmatrix}U^\top,
    \]
    where $\Lambda\in \matR^{p\times p}$ is a diagonal matrix with positive entries. Let $U_p\in \matR^{q\times p}$ be the matrix whose columns are the first $p$ columns of $U$ (i.e., the columns are the eigenvectors associated with non-zero eigenvalues). Then, 
    \[
    Z=U_p\Lambda U_p,
    \]
    and 
    \[
    WZW^\top=(WU_p)\Lambda(WU_p)^\top.
    \]
    This implies that, 
    \begin{align*}
        {\det}^*(WZW^\top)&={\det}^*\left((WU_p)\Lambda^{\frac12}\Lambda^{\frac12}(WU_p)^\top\right)\\
        &\substack{(1)\\=}{\det}^*\left(\Lambda^{\frac12}(WU_p)^\top(WU_p)\Lambda^{\frac12}\right)\\
        &\substack{(2)\\=}\det\left(\Lambda^{\frac12}(WU_p)^\top(WU_p)\Lambda^{\frac12}\right)\\
        &=\det(\Lambda)\det\left((WU_p)^\top(WU_p)\right)\\
        &={\det}^*(L)\det\left((WU_p)^\top(WU_p)\right).
    \end{align*}
    In $(1)$ we used that ${\det}^*(AB)={\det}^*(BA)$, which holds because $AB$ and $BA$ have the same nonzero eigenvalues, counting multiplicities, for any matrices $A,B$ such that the products $AB$, $BA$ are well defined. To obtain $(2)$, we used that the $p\times p$ matrix $\Lambda^{\frac12}(WU_p)^\top(WU_p)\Lambda^{\frac12}$ has full rank $p$, which holds because $W^\top W$ has rank $p$ (i.e., it is full rank) by assumption. In the last two lines we used well-known properties of the determinant and the fact that ${\det}^*(L)=\det(\Lambda)$, by definition. 
\subsection{Proof of Lemma \ref{lem:bound_PUtPU}}\label{app:proof_bound_PUtPU}
    To prove this lemma, we use a generalized version of Ostrowski's inequality \cite[Thm.3.2]{HIGHAM1998261} to rectangular matrices. By that result, we get
    \begin{equation*}
        \lambda_k\left(U^\top_{(t-1)d}P^\top PU_{(t-1)d}\right)=\eta_k\mu_k,\text{ for }k\in[(t-1)d],
    \end{equation*}
    where 
    \[\lambda_{td-k+1}\left(P^\top P\right)\leq \mu_{(t-1)d-k+1}\leq \lambda_{(t-1)d-k+1}\left(P^\top P\right),\]
    and 
    \[\lambda_{(t-1)d}\left(U^\top_{(t-1)d}U_{(t-1)d}\right)\leq \eta_k\leq \lambda_1\left(U^\top_{(t-1)d}U_{(t-1)d}\right).\]
    But, $\lambda_1\left(U^\top_{(t-1)d}U_{(t-1)d}\right)=\lambda_{(t-1)d}\left(U^\top_{(t-1)d}U_{(t-1)d}\right)=1$, which implies that $\eta_k=1$, for all $k\in[(t-1)d]$. From this, we deduce that 
\begin{align*}
     \det\left((PU_{(t-1)d})^\top(PU_{(t-1)d})\right)&=\prod^{(t-1)d}_{k=1}\lambda_k\left((PU_{(t-1)d})^\top(PU_{(t-1)d})\right)\\
     &=\prod^{(t-1)d}_{k=1}\mu_k\\
     &\geq  \prod^{(t-1)d}_{k=1}\lambda_{td-k+1}\left(P^\top P\right) \\
     &= \prod^{td}_{k=d+1}\lambda_k\left(P^\top P\right)
     %&= \prod^{td}_{k=d}\lambda_k\left(P^\top P\right)\\
     =\frac{\det\left(P^\top P\right)}{\prod^d_{k=1}\lambda_k\left(P^\top P\right)}.
\end{align*}
\subsection{Proof of Lemma \ref{lem:det_PtP}}\label{app:proof_lem:det_PtP}
    We will use $A$ instead of $A^*$ to ease notation. Note that, by definition, 
    \begin{equation*}
        P=\begin{bmatrix}
            A^{i_1-1}&A^{i_2-1} &\cdots&\cdots &A^{i_t-1} \\
            \vdots & \vdots&\cdots &\cdots& \vdots\\
            \vdots &\vdots &\cdots &\cdots &I_d \\
            \vdots & A &\cdots &\cdots& 0\\
             A^{i_1-i_2} & I_d&\cdots&\cdots &\vdots \\
            A^{i_1-i_2-1} & 0& \cdots&\cdots&\vdots \\
            \vdots & \vdots&\cdots&\cdots & \vdots\\
            A &\vdots &\cdots&\cdots & \vdots\\
            I_d &0 & \cdots&\cdots&0
        \end{bmatrix}.
    \end{equation*}
    We now define the following matrices in $\matR^{i_1d \times d} $
    \[
    V_1 =
\begin{bmatrix}
  0 \\
  \vdots \\
  \vdots \\
  0 \\
  \vdots\\
  0 \\
  A^{i_1-i_2-1} \\
  \vdots \\
  A\\
  I_d
\end{bmatrix}
\quad
\begin{matrix}
  \left. \begin{matrix}
\, \\ \, \\ \, \\ \,\\\,\\\,\\\,
\end{matrix} \right\} \substack{i_2 \, d\times d \\ \text{zero blocks}}

\\
\\
\\
\\
\\
\\
% \\
% \\
%   \left.\begin{matrix} \, \\ \, \\\, \\ \, \end{matrix}\right\} \text{zeros}
\end{matrix},
V_2 =
\begin{bmatrix}
  0\\
  \vdots\\
  0 \\
  A^{i_2-i_3-1} \\
  \vdots \\
  A \\
  I_d \\
  0 \\
  \vdots\\
  0
\end{bmatrix}
\quad
\begin{matrix}
  \left. \begin{matrix}
\, \\ \, \\ \, 
\end{matrix} \right\} \substack{i_3 \, d\times d \\ \text{zero blocks}}
\\
\\
\\
\\
\\
\\
 \left.\begin{matrix}  \, \\\, \\ \, \end{matrix}\right\} \substack{(i_1-i_2) \, d\times d \\ \text{zero blocks}}
\end{matrix},\,
\cdots, \,
V_t =
\begin{bmatrix}
  A^{i_t-1} \\
  \vdots\\
  I_d \\
  0\\
  \vdots \\
  \vdots \\
  0\\
  \vdots\\
  0
\end{bmatrix}
\quad
% \begin{matrix}
%   \left. \begin{matrix}
% \, \\ \, \\ \, 
% \end{matrix} \right\} \text{ $i_3$ zeros}
\begin{matrix}
\\
\\
\\
 \left.\begin{matrix}\, \\  \, \\\, \\ \,\\ \, \\\,\\\, \end{matrix}\right\} \substack{(i_1-i_t) \, d\times d \\ \text{zero blocks}}
\end{matrix}.
\]
In other words, for $l\in [i_1d]$, and $j\in[d]$, %defining $i=\lceil l/d\rceil$,
\begin{align*}
(V_1)_{lj}&=\begin{cases}
    0\,\quad \text{  if } \,1\leq l\leq i_2d\\
    A^{i_1-\lceil l/d\rceil}_{lj}\, \quad \text{   if  }\, i_2d+1\leq l \leq i_1d,
\end{cases},
\\
(V_k)_{lj}&=\begin{cases}
    0\,\quad \text{  if } \,1\leq l\leq i_{k+1}d\\
    A^{i_k-\lceil l/d\rceil}_{lj}\,\quad \text{  if } \,i_{k+1}d+1\leq l\leq i_{k}d\\
    0\, \quad  \text{  if } \,i_{k}d+1\leq l\leq i_{1}d
\end{cases},\quad \text{for } k\in\{2,\ldots,t-1\},
\\
(V_t)_{lj}&=\begin{cases}
    A^{i_t-\lceil l/d\rceil}_{lj}\,\quad \text{  if } \,1\leq i\leq i_t\\
    0\, \quad \text{   if  }\, i_t+1\leq i \leq i_1
\end{cases}.
\end{align*}
From this definition, it is clear that the matrices $V_1,\ldots,V_t$ are pairwise orthogonal, i.e., $\langle V_k,V_{k'}~\rangle_F~=~0$, for $k,k'\in [t]$. We define the following shifting matrices, $\{S_{k,j-1}\}_{2\leq k\leq t, 2\leq j\leq t}$ in $\matR^{t\times t}$, 
\begin{equation*}
    (S_{k,j-1})_{ll'}=\begin{cases}
        1\,\quad \text{ for }\, l=k,l'=j-1\\
        0\, \quad \text{otherwise}
    \end{cases}.
\end{equation*}
To see the effect of post-multiplying by this matrices, consider  $U=\begin{bmatrix}
    u_1 & u_2 &\ldots&u_{t}
\end{bmatrix}\in \matR^{i_1\times t}$. Then
\begin{equation*}
    US_{k,j-1}=\begin{bmatrix}
        0&\cdots&0&\underbrace{u_k}_{j-1 \text{ position}}&0&\cdots&0
    \end{bmatrix}.
\end{equation*}
In words, post-multiplying $U$ by $S_{k,j-1}$ forms a new matrix with the same dimensions of $U$, with its $(j-1)$-th columns equal to the $k$-th column of $U$, and the rest of the columns are zero. With this, we express $P$ as follows, 
\begin{align*}
P=&\begin{bmatrix}
    V_1&V_2&\cdots &V_t
\end{bmatrix}+\sum^t_{k=2}(I_{i_1}\otimes A^{i_1-i_k})\begin{bmatrix}
    V_k&0&\cdots&0
\end{bmatrix}+\sum^t_{k=3}(I_{i_1}\otimes A^{i_2-i_k})\begin{bmatrix}
    0&V_k&0&\cdots&0
\end{bmatrix}\\
&+\ldots+ (I_{i_1}\otimes A^{i_{t-1}-i_t})\begin{bmatrix}
    0&\cdots&0&V_{t}&0
\end{bmatrix}\\
=&V+\sum^t_{j=2}\sum^t_{k=j}
(I_{i_1}\otimes A^{i_{j-1}-i_k})V(S_{k,j-1}\otimes I_d),
\end{align*}
where $V:=\begin{bmatrix}
    V_1&V_2&\cdots&V_t
\end{bmatrix}$. On the other hand, for all $k,j\in [t]$,
\[
(I_{i_1}\otimes A^{i_{j-1}-i_k})V(S_{k,j-1}\otimes I_d)=V(S_{k,j-1}\otimes I_d)(I_{t}\otimes A^{i_{j-1}-i_k}),
\]
which implies, 
\begin{align*}
P&=V+\sum^t_{j=2}\sum^t_{k=j}V(S_{k,j-1}\otimes I_d)(I_{t}\otimes A^{i_{j-1}-i_k})\\
&=V+\sum^t_{j=2}\sum^t_{k=j}V\left(S_{k,j-1}\otimes A^{i_{j-1}-i_k}\right)\\
&=V\left(I_{td}+\sum^t_{j=2}\sum^t_{k=j}S_{k,j-1}\otimes A^{i_{j-1}-i_k}\right).
\end{align*}
With this, we have 
\begin{equation*}
    P^\top P = \left(I_{td}+\sum^t_{j=2}\sum^t_{k=j}S_{k,j-1}\otimes A^{i_{j-1}-i_k}\right)^\top V^\top V\left(I_{td}+\sum^t_{j=2}\sum^t_{k=j}S_{k,j-1}\otimes A^{i_{j-1}-i_k}\right).
\end{equation*}
Given that $I_{td}+\sum^t_{j=2}\sum^t_{k=j}
S_{k,j-1}\otimes A^{i_{j-1}-i_k}$ and $V^\top V$ are square matrices, we have 
\begin{equation*}
    \det\left(P^\top P\right)=\det(V^\top V)\det\left(I_{td}+\sum^t_{j=2}\sum^t_{k=j}
S_{k,j-1}\otimes A^{i_{j-1}-i_k}\right)^2.
\end{equation*}
The term $\sum^t_{j=2}\sum^t_{k=j}
S_{k,j-1}\otimes A^{i_{j-1}-i_k}$ only contains matrices of the form $S_{k,j-1}$, with $k\geq j$, which are all strictly lower triangular. Then it is easy to see that $I_{td}+\sum^t_{j=2}\sum^t_{k=j}
S_{k,j-1}\otimes A^{i_{j-1}-i_k}$ is a block-lower triangular matrix with $I_d$'s on its main block-diagonal. On the other hand, the determinant of a block-lower triangular matrix equals the determinant of the block-diagonal matrix formed by its diagonal blocks (see \cite[Section 0.9.4]{horn_johnson_2012}). This implies that
\[
\det\left(I_{td}+\sum^t_{j=2}\sum^t_{k=j}
S_{k,j-1}\otimes A^{i_{j-1}-i_k}\right)=\det\left(I_{td}\right)=1.
\]
On other hand, by the orthogonality of the blocks $V_1,\ldots,V_t$ that form $V$, it is easy to see that $V^\top V$ is block diagonal, of the form, 
%
% it is diagonal, with $
% (V^\top V)_{kk}=\|v_k\|^2_2
% $, for $k\in\{1,\ldots,t\}$. Hence, we have
\[
V^\top V=\operatorname{blkdiag}\left(V^\top_1V_1,V^\top_2V_2,\ldots,V^\top_tV_t\right)=\begin{bmatrix}
    V^\top_1V_1&0&\cdots&\cdots&0\\
    0&V^\top_2V_2&0&\cdots&0\\
    \vdots&\vdots&\ddots&\cdots&0\\
    0&\ldots&\ldots&\ldots&V^\top_tV_t
\end{bmatrix}.
\]
Given that the determinant of a block diagonal matrix is the product of the determinant of its blocks, we obtain %\HT{I updated the part below, please check expression for $V_k^\top V_k$}
\begin{align*}
\det\left(P^\top P\right) = \det(V^\top V) =  \prod^t_{k=1}\det(V^\top_kV_k),
%&= \left(\prod^{t-1}_{k=1}\prod^d_{j=1}\,\sum^{i_k-i_{k-%1}-1}_{l=0}{\lambda_{j}
%(A)}^{2l}\right)\left(\prod^d_{j=1}\sum^{i_t-1}_{l=0}{\lambda_{j}
%(A)}^{2l}\right) .
\end{align*}
%
%In the last line we used that 
%
where we note that 
\[
V^\top_k V_k = 
\begin{cases}
\sum^{i_k-i_{k-1}-1}_{l=0}(A^{l})^\top A^l, \text{ for } k\in[t-1] \\
\sum^{i_t-1}_{l=0} (A^{l})^\top A^l,\text{ for } k=t.
\end{cases}
\]
In particular, since $V^\top_k V_k \succeq I_d$ for each $k$, this implies $\det\left(P^\top P\right) \geq 1$.
\subsection{Proof of Lemma \ref{lem:gerschgorin_PtP}}\label{app:proof_gerschgorin_PtP}
We will again use $A$ instead of $A^*$ to ease notation. Notice that $P^\top P$ has the following block structure
\renewcommand{\arraystretch}{1.5} % Increases row spacing
\begin{equation*}
    P^\top P = \begin{bmatrix}
    P^\top_{i_1}P_{i_1}& P^\top_{i_1}P_{i_2}&\cdots & P^\top_{i_1}P_{i_t}\\
    P^\top_{i_2}P_{i_1}& P^\top_{i_2}P_{i_2}&\cdots & P^\top_{i_2}P_{i_t}\\
    \vdots& \ddots &\cdots&\vdots \\
    P^\top_{i_t}P_{i_1}& P^\top_{i_t}P_{i_2}&\cdots & P^\top_{i_t}P_{i_t}
    \end{bmatrix}.
\end{equation*}
\renewcommand{\arraystretch}{1} % Increases row spacing
In order to bound the eigenvalues of $P^\top P$ we use \cite[Theorem 1.13.1]{Tretter2008} which generalizes the Gershgorin disk theorem to the block matrix case. In particular, it says that 
\[
\operatorname{spec}(P^\top P)\in \cup^t_{k=1}\calG_k, 
\]
where 
\[
\calG_k:=\operatorname{spec}(P^\top_{i_k}P_{i_k})\cup \left\{\lambda\notin \operatorname{spec}(P^\top_{i_k}P_{i_k}):dist\big(\lambda,\operatorname{spec}(P^\top_{i_k}P_{i_k})\big)\leq \sum^t_{\substack{k'=1\\k'\neq k}}\|P^\top_{i_k}P_{i_k}\|_2 \right\}.
\]
Using the above result, we have the following estimates (we use $A$ instead of $A^*$ to ease notation).
\begin{enumerate}[label=(\alph*)]
    \item For $k\in[t]$, we have 
    \begin{align*}
    P^\top_{i_k}P_{i_k} &= \sum^{i_k-1}_{j=0}(A^j)^\top A^j=I_d + \sum^{i_k-1}_{j=1}(A^j)^\top A^j, \text{ and }\\
    \norm{\sum^{i_k-1}_{j=1}(A^j)^\top A^j}_2 &\leq \sum^{i_k-1}_{j=1}\|A\|^{2j}_2
    =\frac{\|A\|^2_2}{1-\|A\|^2_2}=:\delta(A).
    \end{align*}
    Then, $\operatorname{spec}(P^\top_{i_k}P_{i_k})\in [1-\delta(A),1+\delta(A)]$.

    \item For $k',k\in [t]$, with $k\neq k'$, we distinguish the following two cases.

    \begin{enumerate}[label=(\alph{enumi}.\arabic*)]
    \item For $k'$ such that $i_{k'}<i_k$, we have 
    \begin{align*}
        P^\top_{i_k}P_{i_{k'}}=A^{i_k-i_{k'}}\big(I+A^2+A^4+\ldots+A^{2(i_{k'}-1)}\big),
    \end{align*}
    which implies that 
    \begin{align*}
    \|P^\top_{i_k}P_{i_{k'}}\|_2 \leq \|A\|^{i_k-i_{k'}}_2\left(\frac{1-\|A\|^{2i_{k'}}_2}{1-\|A\|^2_2}\right)
    \leq \frac{\|A\|^{i_k-i_{k'}}_2}{1-\|A\|^2_2}.
    \end{align*}
    From above, we obtain 
    \[
    \sum_{k':i_{k'}<i_k}\|P^\top_{i_k}P_{i_{k'}}\|_2\leq \frac{\|A\|_2}{(1-\|A\|_2)(1-\|A\|^2_2)}.
    \]
    \item For $k'$ such that $i_{k'}>i_k$, we have 
    \[
    P^\top_{i_k}P_{i_{k'}}=A^{i_{k'}-i_{k}}\big(I+A^2+A^4+\ldots+A^{2(i_{k}-1)}\big),
    \]
    from which we obtain
    \begin{align*}
    \|P^\top_{i_k}P_{i_{k'}}\|_2&\leq \|A\|^{i_{k'}-i_{k}}_2\left(\frac{1-\|A\|^{2i_{k}}_2}{1-\|A\|^2_2}\right)
    \leq \frac{\|A\|^{i_k-i_{k'}}_2}{1-\|A\|^2_2}.
    \end{align*}
    Hence, 
     \[
    \sum_{k':i_{k'}>i_k}\|P^\top_{i_k}P_{i_{k'}}\|_2\leq \frac{\|A\|_2}{(1-\|A\|_2)(1-\|A\|^2_2)}.
    \]
    \end{enumerate}
    Combining the estimates in (b.1) and (b.2), we obtain
    \[
    \sum_{k':i_{k'}\neq i_k}\|P^\top_{i_k}P_{i_{k'}}\|_2\leq \frac{2\|A\|_2}{(1-\|A\|_2)(1-\|A\|^2_2)}:=\kappa(A).
    \]
\end{enumerate}
From the above calculations, we see that (let $\lambda_j^{(k)}$ lie in $\operatorname{spec}(P^\top_{i_k}P_{i_{k}})$)
\begin{align*}
\calG_k&\subseteq [1-\delta(A), 1+\delta(A)]\cup \left\{\cup^d_{j=1}\{\lambda\neq \lambda^{(k)}_j\in \operatorname{spec}(P^\top_{i_k}P_{i_{k}}):|\lambda-\lambda^{(k)}_j|\leq \kappa(A)\}\right\}\\
&\subseteq [1-\delta(A)-\kappa(A),1+\delta(A)+\kappa(A)]\\
&=\left[1-\frac{\|A\|^2_2}{1-\|A\|^2_2}-\frac{2\|A\|_2}{(1-\|A\|_2)(1-\|A\|^2_2)},1+\frac{\|A\|^2_2}{1-\|A\|^2_2}+\frac{2\|A\|_2}{(1-\|A\|_2)(1-\|A\|^2_2)}\right]
\end{align*}
Now, 
\begin{align*}
    1+\frac{\|A\|_2^2}{1-\|A\|^2_2}+\frac{2\|A\|_2}{(1-\|A\|_2)(1-\|A\|^2_2)}&\leq \frac{1}{1-\|A\|^2_2}+\frac{2\|A\|_2}{(1-\|A\|_2)(1-\|A\|^2_2)}\\
    &\leq \frac{1-\|A\|_2^2}{(1-\|A\|^2_2)(1-\|A\|_2)}\\
    &\leq \frac{1}{(1-\|A\|_2)^2}.
\end{align*}
Then,
\[
\lambda_1\left(P^\top P\right)\leq \frac1{(1-\|A\|_2)^2}.
\]
\section{Proof of Theorem \ref{thm:main_upper_bound}}\label{app:proof_main_upper_bound}
The proof of Theorem \ref{thm:main_upper_bound} follows directly from the next three lemmas. 
\begin{lemma}\label{lem:error_bound_genA_regime1}
    Let $s_0:=2^{1/d}$, and assume $A^*$ satisfies $\|A^*\|_2<1$. Suppose that
    \[
    \sigma^2\leq \frac{(1-\|A^*\|_2)^5}{4(s^{\omega(1)}_0T^{4/d}-1)}.
    \]
    Then we have $\expec[|\calE|]\to 0$, when $T\to\infty$. In particular, $|\calE|=0$ with high probability.
\end{lemma}
\begin{proof}
    From Proposition \ref{prop:det_bound}, we get
    \begin{align*}
        \log\det\left(\frac{(1-\|A^*\|_2)^3}{4\sigma^2}L+I_{i_1d}\right)^{-\frac12}&\leq -\frac d2(t-1)\log\left(\frac{(1-\|A^*\|_2)^5}{4\sigma^2} +1\right)\\
        &\leq -\frac d 4 t\log\left(\frac{(1-\|A^*\|_2)^5}{4\sigma^2} +1\right),\nonumber
    \end{align*}
    where the last inequality follows from the fact that $2(t-1)\geq t$, for $t\geq2$. From this, we deduce, 
    \begin{align}
        t\log T+\log{\det\left(\frac{(1-\|A^*\|_2)^5}{4\sigma^2}L+I_{i_1d}\right)^{-\frac 12}}&\leq (t\log T)\left(1-\frac d{4\log T}\log{\left(\frac{(1-\|A^*\|_2)^5}{4\sigma^2}+1\right)}\right)\nonumber\\
        &\leq (t\log T) \left(1-\frac{d}{4\log T}\log{s^{\omega(1)}_0T^{\frac4d}}\right)\nonumber\\
        &\leq (t\log T)\left(-\omega(1)\frac{\log 2}{\log T}\right)\nonumber\\\label{eq:bound_for_expect}
        &\leq -t\omega(1).
    \end{align}
    From \eqref{eq:first_moment}, \eqref{eq:augmenting_proba_bound} and the fact that number of $t$-cycles in $[T]$ is bounded by $T^t/t$, we obtain, 
    \begin{align*}
        \expec[|\calE|]&\leq \sum^T_{t=2}\exp{\left(t\log T+ \log{\det\left(\frac{(1-\|A^*\|_2)^5}{4\sigma^2}L+I_{i_1d}\right)^{-\frac 12}}\right)}\\
        &\leq \sum^T_{t=2}(e^{-\omega(1)})^t=o(1),
    \end{align*}
    where in the second line we used \eqref{eq:bound_for_expect}.
    The high probability statement follows directly from Markov's inequality. 
    %\EA{This is how they wrote the conclusion, one can also be a little bit more explicit here and/or the lemma statement...}
\end{proof}
When $\|A^*\|_2=0$, our result recovers up-to-constants the results in \cite{KuniskyWeed}. 
%, proving an alternative proof of their result.
%
Similarly, we have an analogous result to \cite[Lemma 3.4]{KuniskyWeed}, which leads to a constant error upper bound. 
\begin{lemma}\label{lem:error_bound_genA_regime2}
    Let $s_0=2^{1/d}$, and assume $A^*$ satisfies  $\|A^*\|_2<1$. Suppose that  
    \[
    \sigma^2\leq \frac{(1-\|A^*\|_2)^5}{4(s^{O(1)}_0T^{4/d}-1)}.
    \]
    Then, $\expec[|\calE|]=O(1)$. In particular, for any function $f(T)=w(1)$, we have $|\calE|\leq f(T)$ with high probability. 
\end{lemma}
\begin{proof}
    The proof mimics the argument in the proof of Lemma \ref{lem:error_bound_genA_regime1}, but exchanging $\omega(\cdot)$ by $\calO(\cdot)$. This shows that $\expec[|\calE|]=\calO(1)$, and the rest of the argument follows from an application of Markov's inequality. 
    %\EA{Same as before, one can be a bit more explicit...}
\end{proof}
We now establish an analogue of \cite[Lemma~3.5]{KuniskyWeed}. 
Our proof proceeds in the same manner as the previous lemmas. 
Interestingly, in \cite{KuniskyWeed} the argument for this lemma differs slightly from the others, 
whereas in our case no such modifications are required.
\begin{lemma}\label{lem:error_bound_genA_regime3}
    Let $s_0=2^{1/d}$, and assume $A^*$ satisfies $\|A^*\|_2<1$. Suppose that   
    \[
    \sigma^2\leq \frac{(1-\|A^*\|_2)^5}{4(s^{\omega(1)}_0T^{2/d}-1)}.
    \]
    Then, 
    \begin{align*}
        \expec[|\calE|]=\calO\left(\left(\frac{(1-\|A^*\|_2)^5}{4\sigma^2}+1\right)^{-\frac d2}T^2\right).
    \end{align*}
\end{lemma}
\begin{proof}
    Recall that by Proposition \ref{prop:det_bound}, we have 
    \begin{align*}
        t\log{T}+\log{\det\left(\frac{(1-\|A^*\|_2)^5}{4\sigma^2}L+I_{i_1d}\right)^{-\frac 12}}&\leq t\log{T}-\frac d2(t-1)\log\left(\frac{(1-\|A^*\|_2)^5}{4\sigma^2}+1\right)\\
        &=2\log{T}+(t-2)\log{T} -\frac d 2 (t-2)\log\left(\frac{(1-\|A^*\|_2)^5}{4\sigma^2}+1\right)\\
        &\quad-\frac d 2\log\left(\frac{(1-\|A^*\|_2)^5}{4\sigma^2}+1\right)\\
        &\leq 2\log{T}-\frac d 2\log\left(\frac{(1-\|A^*\|_2)^5}{4\sigma^2}+1\right)-\omega(1)(t-2).
    \end{align*}
    Then, 
    \begin{align*}
        \expec[|\calE|]&\leq \sum^T_{t=2}\exp{\left(t\log{T}+\log{\det\left(\frac{(1-\|A^*\|_2)^5}{4\sigma^2}L+I_{i_1d}\right)^{-\frac 12}}\right)}\\
        &\leq \sum^T_{t=2}\left(\frac{(1-\|A^*\|_2)^5}{4\sigma^2}+1\right)^{-\frac d2}T^2(e^{-\omega(1)})^{t-2}\\
        &=\calO\left(\left(\frac{(1-\|A^*\|_2)^5}{4\sigma^2}+1\right)^{-\frac d2}T^2\right).
    \end{align*}
\end{proof}

\section{Additional experiments}\label{sec:additional_experiments}
In this section, we include additional experiments that complement the experiments in the main paper. In Section \ref{sec:est_A_first} we show the performance of the estimation method described in Remark \ref{rem:estimate_A_first}. %5In Section \ref{sec:oracle_init} we implement an oracle strategy that in a very particular regime exhibits an interesting behavior. 
Finally, in Section \ref{sec:other_algs_Birkhoff} we implement other algorithms for the Birkhoff polytope relaxation and show that they all have similar performance. 

\subsection{Estimate \texorpdfstring{$A^*$}{A*} first}\label{sec:est_A_first}
One natural approach for estimating $\Pi$ consists in the two-step strategy described in Remark~\ref{rem:estimate_A_first}. Here, first $A^*$ is estimated using only the time series $(x_t)_{t\in[T]}$, for which the right order is known, via least-squares (the MLE for estimating $A^*$ under the \textbf{VAR} model). Call this estimator $\est{A}_{\operatorname{MLE}}$. Then we solve \eqref{eq:relaxed_general} replacing $A$ with $\est{A}_{\operatorname{MLE}}$. We show the results in the regimes $d=5,T=50$ and $d=50, T=5$, in Figure \ref{fig:estimate_A_first} using the Mirror Descent (MD) algorithm for the simplex relaxation. We observe that in the regime $d=5,\, T=50$, the performance under the true and the estimated $A^*$ are very similar, likely because the time series is long enough so that $\est{A}_{\operatorname{MLE}}$ is already close to $A^*$. On the other hand, for $d=50,\,T=5$, the performance of \texttt{RelaxMLE-Round} with $\est{A}_{\operatorname{MLE}}$ degrades as the noise increases. In the considered regime, the LA estimator and \texttt{RelaxMLE-Round} with $A^*$ both achieve perfect recovery. 
\begin{figure}[ht]
    \centering
    \begin{subfigure}[b]{0.45\textwidth}
        \centering
        \includegraphics[width=0.95\linewidth,height=0.3\textheight,keepaspectratio]{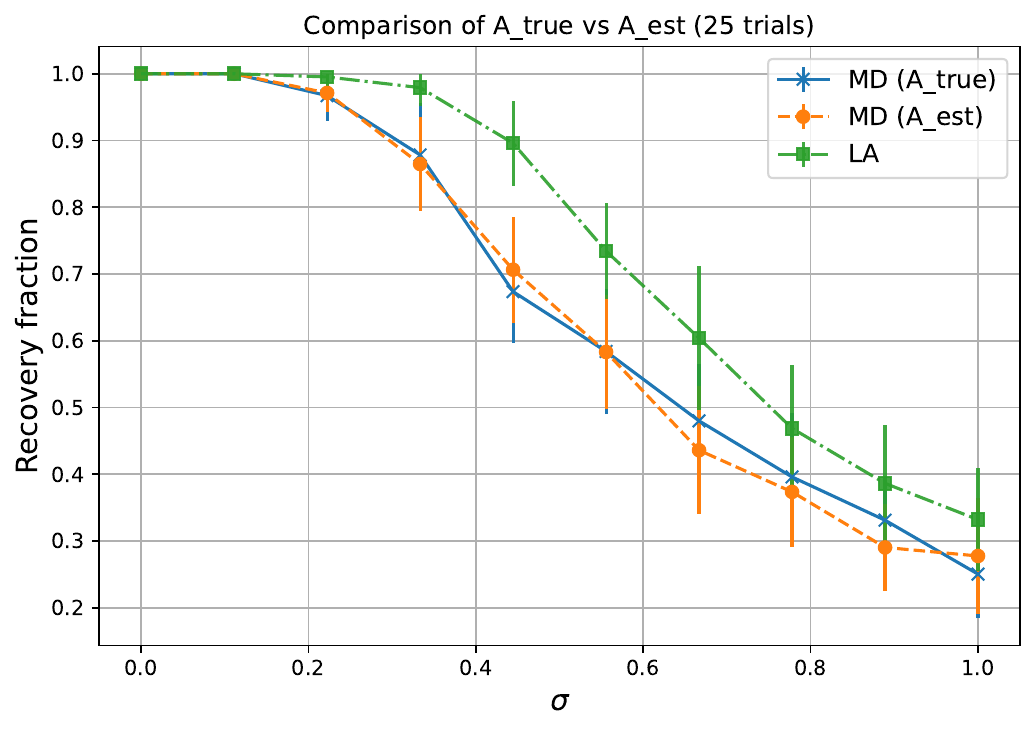}
        \caption{$d=5$, $T=50$}
        \label{fig:sub1_app_A_first}
    \end{subfigure}
    \begin{subfigure}[b]{0.45\textwidth}
        \centering
        \includegraphics[width=0.95\linewidth,height=0.3\textheight,keepaspectratio]{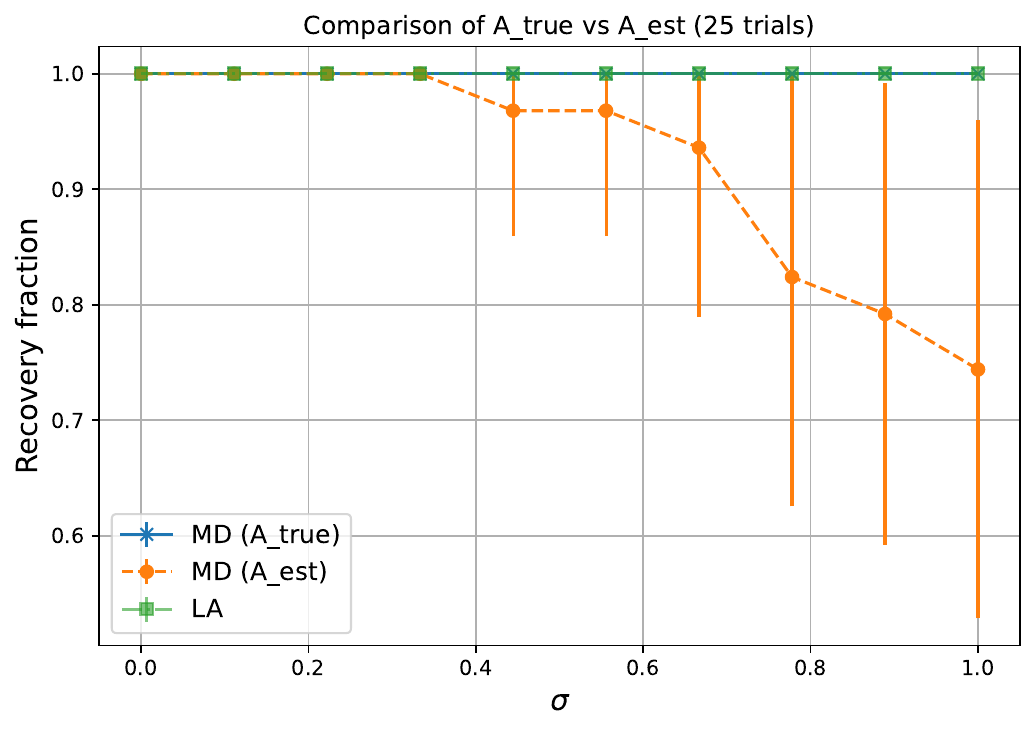}
        \caption{$d=50$, $T=5$}
        \label{fig:sub2_app_A_first}
    \end{subfigure}
 \caption{Recovery fraction vs. $\sigma$ for the MD algorithm, which solves the simplex relaxation. Here $\theta=0.5$ and the average is computed over $30$ Monte Carlo runs. In Fig. \ref{fig:sub2_app_A_first} both the MD estimator with access to $A^*$ and the LA estimator achieve perfect recovery.}
    \label{fig:estimate_A_first}
\end{figure}

\subsection{Other algorithms for the Birkhoff relaxation}\label{sec:other_algs_Birkhoff}
We consider different optimization schemes for the Birkhoff relaxation, including a Frank–Wolfe method, an ADMM-based approach, and a quasi-Newton variant. 
The Frank–Wolfe algorithm \cite{FW} performs iterative linear minimization over the Birkhoff polytope using the gradient of the objective and a line search step, offering a projection-free alternative particularly suited for large-scale problems. 
The ADMM algorithm \cite{2011DistributedOA} enforces the Birkhoff constraints via alternating updates of primal and dual variables, with efficient projections implemented through Dykstra’s algorithm. 
Finally, the quasi-Newton method \cite{quasi_newton} applies an L-BFGS step on the flattened permutation matrix followed by projection onto the Birkhoff polytope, providing a curvature-aware but more computationally demanding alternative. 
In Figure~\ref{fig:diff_algos_Birkhoff}, we compare all three methods for $d=5$, $T=50$, and $\theta=0.5$. We run algorithm \ref{alg:relaxMLE_round}, assuming a known $A^*$, which is sampled at random according to \eqref{eq:parametric_A*}.
In this setting, all algorithms achieve similar recovery performance, and this pattern remains consistent across the different parameter combinations we tested. %\HT{Again, please clarify the experimental setup -- you initialize permutation randomly and then run Alg. \ref{alg:TS_matching_alternating} for some number of runs? ($K = 5$?)}

\begin{figure}
    \centering
    \includegraphics[width=0.5\linewidth]{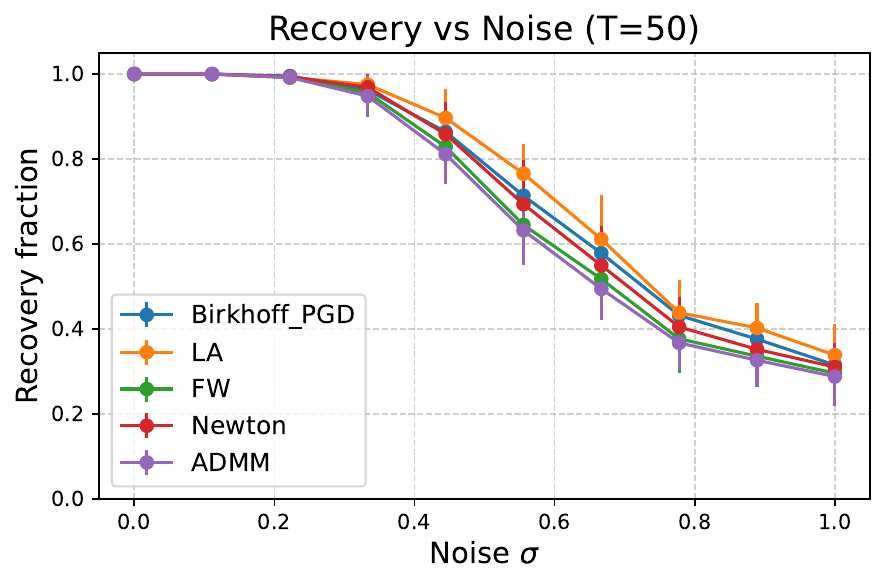}
    \caption{Comparison between different algorithms to solve \eqref{eq:relaxed_general} on the Birkhoff polytope. We evaluate them in the setting of Algorithm \ref{alg:relaxMLE_round}. Here we consider the setting of known $d=5,T=50$, $\theta=0.5$ and different levels of noise $\sigma$.}
    \label{fig:diff_algos_Birkhoff}
\end{figure}
\end{document}